\documentclass[11pt,a4paper]{article}

\usepackage[T1]{fontenc}
\usepackage[utf8]{inputenc}
\usepackage{lmodern}
\usepackage{microtype}
\usepackage{geometry}
\usepackage{amsmath,amssymb,amsthm,mathtools,mathrsfs}
\usepackage{bm}
\usepackage{enumitem}
\usepackage{xcolor}
\usepackage{tikz}
\usetikzlibrary{arrows.meta,positioning,calc,decorations.pathreplacing,cd,patterns,fit}
\usepackage{tikz-cd}
\usepackage{booktabs}
\usepackage{array}
\usepackage{longtable}
\usepackage{aliascnt}
\usepackage{hyperref}
\usepackage[capitalise,noabbrev]{cleveref}
\usepackage{csquotes}

\hypersetup{
  colorlinks=true,
  linkcolor=blue!50!black,
  citecolor=green!40!black,
  urlcolor=blue!60!black,
  pdftitle={AKSZ-BV-BFV on Joyce Generalized Corners},
  pdfauthor={Cristian Anghel},
  pdfkeywords={AKSZ, BV-BFV, generalized corners, b-tangent Lie algebroid,
               monoidal resolutions, regularized integration, BF theory}
}

\newtheorem{theorem}{Theorem}[section]

\newaliascnt{proposition}{theorem}
\newtheorem{proposition}[proposition]{Proposition}
\aliascntresetthe{proposition}

\newaliascnt{lemma}{theorem}
\newtheorem{lemma}[lemma]{Lemma}
\aliascntresetthe{lemma}

\newaliascnt{corollary}{theorem}
\newtheorem{corollary}[corollary]{Corollary}
\aliascntresetthe{corollary}

\newaliascnt{conjecture}{theorem}
\newtheorem{conjecture}[conjecture]{Conjecture}
\aliascntresetthe{conjecture}

\theoremstyle{definition}
\newaliascnt{definition}{theorem}
\newtheorem{definition}[definition]{Definition}
\aliascntresetthe{definition}

\newaliascnt{example}{theorem}
\newtheorem{example}[example]{Example}
\aliascntresetthe{example}

\newaliascnt{assumption}{theorem}
\newtheorem{assumption}[assumption]{Assumption}
\aliascntresetthe{assumption}

\theoremstyle{remark}
\newaliascnt{remark}{theorem}
\newtheorem{remark}[remark]{Remark}
\aliascntresetthe{remark}

\crefname{theorem}{Theorem}{Theorems}
\crefname{proposition}{Proposition}{Propositions}
\crefname{lemma}{Lemma}{Lemmas}
\crefname{corollary}{Corollary}{Corollaries}
\crefname{conjecture}{Conjecture}{Conjectures}
\Crefname{conjecture}{Conjecture}{Conjectures}
\crefname{definition}{Definition}{Definitions}
\crefname{example}{Example}{Examples}
\crefname{assumption}{Assumption}{Assumptions}
\crefname{remark}{Remark}{Remarks}

\newcommand{\F}{\mathcal F}
\newcommand{\Y}{\mathcal Y}
\newcommand{\ev}{\operatorname{ev}}
\newcommand{\Map}{\operatorname{Map}}
\newcommand{\Tot}{\operatorname{Tot}}
\newcommand{\codim}{\operatorname{codim}}
\newcommand{\supp}{\operatorname{supp}}
\newcommand{\Cone}{\operatorname{Cone}}
\newcommand{\Hom}{\operatorname{Hom}}
\newcommand{\id}{\mathrm{id}}
\newcommand{\dd}{\mathrm d}
\newcommand{\bT}{{}^{b}T}
\newcommand{\bOmega}{\Omega_b}

\newcommand{\Rge}{\mathbb R_{\geq 0}}
\newcommand{\Z}{\mathbb Z}
\newcommand{\N}{\mathbb N}
\newcommand{\cO}{\mathcal O}
\newcommand{\cP}{\mathcal P}
\newcommand{\cR}{\mathcal R}
\newcommand{\scale}{\zeta}
\newcommand{\cC}{\mathcal C}
\newcommand{\kfield}{\Bbbk}
\newcommand{\op}{\mathrm{op}}

\newcommand{\Tr}{\operatorname{Tr}}
\newcommand{\Reg}{\operatorname{Reg}}
\newcommand{\ad}{\operatorname{ad}}
\newcommand{\sd}{\operatorname{sd}}
\newcommand{\sgn}{\operatorname{sgn}}
\newcommand{\gp}{\mathrm{gp}}
\newcommand{\db}{\mathrm d_b}
\newcommand{\logOmega}{\Omega_{b,\log}}
\newcommand{\R}{\mathbb R}
\newcommand{\cA}{\mathcal A}
\newcommand{\cT}{\mathcal T}
\newcommand{\cV}{\mathcal V}
\newcommand{\trg}{\operatorname{tr}_{\mathfrak g}}

\title{\textbf{Relative Logarithmic AKSZ Descent on Joyce Generalized Corners}\\
\large Monoidal resolutions, conditional intrinsic $b$-transgression,
and conifold strictification}
\author{Cristian Anghel}
\date{August 17, 2026}

\begin{document}
\maketitle

\begin{abstract}
We study resolution-independent \emph{relative logarithmic} source complexes
for AKSZ--BV--BFV theory on face-oriented Joyce manifolds with generalized
corners.  Under the geometric hypotheses \textup{(G1)}--\textup{(G3)}, the
Dupont--Panzer--Pym total relative logarithmic complexes of admitted smooth
monoidal resolutions represent the common object
$j_!\Omega_{X^\circ}^\bullet$ and carry the same compactly supported derived
trace.  This comparison is deliberately unfiltered: it identifies total
relative cocycle classes through the common interior, but not individual
resolved faces, separate BFV descendants, nonlinear mapping spaces, or
arbitrary absolute regularized integrals.

For the intrinsic $b$-source $\bT[1]X$, we formulate the additional Stokes and
transfer data needed to recover strict facewise structures.  Assuming a
multiplicative regularized Stokes trace system on Joyce's face category gives
the presymplectic intrinsic BV--BFV identity and incidence descent.  The
finite datum \textup{(G4)} is a linear subdivision/aggregation \emph{transfer
criterion}; the separate datum \textup{(G5)} is a cyclic field-level
\emph{transfer criterion} for classical abelian BF theory.  Neither is a
general existence theorem, and a general strict multiplicative intrinsic
trace is not constructed here.

We prove the interval contractions needed for codimension-two resolution
collars and analyze the positive real conifold through the common star
refinement of its two diagonal resolutions.  On an explicit finite
product--Whitney logarithmic coefficient class, the exceptional square
satisfies the full \textup{(G4)} criterion after normal-face totalization and
the transferred differential is exactly the signed intrinsic incidence
differential.  Its finite algebraic BF dual gives a cyclic
logarithmic--cellular shadow, but not the full continuum intrinsic
$b$-de Rham datum \textup{(G5)}.  General nonlinear continuum pushforward,
existence of the intrinsic multiplicative trace, and loop-level logarithmic
graph integrals remain open.
\end{abstract}



\tableofcontents

\section{Introduction}

The AKSZ construction originates in the graded-symplectic sigma-model
formalism of Alexandrov--Kontsevich--Schwarz--Zaboronsky~\cite{AKSZ}.  Its
BV--BFV extension to manifolds with boundary is developed by
Cattaneo--Mnev--Reshetikhin~\cite{CMRClassical,CMRQuantum}, and nonlinear
boundary/globalization questions have been treated in particular for split
AKSZ models~\cite{CMWNonlinearAKSZ}.  A facewise formulation for ordinary
corners is developed in the companion paper~\cite{AnghelPartI}.

Several frameworks converge on the geometry used here.  Kottke--Melrose
introduced generalized blow-up of ordinary corners and the monoidal complexes
that control it~\cite{KottkeMelrose}, and Kottke extended the blow-up to
$g$-corners~\cite{KottkeGCorner}.  Independently, Gillam--Molcho developed
positive log differentiable spaces~\cite{GillamMolcho}, in which manifolds
with $g$-corners appear as the log smooth objects; this is the setting in
which Dupont--Panzer--Pym construct their regularized
integrals~\cite{DPP}.  Joyce's intrinsic differential-geometric
account~\cite{JoyceGCorner}, which builds on his ordinary-corner
theory~\cite{JoyceCorners}, is the one followed below, and we cite the other
two frameworks where their results are used.  On the target side we use the
graded symplectic $QP$-manifolds of Roytenberg and
\v{S}evera~\cite{Roytenberg,Severa}.

Joyce generalized corners replace the local orthant by a monoidal model
\[
 X_P=\Hom_{\mathrm{Mon}}(P,\mathbb R_{\ge0}),
\]
where the weakly toric monoid $P$ need not be free.  Two new difficulties are
therefore simultaneous: ordinary ordered-face combinatorics no longer
describes the intrinsic higher corners, and the natural $b$-forms contain
logarithmic normal factors.  Joyce's $b$-tangent Lie algebroid supplies the
differential graded source, while the regularized integration theory of
Dupont--Panzer--Pym supplies the model for the Stokes datum after passing to an
ordinary-corner resolution~\cite{JoyceGCorner,DPP}.

Three source models must be distinguished.  First, ordinary AKSZ on
$T[1]X_\cR$ gives a resolution-level facewise descent system and a useful
linear baseline.  Second, the \emph{resolved $b$/logarithmic route} uses
$\bT[1]X_\cR$ together with the Dupont--Panzer--Pym relative logarithmic
complex.  Third, the \emph{intrinsic $b$-route} uses $\bT[1]X$ and Joyce's
intrinsic face category.  Its facewise AKSZ transgression requires an
additional multiplicative regularized Stokes trace system; the paper treats
this as structured input and does not prove its existence for arbitrary
Joyce $g$-corners.  The unconditional comparison between resolved logarithmic
models is therefore the unfiltered relative comparison through the common
interior.  A stricter resolved--intrinsic comparison is obtained only after
the separate finite transfer criteria \textup{(G4)} and \textup{(G5)} are
supplied.  The ordinary model maps to the
resolved $b$-model by the anchor but is not claimed to be equivalent to it.
Regularizations on distinct resolutions are chosen independently and are
compared only at the level of total relative classes through the common
interior; a character of the vertex monoid records intrinsic vertex values but
does not determine the full resolved DPP datum.
The positive real conifold
\[
 X_\square=\{x\in\mathbb R_{\ge0}^4:x_1x_2=x_3x_4\}
\]
is the basic non-ordinary test case throughout.

\subsection*{Hypotheses and scope}

\begin{definition}[Geometric hypotheses and transfer data]\label{def:admissible-package}
The following hypotheses and additional transfer data are used separately.
\begin{enumerate}[label=\textup{(G\arabic*)},leftmargin=2.8em]
\item $X$ is a compact face-oriented Joyce manifold with generalized corners and
embedded boundary faces, and its associated monoidal complex is finite; for
$G\prec F$ the chosen orientations determine incidence signs $[F:G]$.
\item A smooth global monoidal resolution $\beta_\cR:X_\cR\to X$ is fixed together with a
Dupont--Panzer--Pym regularization on its complete ordered boundary diagram.
\item The admitted comparison class has smooth common refinements whose comparison maps,
on compact subsets, factor into finite stellar towers.
\item For linear strictification, the resolved total face complex carries the finite,
centre-compatible, incidence-local subdivision/aggregation contraction of
\cref{def:face-compatible-stellar,sec:global-linear-strict}.
\item For the abelian BF/BV--BFV interpretation, the paired contravariant
$A$-cochain and covariant BF-dual $B$-chain systems carry the
deformation-retract and cyclicity data of
\cref{def:bf-cyclic-enhancement}.
\end{enumerate}
\end{definition}

The derived relative construction uses \textup{(G1)}--\textup{(G3)}.
Clauses \textup{(G4)} and \textup{(G5)} are of a different logical nature:
they are explicit chain-level \emph{transfer criteria}, not consequences of
\textup{(G1)}--\textup{(G3)} and not general existence theorems.  In
particular, \textup{(G5)} includes the cyclicity, adjointness and exact
matching of the transferred degree-one terms with prescribed signed face
arrows that are needed for the strict abelian BF interpretation.  The
conifold construction below verifies \textup{(G4)} on a concrete finite
logarithmic class and supplies a finite algebraic-dual BF shadow, but not the
full continuum datum \textup{(G5)}.  Continuum quantum Gaussian statements
additionally use \cref{ass:gaussian-determinant}.  General nonlinear and
loop-level continuum pushforwards are not claimed.

\begin{remark}[Choice dependence]\label{rem:choice-dependence}
Face orientations, smooth refinements, DPP normal sections, finite coefficient
subcomplexes, cellular bases, initial contractions and polarizations are part
of the input.  The common-interior object and its trace are independent of the
chosen resolution in the derived sense proved below.  By contrast, the
displayed chain-level retracts and finite BF actions depend on the stated
choices unless an explicit homotopy, canonical transformation, or
determinant-line comparison is supplied.  Throughout, \emph{canonical} for a
derived roof means canonical after fixing the indicated common-interior
identification; it does not erase these chain-level choices.
\end{remark}

All statements involving infinite-dimensional AKSZ mapping spaces are understood under
the same standing formal hypothesis as in the companion paper: evaluation, restriction,
contraction by the lifted source and target vector fields, and fiber integration of local
forms are defined and satisfy graded Cartan calculus and Stokes' formula.  Equivalently,
the identities may be read on a regular finite-dimensional approximation on which these
operations are defined.  The derived source-complex and cellular statements do not use
this additional formal mapping-space hypothesis.

Unless a topology is explicitly introduced, every tensor product in this
paper is algebraic.  In particular, tensoring a continuum complex with a
finite cellular complex means finite direct sums of copies of the continuum
complex.  We do not identify a full space of smooth or logarithmic forms on a
product with an algebraic tensor product of the two factor spaces.

Finally, a bibliographic convention.  The numbering of definitions,
propositions, corollaries and displayed equations quoted from
\cite{DPP} is that of the published version,
\emph{J. Éc. polytech. Math.} \textbf{13} (2026), which may differ from the
numbering of the preprint arXiv:2312.17720; likewise, quotations from
\cite{CMRClassical,CMRCellular} follow the published versions.

\subsection*{Main results}

The results have three logically distinct levels.  The first is
\emph{unconditional within the geometric package}: relative logarithmic
complexes and traces are compared through the common interior under
\textup{(G1)}--\textup{(G3)}.  The second is \emph{conditional}: intrinsic
multiplicative transgression and strict linear/cyclic face models require the
additional trace system and the transfer criteria \textup{(G4)}--\textup{(G5)}.
The third is \emph{constructive}: the positive real conifold verifies the
full \textup{(G4)} criterion on an explicit finite product--Whitney class,
while only a finite logarithmic--cellular BF shadow of \textup{(G5)} is
constructed.  The theorem statements below keep these levels separate.

\begin{theorem}[Resolution-level and relative logarithmic models]\label{thm:main-resolution-model}
Under \textup{(G1)}--\textup{(G3)}, admitted smooth resolutions have chain-homotopy
equivalent oriented subdivision complexes.  More generally, a facewise linear coefficient
system equipped with filtered refinement-transfer data and locally acyclic reduced
refinement fibres has quasi-isomorphic total descent complexes on any two admitted
resolutions.  For abelian BF coefficients, elementary radial blow-ups have
resolution-independent relative totalizations through the common-interior
roof; a filtered comparison still requires explicit refinement-transfer data.
For a finite-dimensional unimodular Lie algebra and chosen blow-up-adapted
finite ball-complex models, the Cattaneo--Mnev--Reshetikhin cellular nonabelian BF
actions are likewise invariant in residual BV
cohomology under the admitted elementary resolution moves.
\end{theorem}

\begin{proof}
The oriented subdivision statement is
\cref{cor:comb-invariance}.  The filtered linear comparison and its passage
through a common refinement are
\cref{thm:derived-comparison,cor:common-derived}.  The relative abelian BF
comparison is \cref{thm:bf-facewise,cor:bf-compositions}, and the cellular
nonabelian statement is \cref{thm:nonabelian-cellular}.
\end{proof}

\begin{theorem}[Intrinsic trace formalism and conditional transfer]\label{thm:main}
Let $(\Y,\omega_\Y=\delta\alpha_\Y,\Theta)$ be an exact $QP$-target of degree
$\dim X-1$.
\begin{enumerate}[label=\textup{(\roman*)},leftmargin=2.5em]
\item Assume, in addition to the geometric hypotheses, that
$\mathsf{Face}(X)$ carries a multiplicative regularized Stokes trace system in
the sense of \cref{def:trace-system}.  Then the intrinsic $b$-source carries
the presymplectic AKSZ--BV--BFV identity on every face, with higher descent
encoded by the incidence differential.  If the transgressed two-forms are
weakly nondegenerate on the chosen field class, the usual Hamiltonian
BV--BFV structures follow.  No general existence theorem for this
multiplicative trace system is asserted.
\item Under \textup{(G1)}--\textup{(G3)}, the DPP total complexes of smooth
resolutions represent the common relative de Rham object
$j_!\Omega_{X^\circ}^\bullet$ and its compactly supported trace.  This is the
unconditional resolved comparison in the paper.  It globalizes to a derived
sheaf satisfying \v{C}ech descent but deliberately forgets the face
filtration and does not identify individual face components or descendants.
\item If the linear transfer criterion \textup{(G4)} is supplied, basic
logarithmic trace/descent coefficients in its specified class strictify to an
intrinsic incidence complex.  If, separately, the cyclic field-level transfer
criterion \textup{(G5)} is supplied---including the required adjointness,
pathwise face data and exact matching of transferred degree-one terms with the
prescribed signed face arrows---the paired $A$-cochain/$B$-chain systems carry
a strict classical intrinsic abelian BF model with incidence-level
BV--BFV/Stokes descendants and a bivariant maximally extended face system.
These are transfer theorems conditional on the stated data; no general
existence of \textup{(G4)} or \textup{(G5)} is claimed, and no identification
of the two retracts is implied without an extra comparison datum.
\end{enumerate}
\end{theorem}

\begin{proof}
Part \textup{(i)} is
\cref{thm:bvbfv,cor:higher}.  Part \textup{(ii)} is proved by the
common-interior roof \cref{thm:refinement-invariance}, its sheaf version
\cref{prop:sheaf-refinement,thm:sheaf-descent}, and the global trace
\cref{thm:global-trace}.  Part \textup{(iii)} is
\cref{thm:global-linear-strict,thm:cyclic-hpl-bf,cor:global-strict-bf}.
\end{proof}

\begin{theorem}[Resolved--intrinsic comparison: unconditional and conditional layers]\label{thm:C}
The comparison in this theorem is between the resolved $b$/logarithmic source
and the intrinsic $b$-source; it is not an equivalence with the ordinary AKSZ
source on the resolution.
For a smooth monoidal resolution $\beta:X_\cR\to X$, the $b$-derivative is a vector-bundle
isomorphism, so pullback identifies the intrinsic $b$-source with the basic bulk sector of
the resolved $b$-source.  On the complete face system the canonical comparison under
\textup{(G1)}--\textup{(G3)} is the unfiltered derived comparison through the common
interior; it is not a bijection of faces.  Under \textup{(G4)} the subdivision/aggregation
contraction gives a strict linear logarithmic incidence comparison after
totalization.  Under the separate
field-level datum \textup{(G5)}, the corresponding paired $A$-cochain and
BF-dual $B$-chain retracts are compatible with the classical abelian BF
pairing and give an incidence-level BV--BFV/Stokes family, contravariant in
$A$ and covariant in $B$ along all face inclusions.
\end{theorem}

\begin{proof}
The $b$-source identification is \cref{prop:betale}.  The unfiltered
facewise comparison and its trace qualification are
\cref{thm:refinement-invariance,prop:tracecompat}.  The strict linear and
cyclic BF comparisons are collected in
\cref{thm:basic-comparison,thm:global-linear-strict,thm:cyclic-hpl-bf}.
\end{proof}

The principal unconditional statement is therefore the relative logarithmic
comparison through the common interior; by construction it forgets the face
filtration.  The genuinely stratified content begins with the separate
transfer criterion \textup{(G4)}.  When that datum is available, the actual
full-differential reduced sector is $\ker p_\infty$ and is contracted by
$h_\infty$; it is not identified a priori with the span of the geometrically
exceptional faces.
The conifold makes this visible: Joyce's intrinsic corner spaces and iterated boundaries
differ, while the two smooth diagonal resolutions become comparable after a common star
refinement.  Besides the local codimension-two calculation, the product--Whitney
class below provides on the common star resolution a finite global linear
realization and an algebraic cyclic shadow of the mechanism.  Nonlinear
strictification, a type-correct conifold
\textup{(G5)} realization, and loop-level logarithmic graph
integrals are left as explicit open problems in \cref{sec:open-problems}.

\part*{I.\ Geometry of generalized corners}\addcontentsline{toc}{part}{I.\ Geometry of generalized corners}

\section{Toric monoids, Joyce generalized corners, and the
\texorpdfstring{$b$}{b}-Lie algebroid}

\subsection{Weakly toric monoids}\label{subsec:weakly-toric}

We use additive notation for monoids.  A commutative monoid $P$ is called
\emph{weakly toric} in Joyce's terminology if it is finitely generated,
integral, saturated, and torsion free.  Kottke calls the same objects toric
monoids.  For the general theory of such monoids, and for the notions of
sharpening and face used below, we refer to Ogus~\cite[Ch.~I]{Ogus}.  Its group completion $P^{\gp}$ is a lattice, and we write
\[
 V_P=P^{\gp}\otimes_\Z\R.
\]

The subgroup of units is denoted $P^\times$, and the sharpening is
\[
  P^\sharp=P/P^\times.
\]
The monoid is \emph{sharp} if $P^\times=\{0\}$.  It is \emph{smooth} if
$P^\sharp$ is freely generated, equivalently
\[
  P\cong\N^k\times\Z^\ell
\]
for some $k,\ell$.  A sharp smooth monoid is therefore $\N^k$.

Every sharp weakly toric monoid may be represented as
\[
  P=\sigma^\vee\cap\Lambda^*
\]
for a full-dimensional strictly convex rational polyhedral cone $\sigma$ in
$\Lambda_{\R}=\Lambda\otimes\R$, where $\Lambda^*=\Hom(\Lambda,\Z)$.  Faces
of $P$ correspond contravariantly to faces of $\sigma$.  Full-dimensionality of
$\sigma$ is what makes $P$ sharp, and strict convexity is what gives
$P^{\gp}=\Lambda^*$.

\begin{definition}
A submonoid $S\subseteq P$ is a \emph{face} if its complement is a prime
ideal.  We write $S\leq P$.  The face poset is denoted $\mathsf{Face}(P)$.
\end{definition}

\begin{remark}
For $P=\N^k$, the face poset is Boolean.  For a non-simplicial cone, the
number of codimension-one faces meeting the vertex may exceed $k$.  This is
the first combinatorial obstruction to transporting the ordered-cube
formalism literally from ordinary corners.
\end{remark}

\subsection{The model space \texorpdfstring{$X_P$}{X P}}

Regard $\Rge$ as a monoid under multiplication.  Joyce defines
\[
  X_P=\Hom_{\mathrm{Mon}}(P,\Rge),
\]
a space of dimension $\operatorname{rank}P^{\gp}$.  For $p\in P$, evaluation gives
a smooth non-negative function
\[
  \lambda_p:X_P\longrightarrow\Rge,
  \qquad \lambda_p(x)=x(p),
\]
and these are multiplicative,
\[
  \lambda_{p+q}=\lambda_p\lambda_q.
\]
The smooth structure is generated by these monomial functions together with
ordinary smooth functions of finitely many generators.  A face $F\leq P$
determines an inclusion $X_F\hookrightarrow X_P$ by extending a homomorphism
on $F$ by zero on $P\setminus F$.  If
\[
  P=\N^k\times\Z^{n-k},
\]
then
\[
  X_P\cong [0,\infty)^k\times\mathbb R^{n-k}.
\]
Thus ordinary corners are precisely the smooth local models.

A manifold with generalized corners is a Hausdorff, second-countable space
with an atlas of open subsets of the $X_P$ and Joyce-smooth transition maps.
The ordinary-corner theory on which this generalization is modelled is
developed in \cite{JoyceCorners}.
We impose the additional hypothesis, used by Kottke, that boundary faces are
embedded.

\begin{definition}
A \emph{resolved-admissible generalized-corner manifold} is a compact
oriented Joyce manifold with generalized corners whose boundary faces are
embedded, whose associated monoidal complex is finite, and whose intrinsic
faces carry the face-orientation datum used in \textup{(G1)}.
\end{definition}

The finiteness assumption is convenient for a compact source and for finite
totalizations.  It may later be replaced by local finiteness and support
conditions.

\subsection{The \texorpdfstring{$b$}{b}-tangent bundle}

The ordinary tangent bundle of a $g$-corner may fail to be a rank-$n$ vector
bundle.  The $b$-tangent bundle is well behaved, and it is what makes the
intrinsic route possible at all.  On a manifold with ordinary corners it is
Melrose's $b$-tangent bundle, the bundle whose sections are vector fields
tangent to all boundary faces \cite{Melrose}; Joyce's construction extends it
to arbitrary $g$-corners.

\begin{proposition}[Joyce; Argüz--Joyce~\cite{JoyceGCorner,ArguzJoyce}]\label{prop:btangent}
On $X_P$ there are canonical identifications
\[
 \bT X_P\cong X_P\times V_P^*,
 \qquad
 \bT^*X_P\cong X_P\times V_P.
\]
If $\alpha\in V_P^*=\Hom(P^{\gp},\R)$ and $v_\alpha$ is the corresponding
constant $b$-vector field, then
\[
 v_\alpha(\lambda_p)=\alpha(p)\lambda_p.
\]
The space of $b$-vector fields has a Lie bracket and an anchor to
derivations of $C^\infty(X_P)$, making $\bT X_P$ a Lie algebroid.
\end{proposition}

\begin{proof}
The local trivialization and the action on monomial functions are part of
the construction of $\bT X_P$.  The bracket is characterized by the
commutator of derivations.  On constant sections one has
$[v_\alpha,v_\beta]=0$, and the Leibniz rule determines the bracket on
arbitrary local sections.  These local structures are invariant under
changes of $g$-chart and therefore glue.
\end{proof}

\subsection{The differential graded source}

For a Lie algebroid $A\to X$, the degree-shifted space $A[1]$ is encoded by
its Chevalley--Eilenberg algebra; the identification of Lie algebroid
structures on $A$ with degree-one homological vector fields on $A[1]$ is due
to Vaintrob~\cite{Vaintrob}.  This remains meaningful when the base is a
$g$-corner.

\begin{definition}\label{def:bsource}
The differential graded locally ringed space $\bT[1]X$ is defined by
\[
 C^\infty(\bT[1]X)=\Gamma(X,\wedge^\bullet\bT^*X)=\bOmega^\bullet(X),
\]
with homological vector field $\db$ equal to the Chevalley--Eilenberg
differential of the Lie algebroid $\bT X$.
\end{definition}

\begin{proposition}\label{prop:dbsquare}
The operator $\db$ has degree $+1$, is a graded derivation, and satisfies
$\db^2=0$.  On a local monomial function,
\[
 \db\lambda_p=\lambda_p\vartheta_p,
\]
where $\vartheta_p\in\Gamma(\bT^*X_P)$ is the constant $b$-one-form
corresponding to $p\in V_P$.
\end{proposition}

\begin{proof}
The first assertion is the standard Chevalley--Eilenberg construction for a
Lie algebroid.  Pairing the second formula with $v_\alpha$ gives
\[
 \langle v_\alpha,\db\lambda_p\rangle=v_\alpha(\lambda_p)=\alpha(p)\lambda_p
 =\langle v_\alpha,\lambda_p\vartheta_p\rangle,
\]
which proves the identity.
\end{proof}

Formally, on the interior one may write
\[
 \vartheta_p=\db\log\lambda_p.
\]
Unlike $\log\lambda_p$, the form $\vartheta_p$ extends as a smooth section
of $\bT^*X_P$ across the boundary.  This is the whole point: the resolved and
intrinsic $b$/logarithmic routes
of this paper are two ways of making sense of integration against such
forms, one by pulling them back to a resolution where they become ordinary
logarithmic forms, the other by regularizing directly on $X$.

\subsection{The basic non-ordinary model}

Consider
\begin{equation}\label{eq:conifold-model}
  X_\square=
  \{(x_1,x_2,x_3,x_4)\in\Rge^4:x_1x_2=x_3x_4\}.
\end{equation}
Joyce identifies this with $X_P$ for
\begin{equation}\label{eq:conifold-monoid-presentation}
  P=\{(a,b,c)\in\N^3:c\leq a+b\}.
\end{equation}
Its normal cone is combinatorially the cone over a square.  Four boundary
hypersurfaces and four one-dimensional edges meet at the vertex.  By
contrast, an ordinary three-dimensional corner has only three boundary
hypersurfaces at a vertex.  By \cref{prop:btangent}, however, $\bT X_\square$
is still a trivial bundle, of rank three; the mismatch between four
hypersurfaces and three logarithmic directions is exactly the single
monoidal relation, and it is what \cref{sec:conifold} unfolds.

A convenient lattice realization of the cone over the square uses the rays
\begin{equation}\label{eq:square-rays}
\begin{aligned}
 v_1&=(0,0,1),& v_2&=(1,0,1),\\
 v_3&=(0,1,1),& v_4&=(1,1,1).
\end{aligned}
\end{equation}
The two diagonals of the square give two smooth triangulations:
\begin{align}
 \cR_+&=\{\Cone(v_1,v_2,v_4),\ \Cone(v_1,v_3,v_4)\},\label{eq:triang-plus}\\
 \cR_-&=\{\Cone(v_1,v_2,v_3),\ \Cone(v_2,v_3,v_4)\}.
 \label{eq:triang-minus}
\end{align}
Each maximal cone is unimodular.  The two corresponding blow-ups are
manifolds with ordinary corners, but their face decompositions differ.

\section{Monoidal complexes, refinements, and blow-up}

\subsection{The monoidal complex of a generalized-corner manifold}

For a generalized-corner manifold $X$ with embedded faces, Kottke associates a monoidal complex
\[
  \cP_X.
\]
Informally, every boundary face $F$ carries a normal monoid, and inclusions $G\subseteq F$ induce face inclusions between the corresponding monoids.  The monoidal complex records these monoids and their incidence maps.  For ordinary corners this structure, and the generalized blow-up it supports, are due to Kottke--Melrose~\cite{KottkeMelrose}; the $g$-corner case used here is \cite{KottkeGCorner}.

We use only the following structural properties.
\begin{enumerate}[label=(M\arabic*),leftmargin=2.5em]
\item $\cP_X$ is functorial for interior $b$-maps.
\item Its smooth monoids correspond to ordinary orthant models.
\item Restricting to a boundary face gives a monoidal subcomplex.
\item A refinement replaces each cone by a compatible polyhedral subdivision with the same support.
\end{enumerate}

\begin{definition}
A \emph{refinement} of a finite monoidal complex $\cP$ is a morphism
\[
  \psi:\cR\longrightarrow\cP
\]
such that, cone by cone, the monoids of $\cR$ give a finite subdivision of the corresponding rational cone without changing its support.  It is \emph{smooth} if every monoid of $\cR$ is smooth.
\end{definition}

\subsection{Kottke blow-up}

The geometric theorem that makes the resolution strategy possible is the following.

\begin{theorem}[Kottke]\label{thm:kottke}
Let $X$ be a manifold with generalized corners with embedded boundary faces, and let
\[
  \psi:\cR\longrightarrow\cP_X
\]
be a refinement.  There is a blow-up
\[
  \beta_\cR:[X;\cR]\longrightarrow X
\]
unique up to diffeomorphism, such that:
\begin{enumerate}[label=(\roman*),leftmargin=2.2em]
\item $\beta_\cR$ is a diffeomorphism on interiors;
\item $\cP_{[X;\cR]}\cong\cR$;
\item $\beta_\cR$ satisfies the universal lifting property for interior $b$-maps whose monoidal morphisms factor through $\cR$;
\item blow-up commutes with pullback by an arbitrary interior $b$-map.
\end{enumerate}
\end{theorem}

This is a summary of Theorems 3.7 and 3.10 of \cite{KottkeGCorner}.  If $\cR$ is smooth, then $[X;\cR]$ has ordinary corners.  We call it a \emph{smooth monoidal resolution} of $X$.

\begin{proposition}[Properness and the common interior]\label{prop:proper-blowdown}
Let $\cR\to\cP_X$ be a finite refinement and let
$\beta_{\cR}:[X;\cR]\to X$ be Kottke's blow-down.  Then $\beta_{\cR}$ is
proper and surjective, and
\begin{equation}\label{eq:blowdown-interior-preimage}
 \beta_{\cR}^{-1}(X^\circ)=[X;\cR]^\circ,
 \qquad
 \beta_{\cR}^{-1}(X\setminus X^\circ)
 =[X;\cR]\setminus [X;\cR]^\circ.
\end{equation}
If $\cR$ is smooth, the set on the right in the second identity is the
ordinary topological boundary of the resolved manifold.
\end{proposition}

\begin{proof}
Properness is local on the target.  In a Kottke model $X_P$, the
construction is obtained by gluing the non-negative real toric charts
$X_{R^\vee}$ over the cones $R$ of the refinement; see
\cite[Eq.~(3.4) and Prop.~3.2]{KottkeGCorner}.  Complexifying the same monoid
charts gives the toric morphism defined by the identity on the ambient lattice
and the subdivision of the cone $P^\vee$.  The toric properness criterion
states that this complex toric morphism is proper precisely when the support
of the source fan is the inverse image of the support of the target fan
\cite[Thm.~3.4.11]{CLS}.  A refinement has, by definition, the same support,
so the complex toric morphism is proper.  Its non-negative real locus is
closed and is preserved by the morphism; hence the restriction to the
non-negative real toric charts is proper as well.  Therefore the local
blow-down is proper, and properness local on the target gives the global
claim.

For the second assertion, a non-interior face of $[X;\cR]$ corresponds to a
nonzero cone of $\cR$.  Under the monoidal morphism
$\cR\to\cP_X$ this cone maps injectively to a nonzero cone of $\cP_X$, hence
its face maps into $X\setminus X^\circ$.  Together with Kottke's
diffeomorphism on interiors this gives
\eqref{eq:blowdown-interior-preimage}.  Finally, the image of a proper map
between locally compact Hausdorff spaces is closed; it contains the dense
interior $X^\circ$, so it is all of $X$.  Thus $\beta_{\cR}$ is surjective.
\end{proof}

\begin{remark}
The word resolution is used here in the differential-geometric toric sense.  The interior is unchanged, while the singular combinatorics of the boundary is replaced by a smooth fan.  No singularity in the interior of the source is introduced or removed.
\end{remark}

\subsection{Common refinements}

\begin{proposition}[Common rational refinement]\label{prop:common-refinement}
Let $\cR_1\to\cP$ and $\cR_2\to\cP$ be finite refinements of a finite
monoidal complex.  There exists a finite rational polyhedral common refinement
$\cR_{12}^{\mathrm{rat}}$ with maps
\[
  \cR_{12}^{\mathrm{rat}}\longrightarrow\cR_i\longrightarrow\cP,
  \qquad i=1,2.
\]
The construction is compatible with restriction to monoidal subcomplexes.
If the two refinements belong to an admissible comparison class satisfying
\textup{(G3)}, the additional choice of a \emph{smooth} common refinement,
with comparison maps factoring into finite stellar towers, is part of the
comparison datum specified there.
\end{proposition}

\begin{proof}
On a single rational cone $\sigma$, intersect each cone of the first
subdivision with each cone of the second.  The nonempty intersections and all
of their faces form a finite rational polyhedral subdivision of $\sigma$ that
refines both given subdivisions.  Performing this construction cone by cone
is compatible on common faces, because intersection commutes with restriction
to a face.  Since $\cP$ is finite, the resulting monoidal complex is finite.
The last sentence is exactly the smooth-comparison hypothesis \textup{(G3)};
no relative toric desingularization theorem stronger than that hypothesis is
being asserted here.
\end{proof}

By \cref{thm:kottke}, a common refinement gives a diagram
\begin{equation}\label{eq:common-resolution-diagram}
\begin{tikzcd}[column sep=large]
 & { [X;\cR_{12}] }
   \arrow[dl,"\gamma_1"']
   \arrow[dr,"\gamma_2"] & \\
 { [X;\cR_1] }
   \arrow[dr,"\beta_1"'] &&
 { [X;\cR_2] }
   \arrow[dl,"\beta_2"]\\
 & X &
\end{tikzcd}
\end{equation}
whose maps are diffeomorphisms on the common interior.

\section{The oriented resolution complex}

\subsection{Polyhedral slices}

Let $\sigma\subset\Lambda_{\R}$ be a strictly convex rational polyhedral cone.  Choose a linear functional
\[
  \ell:\Lambda_{\R}\longrightarrow\R
\]
strictly positive on $\sigma\setminus\{0\}$.  The slice
\[
  K_\sigma=\{x\in\sigma:\ell(x)=1\}
\]
is a compact convex polytope.  A fan subdivision of $\sigma$ induces a polyhedral subdivision of $K_\sigma$.

This elementary observation turns the conical refinement problem into a compact cellular
problem.  Orient $\Lambda_{\R}$ and orient $K_\sigma$ by the outward radial convention.
Choose an orientation on every cell of every subdivision, compatibly on cells that are
identified along common faces.  (An ambient orientation alone does not canonically orient
all lower-dimensional cells.)  Write $[\tau:\eta]\in\{0,\pm1\}$ for the resulting
incidence number.  Changing a cell orientation only conjugates the cellular differential
by the corresponding diagonal sign change.

\begin{definition}
For a refinement $\cR$ of $\sigma$, the \emph{oriented resolution chain complex} is the cellular chain complex
\[
  C_\bullet(\cR;\Z)=C_\bullet(K_\sigma(\cR);\Z)
\]
of the induced subdivision of $K_\sigma$.
\end{definition}

For the global gluing it is convenient to remove the apparent dependence on
the functionals $\ell$.  Put
\[
 \mathbb P_+(\sigma)=(\sigma\setminus\{0\})/\R_{>0}.
\]
Radial projection identifies every positive slice $K_\sigma$ with
$\mathbb P_+(\sigma)$ by a face-preserving cellular homeomorphism.  A face
map of cones induces a canonical closed cellular inclusion of the associated
positive projectivizations.  Thus, for a monoidal complex, we glue the
$\mathbb P_+(\sigma)$ along these canonical face maps and take the cellular
complex of the subdivision induced by $\cR$.  Equivalently one may use any
cone-wise positive slices and transport them through radial projection; no
unmentioned compatibility among the functionals is required.  With the
cell orientations fixed above, the resulting complex is denoted
$C_\bullet(\cR)$.

\subsection{Subdivision maps}

Let $\cR'\to\cR$ be a further refinement.  The identity map on the support is cellular only after passing to a common subdivision, but it induces a standard subdivision chain map
\[
  s_{\cR'\cR}:C_\bullet(\cR)\longrightarrow C_\bullet(\cR').
\]
It sends each oriented cell to the signed sum of the cells that subdivide it.

\begin{lemma}\label{lem:subdivision-chain-map}
The subdivision operator is a chain map:
\[
  \partial s_{\cR'\cR}=s_{\cR'\cR}\partial.
\]
It is natural under restriction to a face of the original cone.
\end{lemma}

\begin{proof}
For an oriented cell $e$, interior facets of the subdivision occur twice with opposite induced orientations and cancel.  The remaining facets lie in $\partial e$ and reproduce the subdivision of the oriented boundary.  Restriction to a face merely restricts the set of subdividing cells, so the same formula is natural.
\end{proof}

\begin{theorem}[Subdivision chain equivalence]\label{thm:subdivision-equivalence}
For every finite refinement $\cR'\to\cR$, the subdivision chain map
\[
  s_{\cR'\cR}:C_\bullet(\cR)\longrightarrow C_\bullet(\cR')
\]
is a chain-homotopy equivalence.  The homotopy inverse and homotopies may be chosen relative to every common subcomplex on which the refinement is trivial.
\end{theorem}

\begin{proof}
Use the relative acyclic-carrier theorem on the finite polyhedral complex
$K_\sigma$~\cite[\S13]{Munkres}.  For every coarse closed cell
$\overline e$, let the carrier be the cellular chains of the subdivision of
$\overline e$.  This carrier is acyclic, and the subdivision map is carried by
it.  Sending a fine cell to chains in the smallest coarse closed cell which
contains it gives, by the carrier theorem, an augmentation-preserving chain map
\[
 a_{\cR\cR'}:C_\bullet(\cR')\longrightarrow C_\bullet(\cR).
\]
The maps $a_{\cR\cR'}s_{\cR'\cR}$ and $\id$ are carried by the coarse closed
cells, while $s_{\cR'\cR}a_{\cR\cR'}$ and $\id$ are carried by their subdivided
closed cells.  The uniqueness-up-to-carried-homotopy part of the same theorem
therefore supplies the two chain homotopies.

If the refinement is trivial on a subcomplex $L$, prescribe
$a_{\cR\cR'}|_L=\id$ and prescribe both homotopies to vanish on $L$.  The
relative extension step in the acyclic-carrier induction extends these choices
cell by cell.  Applying this construction once to the global polyhedral complex
obtained by gluing the projectivized cone slices makes the choices simultaneous
on all unresolved boundary subcomplexes.
\end{proof}

\begin{corollary}[Combinatorial refinement invariance]\label{cor:comb-invariance}
Within an admissible comparison class satisfying \textup{(G3)}, any two finite smooth refinements of $\cP_X$ have oriented resolution chain complexes related by chain-homotopy equivalences through an admitted smooth common refinement.  The equivalences commute, up to relative chain homotopy, with restriction to any unresolved boundary face of $X$.
\end{corollary}

\begin{proof}
Choose the smooth common refinement supplied by \textup{(G3)}.  The
projectivized cone slices of the finite monoidal complex, glued by their face
maps, form one finite polyhedral complex.  Each admitted refinement is a
subdivision of this global complex.  Apply the relative acyclic-carrier
construction of \cref{thm:subdivision-equivalence} once to that global
complex, prescribing the identity and zero homotopies on every subcomplex
where the refinement is trivial.  Since the carriers of cells in an
unresolved boundary subcomplex remain inside that subcomplex, the resulting
maps and homotopies restrict simultaneously to all such boundaries.  Applying
this construction to the two arms of the common
refinement gives the assertion.
\end{proof}

\subsection{Exceptional complexes}

Let $q:\cR\to\cP$ be a refinement.  For a cone $\sigma\in\cP$, write
\[
  \cR_\sigma=\{\tau\in\cR:\supp(\tau)\subseteq\supp(\sigma)\}
\]
for the subdivision lying over $\sigma$.  Its boundary subcomplex is the union of cells supported on proper faces of $\sigma$.

\begin{definition}
The \emph{local exceptional complex} over $\sigma$ is
\[
  E_\bullet(\cR/\sigma)
  =C_\bullet(K_\sigma(\cR),\partial K_\sigma(\cR)).
\]
Its reduced boundary version is
\[
  \widetilde E_\bullet(\cR/\sigma)
  =\widetilde C_\bullet(K_\sigma(\cR)).
\]
Here and below $\widetilde C_\bullet(K)$ denotes the augmented reduced
cellular complex: in degree zero it is $\ker(\epsilon:C_0(K)\to\mathbb Z)$,
and in positive degrees it agrees with the ordinary cellular complex.  A
contraction relative to a vertex $v$ refers to the splitting
$C_0(K)=\mathbb Z v\oplus\ker\epsilon$.
\end{definition}

Because $K_\sigma$ is a ball, the relative complex has the homology of a sphere shifted by one, while the reduced absolute complex is acyclic.  Which one appears in field theory depends on whether the unresolved top face is retained as a generator or replaced by the entire subdivision.

\begin{lemma}[Acyclic reduced exceptional sector]\label{lem:exceptional-acyclic}
For every nonzero sharp cone $\sigma$ and every finite subdivision,
\[
  H_\bullet\bigl(\widetilde E_\bullet(\cR/\sigma)\bigr)=0.
\]
The contraction can be chosen relative to a fixed vertex of the slice.
\end{lemma}

\begin{proof}
The slice $K_\sigma$ is convex, hence contractible, so its augmented reduced
cellular complex is exact.  Write $C_0=\mathbb Z v\oplus\ker\epsilon$ and set
$i_v(1)=v$.  Put $B_q=\operatorname{im}(d:C_{q+1}\to C_q)$.  Exactness gives
$B_0=\ker\epsilon$, and for every $q\geq1$ there is a short exact sequence
\[
 0\longrightarrow B_q\longrightarrow C_q
 \xrightarrow{\ d\ }B_{q-1}\longrightarrow0.
\]
Each $B_{q-1}$ is a subgroup of the finitely generated free abelian group
$C_{q-1}$, hence is free and therefore projective; the sequence consequently
splits.  Choose a complement $L_q$ so that
\[
 C_q=B_q\oplus L_q,\qquad
 d:L_q\xrightarrow{\;\cong\;}B_{q-1}
 \qquad(q\geq1).
\]
Define $h$ to be the inverse of $d:L_q\to B_{q-1}$ on $B_{q-1}$, and zero
on every $L_q$ and on $\mathbb Z v$.  Then, directly on the displayed decomposition,
\[
 dh+hd=\id-i_v\epsilon,\qquad
 hi_v=0,\qquad \epsilon h=0,\qquad h^2=0.
\]
Restricting to $\ker\epsilon$ gives the required contraction of
$\widetilde E_\bullet(\cR/\sigma)$ relative to $v$.
\end{proof}

\part*{II.\ Resolution-level models}\addcontentsline{toc}{part}{II.\ Resolution-level models}

\section{AKSZ descent on a smooth monoidal resolution}

\subsection{Resolution-level AKSZ data}\label{sec:resolution-aksz-data}

Let $X$ be a resolved-admissible generalized-corner manifold of dimension $n$, and let
\[
  \beta_\cR:X_\cR=[X;\cR]\longrightarrow X
\]
be a smooth monoidal resolution.  Then $X_\cR$ is a manifold with ordinary corners.

Let
\[
  (\Y,\omega_\Y=\delta\alpha_\Y,Q_\Y,\Theta)
\]
be an exact Hamiltonian dg symplectic target of degree $n-1$.  For every oriented face $F$ of $X_\cR$, define
\[
  \F_F=\Map(T[1]F,\Y).
\]
Ordinary AKSZ transgression produces
\[
  (\F_F,\omega_F,\alpha_F,S_F,Q_F),
\]
where a codimension-$r$ face carries a two-form of degree $r-1$ and an action of degree $r$.  The modified Hamiltonian identity is
\begin{equation}\label{eq:facewise-mhi}
  \iota_{Q_F}\omega_F
  =(-1)^{\dim F}\delta S_F
   +\sum_{G\prec F}[F:G]\,\rho_{G/F}^*\alpha_G.
\end{equation}
Here $G\prec F$ means that $G$ is a codimension-one face of $F$, the incidence
sign $[F:G]\in\{\pm1\}$ is fixed by the chosen orientations, and
$\rho_{G/F}:\F_F\to\F_G$ is restriction of fields.  The twice-iterated defect
vanishes by the incidence relation $\partial^2=0$.

\begin{definition}
The family of facewise AKSZ data on $X_\cR$ is the \emph{resolution-level AKSZ descent system} and is denoted
\[
  \mathfrak{AKSZ}_{\mathrm{res}}(X,\Y;\cR).
\]
\end{definition}

This system is well defined for every smooth refinement.  The unresolved theory should not be defined by choosing one such system arbitrarily; it should be an invariant object extracted from all of them.

This ordinary-source system is a preliminary resolution-level object.  It maps
by precomposition with the anchor into the resolved $b$/logarithmic system
introduced below, but that map is not asserted to be an equivalence.  The main
resolution--intrinsic comparison concerns the resolved and intrinsic
$b$/logarithmic models.

\subsection{Further refinement and pullback}

Suppose $\cR'\to\cR$ is a smooth further refinement.  The associated map
\[
  \gamma:X_{\cR'}\longrightarrow X_\cR
\]
is a generalized blow-down between ordinary-corner manifolds and is a diffeomorphism on interiors.  Pullback of fields gives a formal map
\[
  \gamma^*:\Map(T[1]X_\cR,\Y)
  \longrightarrow
  \Map(T[1]X_{\cR'},\Y).
\]
On a face $F'$ of the finer resolution, the image lies in a unique smallest face $F$ of the coarser resolution, yielding facewise pullback maps.

There are two immediate limitations.
\begin{enumerate}[label=(\roman*),leftmargin=2.2em]
\item Pullback is not expected to be a symplectomorphism of the full BV spaces: the source domains and their boundary decompositions differ.
\item The finer resolution has exceptional faces with no one-to-one counterpart on the coarser resolution.
\end{enumerate}
Thus refinement invariance must involve a totalization and an elimination of exceptional sectors.

\subsection{Linear coefficient systems}

\begin{definition}
Let $\mathsf{Face}(\cR)$ be the face category of $X_\cR$, with one morphism $G\to F$ whenever $G\subseteq F$.  A \emph{facewise cochain coefficient system} is a contravariant functor
\[
  \cA_\cR:\mathsf{Face}(\cR)^{\op}
  \longrightarrow \mathsf{Ch}(\kfield).
\]
\end{definition}

Examples include:
\begin{enumerate}[label=(\alph*),leftmargin=2.2em]
\item linearized BV fields with differential $Q$;
\item polynomial local observables with the linearized BRST differential;
\item de Rham complexes of faces tensored with a fixed graded vector space;
\item tangent complexes to a regular facewise moduli problem.
\end{enumerate}

Choose orientations of all faces.  Define the bicomplex
\[
  \cC^{p,q}(\cR;\cA)
  =\bigoplus_{F\in\mathsf{Face}_q(\cR)}\cA_\cR^p(F),
\]
where $q$ is codimension.  The vertical differential is the internal differential $d_\cA$, and the horizontal differential is
\begin{equation}\label{eq:horizontal-coeff}
  (d_{\mathrm{face}}a)_G
  =\sum_{F:\,G\prec F}[F:G]\,r_{FG}(a_F).
\end{equation}
The total differential is
\[
  D=d_\cA+(-1)^p d_{\mathrm{face}}.
\]
Functoriality and incidence cancellation imply $D^2=0$.

\begin{definition}
The \emph{linear descent totalization} is
\[
  \Tot(\cR;\cA)
  =\Tot\bigl(\cC^{\bullet,\bullet}(\cR;\cA)\bigr).
\]
\end{definition}

\section{A derived criterion for refinement invariance}

\subsection{Base change}

Let $u:\cR'\to\cR$ be a further refinement.  It induces a support functor
\[
  u_*:\mathsf{Face}(\cR')\longrightarrow\mathsf{Face}(\cR)
\]
that assigns to a fine face the smallest coarse face containing its image.

\begin{definition}[Refinement descent datum]\label{def:refinement-descent-datum}
A family of coefficient systems $\{\cA_\cR\}$ has a \emph{refinement
descent datum} if every further refinement $u:\cR'\to\cR$ is equipped
with:
\begin{enumerate}[label=\textup{(\alph*)},leftmargin=2.5em]
\item a natural transformation
\[
  \eta_u:u_*^*\cA_\cR\longrightarrow\cA_{\cR'}
\]
compatible with compositions and a quasi-isomorphism on faces mapped
diffeomorphically to their support;
\item a degree-zero cochain map
\[
 U_u:\Tot(\cR;\cA_\cR)\longrightarrow
 \Tot(\cR';\cA_{\cR'})
\]
compatible with compositions and with the support filtration, such that the
associated-graded mapping-cone summand over every coarse face $F$ is canonically
identified, up to the codimension shift of $F$, with
\[
 \widetilde\Tot\bigl(\mathsf{Fib}_u(F);\cA_{\cR'}\bigr).
\]
\end{enumerate}
For ordinary cellular coefficient systems, $U_u$ is the usual subdivision
map.  For logarithmic relative de Rham systems, the existence of such a
\emph{filtered} operator is additional structure: ordinary pullback along a
blow-down does not in general preserve face degree because an exceptional
hypersurface may map to a coarse face of higher codimension.  The unfiltered
derived comparison used in \cref{subsec:interior-bridge} is therefore
constructed by an interior roof and does not assume such an operator.
\end{definition}

\subsection{Local exceptional acyclicity}

Fix a face $F$ of the coarse resolution.  Let
\[
  \mathsf{Fib}_u(F)=\{F'\in\mathsf{Face}(\cR'):u_*(F')=F\}
\]
be the exceptional fiber category.  The refinement descent datum induces
an augmentation
\[
  \cA_\cR(F)\longrightarrow
  \Tot\bigl(\mathsf{Fib}_u(F);\cA_{\cR'}\bigr),
\]
and the \emph{augmented reduced totalization}
$\widetilde\Tot(\mathsf{Fib}_u(F);\cA_{\cR'})$ is its mapping cone.

\begin{definition}
The refinement $u$ is \emph{$\cA$-acyclic over $F$} if the augmented reduced totalization
\[
  \widetilde\Tot\bigl(\mathsf{Fib}_u(F);\cA_{\cR'}\bigr)
\]
is acyclic.  It is \emph{locally $\cA$-acyclic} if this holds for every coarse face $F$.
\end{definition}

For a locally constant coefficient system, this condition reduces to the contractibility of the subdivision slice, which is \cref{lem:exceptional-acyclic}.  For a nonlinear or nonconstant field system it is a substantive locality requirement.

\subsection{Comparison theorem}

\begin{theorem}[Derived refinement comparison]\label{thm:derived-comparison}
Let $u:\cR'\to\cR$ be a finite smooth further refinement.  Suppose:
\begin{enumerate}[label=(\roman*),leftmargin=2.2em]
\item the coefficient systems carry a refinement descent datum in the sense of \cref{def:refinement-descent-datum};
\item $u$ is locally $\cA$-acyclic;
\item all direct sums in the totalization are finite, or the systems satisfy a boundedness condition making the filtration complete.
\end{enumerate}
Then the comparison map $U_u$ of the refinement descent datum,
\[
  U_u:\Tot(\cR;\cA_\cR)
  \longrightarrow
  \Tot(\cR';\cA_{\cR'})
\]
is a quasi-isomorphism.
\end{theorem}

\begin{proof}
Filter both total complexes by the codimension of the support in the coarse resolution.  By \cref{def:refinement-descent-datum}, $U_u$ preserves this filtration and the mapping cone of $U_u$ inherits it.  Its associated graded piece over a coarse face $F$ is precisely the reduced exceptional totalization
\[
  \widetilde\Tot\bigl(\mathsf{Fib}_u(F);\cA_{\cR'}\bigr),
\]
up to the codimension shift of $F$.  By local $\cA$-acyclicity, every associated graded piece is acyclic.  The spectral sequence of the finite filtered mapping cone therefore has zero $E_1$ page and converges to zero.  Hence the mapping cone of $U_u$ is acyclic, so $U_u$ is a quasi-isomorphism.
\end{proof}

\begin{corollary}[Comparison through an admitted common refinement]\label{cor:common-derived}
Let $\cR_1$ and $\cR_2$ be two finite smooth refinements in an admissible
comparison class satisfying \textup{(G3)}.  If an admitted smooth common
refinement $\cR_{12}$ is locally $\cA$-acyclic over both, then there is a zigzag of quasi-isomorphisms
\[
  \Tot(\cR_1;\cA_1)
  \xrightarrow{\simeq}
  \Tot(\cR_{12};\cA_{12})
  \xleftarrow{\simeq}
  \Tot(\cR_2;\cA_2).
\]
Assume, moreover, that the common refinements under consideration form a
directed subclass in which every map from a dominating common refinement is
equipped with compatible refinement-transfer data and is locally
$\cA$-acyclic.  Then the equivalence class of the displayed roof in the
derived category is independent of the chosen common refinement.
\end{corollary}

\begin{proof}
Apply \cref{thm:derived-comparison} to the two refinement maps.  Given two
common refinements in the directed subclass of the statement, choose a
dominating common refinement there.  The two maps from that domination are
locally $\cA$-acyclic and carry compatible refinement-transfer data, so
\cref{thm:derived-comparison} makes both of them quasi-isomorphisms.
Compatibility of the subdivision maps with composition then shows that the
resulting roofs agree in the localization at quasi-isomorphisms.
\end{proof}

\begin{definition}
When the hypotheses of \cref{cor:common-derived} hold for a directed cofinal
admissible class of smooth refinements, the resulting derived object is called the \emph{resolution-independent linear descent complex} and is denoted
\[
  \mathbb{AKSZ}_{\mathrm{lin}}(X,\Y).
\]
\end{definition}

The notation is intentionally conservative: the construction is a derived linear shadow of AKSZ data, not yet a nonlinear BV manifold.

\section{Elementary radial blow-up}

\subsection{The local model}

After shrinking a product collar of a codimension-$k$ ordinary face $F$, use
the equivalent $\ell^1$-collar
\[
 U=F\times Q_\varepsilon,
 \qquad
 Q_\varepsilon=\{x\in\mathbb R_{\geq0}^k:\textstyle\sum_i x_i<\varepsilon\}.
\]
The radial blow-up replaces the normal cone by polar variables
\[
  x_i=r u_i,
  \qquad
  r\in[0,\varepsilon),
  \qquad
  u=(u_1,\ldots,u_k)\in\Delta^{k-1},
\]
where
\[
  \Delta^{k-1}=\{u_i\geq0:\sum_i u_i=1\}.
\]
The local blow-down is
\[
  \beta:F\times[0,\varepsilon)\times\Delta^{k-1}
  \longrightarrow F\times Q_\varepsilon,
  \qquad
  \beta(y,r,u)=(y,ru).
\]
The exceptional hypersurface is $F\times\{0\}\times\Delta^{k-1}$.

\begin{proposition}[Homotopy invariance of radial blow-up]\label{prop:radial-homotopy}
The blow-down of an elementary radial blow-up is a homotopy equivalence.  It is a diffeomorphism away from the blown-up face, and both the original collar and the blown-up collar deformation retract onto $F$.
\end{proposition}

\begin{proof}
View $Q_\varepsilon$ as the cone on $\Delta^{k-1}$, writing its points as
$[r,u]$, with all $[0,u]$ identified.  Fix $u_0\in\Delta^{k-1}$ and choose a
continuous function $a:[0,\varepsilon)\to[0,1]$ which is zero near $0$ and one
for $r\geq\varepsilon/2$.  Put
\[
 q_r(u)=(1-a(r))u_0+a(r)u.
\]
The map
\[
 g:F\times Q_\varepsilon\longrightarrow
 F\times[0,\varepsilon)\times\Delta^{k-1},\qquad
 g(y,[r,u])=(y,r,q_r(u)),
\]
is continuous at the cone vertex because $q_r(u)\to u_0$ uniformly as
$r\to0$.  It is the usual inverse of blow-down for
$r\geq\varepsilon/2$.

Linear interpolation in the convex simplex from $u$ to $q_r(u)$ gives
homotopies $\beta g\simeq\id$ on the cone and $g\beta\simeq\id$ on the
blown-up collar.  At $r=0$ the second homotopy contracts the exceptional
simplex to $u_0$; on the cone the first homotopy is well defined because all
angular values at $r=0$ represent the same vertex.  Both homotopies are fixed
for $r\geq\varepsilon/2$.  Hence they glue literally to the identity on the
unchanged complement of the collar, proving the global statement.
\end{proof}

\begin{corollary}\label{cor:de-rham-blowup}
Let $N$ be a manifold with ordinary corners and $F$ an ordinary corner face.
Pullback induces an isomorphism
\[
  \beta^*:H_{\mathrm{dR}}^\bullet(N)
  \xrightarrow{\cong}
  H_{\mathrm{dR}}^\bullet([N;F]).
\]
The same holds for a finite composition of elementary radial blow-ups.
\end{corollary}

\begin{proof}
De Rham cohomology is homotopy invariant.  Apply \cref{prop:radial-homotopy} and compose.
\end{proof}

\subsection{The exceptional simplex as a face complex}

The new faces over $F$ are indexed by faces of $\Delta^{k-1}$.  A
$q$-dimensional simplex face corresponds to a resolved face of normal
codimension $k-q$.  Accordingly we use the face-graded reduced complex
\[
 \widetilde C_{\mathrm{face}}^r(\Delta^{k-1})
 :=\widetilde C_{k-r}(\Delta^{k-1}),
 \qquad 1\le r\le k,
\]
whose differential is the cellular boundary, now of face degree $+1$.

\begin{proposition}[Exceptional acyclicity for constant coefficients]\label{prop:simplex-acyclic}
Let $A^\bullet$ be any cochain complex.  The exceptional total complex
\[
  A^\bullet\otimes
  \widetilde C_{\mathrm{face}}^\bullet(\Delta^{k-1};\kfield)
\]
is acyclic.
\end{proposition}

\begin{proof}
The reduced cellular chain complex of a simplex is contractible.  Regrading
by $r=k-q$ does not change acyclicity; the cone homotopy raises cellular
chain degree by one and therefore lowers face degree by one.  Tensoring with
$A^\bullet$ gives the stated contraction with the standard Koszul sign.
\end{proof}

This proposition is the local algebraic reason that an elementary blow-up does not change a facewise theory whose coefficients are constant in the normal simplex direction.

\part*{III.\ The intrinsic route}\addcontentsline{toc}{part}{III.\ The intrinsic route}

\section{Monoidal logarithmic forms}

\subsection{The logarithmic extension}

The AKSZ integrand on $\bT[1]X$ is a $b$-form.  To regularize it and to restrict
logarithmic primitives to boundary faces, we need a logarithmic extension which does
not overcount expressions that are already smooth.  This point is essential: the
logarithmic-function algebra of Dupont--Panzer--Pym is not the free algebra on formal
logarithms modulo additivity alone; it also identifies a logarithmic expression with
its ordinary smooth value whenever that value extends smoothly \cite[Def.~5.1]{DPP}.

On a monoidal chart $X_P$, write $\lambda_p$ for the nonnegative monomial associated
with $p\in P$.  Let
\[
 \mathcal M^{\mathrm{bas}}_P
 :=C^\infty_{X_P,>0}\cdot \lambda^P
\]
be the sheaf of positive smooth units times basic monomials, with structure map
$\alpha:\mathcal M^{\mathrm{bas}}_P\to C^\infty_{X_P,\ge0}$.  We use multiplicative
notation for this monoid and write $(\mathcal M^{\mathrm{bas}}_P)^{\gp}$ for its
groupification.  This is the positive log structure of Gillam--Molcho
\cite{GillamMolcho}, in whose framework Dupont--Panzer--Pym work
\cite[\S3]{DPP}.

\begin{definition}[DPP-compatible monoidal logarithmic functions]\label{def:logalgebra}
Let $\cO_{P}^{\log}$ be the sheaf of $C^\infty_{X_P}$-algebras generated by formal
symbols $\ell_f$, $f\in(\mathcal M^{\mathrm{bas}}_P)^{\gp}$, modulo the following
relations:
\begin{enumerate}[label=\textup{(\roman*)},leftmargin=2.5em]
\item $\ell_{f_1f_2}=\ell_{f_1}+\ell_{f_2}$ and $\ell_1=0$;
\item if $f\in\mathcal M^{\mathrm{bas}}_P$ and $g\in C^\infty_{X_P}$ are such that
$g\log(\alpha(f))$ extends to a smooth function $h$ on $X_P$, then
\[
 g\,\ell_f=h.
\]
\end{enumerate}
The monoidal logarithmic de Rham algebra $\logOmega^\bullet(X_P)$ is the commutative
dg algebra generated by $\bOmega^\bullet(X_P)$ and $\cO_P^{\log}$, modulo the
differential ideal generated by the relations above, with differential extending
$\db$ and satisfying, for a positive smooth unit $u$ and $p\in P$,
\begin{equation}\label{eq:dlog-monomial}
 \db\ell_{u\lambda_p}=\frac{\dd u}{u}+\vartheta_p.
\end{equation}
Equivalently, it is the universal logarithmic extension of the $b$-de Rham algebra
with these degree-zero relations.
The two displayed subalgebras are amalgamated over their common
\(C^\infty_{X_P}\)-subalgebra; equivalently, the construction is the
corresponding pushout of graded-commutative algebras, followed by the stated
differential quotient.
\end{definition}

\begin{remark}[Why relation \textup{(ii)} is necessary]\label{rem:dpp-second-relation}
Relation \textup{(ii)} is the smooth-identification relation of
Dupont--Panzer--Pym.  In particular, if $u>0$ is smooth then
$\ell_u=\log u$, and hence
\[
 \ell_{u\lambda_p}=\log u+\ell_{\lambda_p}.
\]
Thus changes of monoidal chart by positive smooth units are built into the algebra,
rather than being imposed only later at the level of scales.  More generally, if a
coefficient kills the logarithmic singularity so that $g\log(\alpha(f))$ is smooth,
the formal expression $g\ell_f$ is identified with that smooth function.  This is the
relation missing from the naive free-logarithm algebra.
\end{remark}

For $p\in P^{\gp}$, choose $p=a-b$ with $a,b\in P$ and put
\[
 \ell_p:=\ell_{\lambda_a\lambda_b^{-1}}.
\]
Relation \textup{(i)} makes this independent of the presentation $p=a-b$.  For
$p\in P$, the symbol $\ell_p$ represents $\log\lambda_p$ on the interior, but its
value along a vanishing face is defined only after regularization.

\begin{lemma}[Monoidal relations]\label{lem:monoidal-relations}
For all $p,q\in P^{\gp}$,
\[
 \vartheta_{p+q}=\vartheta_p+\vartheta_q.
\]
The corresponding formal logarithms are additive in $P^{\gp}$, while multiplication
of a monomial by a positive smooth unit changes its logarithm by the ordinary smooth
function $\log u$.
\end{lemma}

\begin{proof}
The first assertion is the logarithmic derivative of
$\lambda_{p+q}=\lambda_p\lambda_q$ and extends to $P^{\gp}$ by additivity.  The second
is relation \textup{(i)} together with \cref{rem:dpp-second-relation}.
\end{proof}

\begin{remark}[Intrinsic versus resolved logarithmic complexes]\label{rem:intrinsic-vs-dpp-log}
Definition~\ref{def:logalgebra} is an intrinsic monoidal candidate on a generalized
corner chart.  No logarithmic de Rham theorem for this unresolved nonsimplicial object
is assumed in this paper.  Whenever we pass to a smooth resolution, every resolved
face is an ordinary manifold with corners endowed with its full basic positive
log structure, which contains the pullbacks of the basic monomials from the
unresolved chart, and we use the logarithmic function and form sheaves
$C^{\infty,\log}$ and $A^{\bullet,\log}$ of Dupont--Panzer--Pym themselves.  Their
Definition~5.1 contains exactly relations \textup{(i)}--\textup{(ii)}, and their
Definition~6.4 constructs logarithmic forms from the logarithmic cotangent sheaf
\cite{DPP}.  Thus the DPP de Rham theorem is invoked only on the resolved
ordinary-corner spaces where its hypotheses apply.
\end{remark}

\begin{remark}
The extension $\logOmega^\bullet$ is larger than the Chevalley--Eilenberg algebra of
the source.  The fields live on $\bT[1]X$, while the regularization of their action is
allowed to use logarithmic primitives in $\logOmega^\bullet$.
\end{remark}

\begin{lemma}[Functoriality of the corrected logarithmic algebra]\label{lem:log-functoriality}
Let $\phi:X_Q\to X_P$ be a smooth monoidal map for which pullback carries
$\mathcal M_P^{\mathrm{bas}}$ to $\mathcal M_Q^{\mathrm{bas}}$.  Then pullback of
smooth functions and logarithmic monoid sections extends uniquely to a morphism of
commutative dg algebras
\[
 \phi^*:\logOmega^\bullet(X_P)\longrightarrow\logOmega^\bullet(X_Q).
\]
In particular, the local constructions glue across monoidal chart transitions.
\end{lemma}

\begin{proof}
Multiplicativity preserves relation \textup{(i)}.  If
$g\log(\alpha(f))=h$ is smooth, then
\[
 \phi^*g\,\log\!\bigl(\alpha(\phi^*f)\bigr)=\phi^*h
\]
is smooth, so relation \textup{(ii)} is preserved.  Pullback also preserves the differentials of these
relations, hence descends from the free dg algebra to the stated quotient.
\end{proof}

\begin{lemma}[Pullback to a smooth monoidal resolution]\label{lem:log-pullback-resolution}
Let $\beta_\cR:X_\cR\to X_P$ be a smooth monoidal resolution.  Pullback of smooth
functions and monomials induces a canonical morphism of commutative dg algebras
\begin{equation}\label{eq:log-pullback-resolution}
 \beta_\cR^*:\logOmega^\bullet(X_P)
 \longrightarrow A_{X_\cR}^{\bullet,\log},
\end{equation}
where the target is the Dupont--Panzer--Pym logarithmic de Rham algebra of the
ordinary-corner resolution.  After a compatible DPP regularization
$s_\cR$ has been chosen, composition with its virtual pullback gives the
corresponding morphism on any resolved face.  This facewise morphism is
relative to $s_\cR$ and is not canonical from $\beta_\cR$ alone.
\end{lemma}

\begin{proof}
The monoidal map pulls a basic monomial $u\lambda_p$ back to a basic logarithmic
section on the resolution and preserves products, so relation \textup{(i)} is
preserved.  If $g\log(\alpha(f))=h$ is smooth on $X_P$, then
\[
 \beta_\cR^*g\,\log\!\bigl(\alpha(\beta_\cR^*f)\bigr)=\beta_\cR^*h
\]
is smooth on $X_\cR$; hence the DPP smooth-identification relation implies that
relation \textup{(ii)} is preserved as well.  Finally,
$\db\ell_{u\lambda_p}$ pulls back to the logarithmic differential of the pulled-back
section, so the differential ideal is preserved.  This gives
\eqref{eq:log-pullback-resolution}.  Given $s_\cR$, the same argument followed
by its virtual pullback gives the asserted face morphism; its dependence on
the regularization is exactly the qualification in the statement.
\end{proof}

\subsection{Local scales}

At a zero-dimensional stratum of $X_P$ for a sharp monoid, a scale should
assign finite regularized values to all $\ell_p$ without violating monoidal
relations.  If units are present, their logarithms are ordinary smooth
functions and are not independent normal-scale parameters.

\begin{definition}\label{def:vertex-scale}
Let $P$ be sharp.  A monoidal scale at the vertex $v$ of $X_P$ is an additive homomorphism
\[
 \scale:P^{\gp}\longrightarrow\R.
\]
Its regularized evaluation is required to agree with ordinary evaluation on smooth
functions and, on a logarithmic monomial with positive smooth unit $u$, is
\begin{equation}\label{eq:vertex-scale-unit}
 \Reg_\scale(\ell_{u\lambda_p})
 =\log u(v)+\scale(p),
 \qquad
 \Reg_\scale(\vartheta_p)=0.
\end{equation}
\end{definition}

For a nonsharp local model at a point $v$, the normal scale is instead a character of
$(P^\sharp)^{\gp}$.  To assign regularized values to all of $P^{\gp}$ one must also
choose a splitting, or equivalently an extension, whose restriction to
$(P^\times)^{\gp}$ is the additive character
$u\mapsto\log\lambda_u(v)$ coming from ordinary evaluation.  Such an extension is not
canonical; different choices encode tangential data.  In particular, the values of unit
monomials are fixed by ordinary evaluation and are not independent normal-scale
parameters.

The unit term in \eqref{eq:vertex-scale-unit} is forced by the smooth-identification
relation in \cref{def:logalgebra}.  The second formula is the logarithmic analogue of
setting the normal form $\dd r/r$ to zero under regularized restriction while retaining
a finite value for $\log r$.  This is the linear/constant-scale special case of the
virtual-morphism notion of scale used by Dupont--Panzer--Pym on ordinary resolved
faces \cite[Def.~4.12]{DPP}.  It is a pointwise normalization in a chosen
local trivialization, not a construction of compatible normal sections on an
entire resolved boundary diagram; see
\cref{prop:vertex-character-insufficient}.

For a positive-dimensional face $F\le P$, the normal group is
\[
 N_{F/P}^{\gp}=P^{\gp}/F^{\gp}.
\]
A normal scale is an element of $\Hom(N_{F/P}^{\gp},\R)$.  A choice of splitting of
\[
 0\longrightarrow F^{\gp}\longrightarrow P^{\gp}
 \longrightarrow N_{F/P}^{\gp}\longrightarrow0
\]
then gives a regularized restriction of logarithmic generators.  Different splittings change the restriction by tangential logarithms; globally this is best encoded by the virtual-morphism language used in logarithmic geometry.

\section{Intrinsic face categories and regularized Stokes trace systems}

\subsection{The intrinsic face category}

For a Joyce manifold with generalized corners whose boundary faces are
embedded, the intrinsic incidence object used below is the category of
embedded boundary faces, not the iterated boundary $\partial^kX$.
This distinction is essential for genuine $g$-corners.  Joyce shows that,
unlike the ordinary-corner case, $\partial^kX$ need not carry a natural
$S_k$-action for $k\geq3$.  By contrast, Kottke's monoidal complex records
canonically the embedded boundary faces and all of their inclusions.

\begin{definition}[Intrinsic face category]\label{def:intrinsic-face-category}
Let $X$ be a manifold with generalized corners with embedded boundary faces.
The \emph{intrinsic face category} $\mathsf{Face}(X)$ is the poset category
whose objects are $X$ and the connected embedded boundary faces of $X$, and
with a unique morphism
\[
 G\longrightarrow F
\]
whenever $G\subseteq F$.  Put $|F|=\codim_XF$.  We write $G\prec F$ when
$G\subseteq F$ and $|G|=|F|+1$.
\end{definition}

Under Kottke's correspondence, the objects of $\mathsf{Face}(X)$ are the
strata encoded by the monoidal complex $\cP_X$, and inclusions of faces are
encoded by face inclusions of the associated normal monoids.  A
\emph{flag} is a chain
\[
 F_0\succ F_1\succ\cdots\succ F_r
\]
of codimension-one incidences.  Flags are simply simplices in the nerve of
$\mathsf{Face}(X)$; in particular, they are not identified with points of
$\partial^rX$.

Use the face-orientation datum of \textup{(G1)}.  Thus every
codimension-one inclusion $G\prec F$ has an incidence sign
$[F:G]\in\{\pm1\}$, fixed by the comparison between the chosen
orientation of $F$, the boundary coorientation given by the positive normal
ray, and the chosen orientation of $G$.  These signs satisfy the following
relation.

\begin{lemma}[Intrinsic incidence cancellation]\label{lem:intrinsic-incidence}
If $H\subset F$ and $|H|=|F|+2$, then
\begin{equation}\label{eq:intrinsic-incidence}
 \sum_{G:\,H\prec G\prec F}[F:G][G:H]=0.
\end{equation}
Consequently the signed cellular incidence operator on
$\mathsf{Face}(X)$ squares to zero.
\end{lemma}

\begin{proof}
The statement is local along the interior of $H$.  The normal monoid of $H$
inside $F$ has rank two.  Its real cone is therefore a strictly convex
rational polyhedral cone of dimension two, and a transverse compact slice is
an oriented interval.  The codimension-one intermediate faces
$H\prec G\prec F$ correspond to the two endpoints of this interval.  Their
induced boundary orientations are opposite, which is precisely
\eqref{eq:intrinsic-incidence}.  Equivalently, this is the cellular identity
$\partial^2=0$ on the normal slice.  Gluing these local normal slices over
$H$ gives the global incidence identity.
\end{proof}

\begin{remark}[Orientation scope]\label{rem:orientation-scope}
The scalar incidence signs used in this paper are therefore not asserted to
exist canonically for an arbitrary oriented $g$-corner.  The main theorems are
stated in the face-oriented setting of \textup{(G1)}.  A formulation with
orientation-line coefficients should remove this auxiliary choice, but that
twisted version is not developed here.
\end{remark}

\begin{remark}[Why iterated boundaries are not used intrinsically]\label{rem:no-sk-gcorners}
For ordinary corners one may identify $\partial^kX$ with ordered
$k$-tuples of distinct local boundary components and divide by the free
$S_k$-action.  Joyce proves that this description fails for general
$g$-corners: already in the cone-over-a-square example, $\partial^3X$ has
eight points over the vertex whereas $C_3(X)$ has one, and the natural
transpositions on iterated boundaries do not in general satisfy the
relations of $S_k$.  The category $\mathsf{Face}(X)$ avoids this obstruction
and retains exactly the incidence information needed for Stokes descent.
\end{remark}

\subsection{The trace axiom}

\begin{definition}[Multiplicative regularized Stokes trace system]\label{def:trace-system}
Let $X$ be a compact $g$-corner with embedded boundary faces and the coherent
face-orientation datum of \textup{(G1)}.  A \emph{multiplicative regularized
Stokes trace system} $\cT$ consists of:
\begin{enumerate}[label=\textup{(\roman*)},leftmargin=2.5em]
\item a commutative differential graded algebra $\logOmega^\bullet(F)$
containing $\bOmega^\bullet(F)$ for every $F\in\mathsf{Face}(X)$;
\item for every inclusion $G\subseteq F$, a unital morphism of commutative
differential graded algebras
\[
 \Reg_{G/F}:\logOmega^\bullet(F)\longrightarrow\logOmega^\bullet(G),
\]
called regularized restriction, such that
$\Reg_{G/F}(\bOmega^\bullet(F))\subseteq\bOmega^\bullet(G)$;
\item strict functoriality along the face category,
\begin{equation}\label{eq:reg-functorial}
 \Reg_{H/G}\circ\Reg_{G/F}=\Reg_{H/F}
 \qquad(H\subseteq G\subseteq F),
\end{equation}
and $\Reg_{F/F}=\id$;
\item a degree-zero linear map
\[
 \Tr_F^{\cT}:{\logOmega}_c^\bullet(F)\longrightarrow\R[-\dim F]
\]
for every face $F$;
\item the exact Stokes identity
\begin{equation}\label{eq:stokes-trace}
 \Tr_F^{\cT}(\db\eta)
 =\sum_{G\prec F}[F:G]\,
 \Tr_G^{\cT}(\Reg_{G/F}\eta)
\end{equation}
for every compactly supported
$\eta\in\logOmega^{\dim F-1}(F)$.
\end{enumerate}
\end{definition}

\begin{remark}[Trace versus nondegeneracy]\label{rem:trace-not-nondegenerate}
The axioms in \cref{def:trace-system} are Stokes and multiplicativity axioms.
They do not imply that a two-form obtained by transgression is weakly
nondegenerate, and they do not by themselves construct a BV bracket on an
infinite-dimensional field space.  Whenever Hamiltonian vector fields or a
classical master equation are asserted below, the required nondegeneracy and
formal mapping-space hypotheses are stated separately.
\end{remark}

\begin{remark}[Linear variant]\label{rem:linear-trace-system}
For purely linear cochain theories one may weaken \textup{(ii)} to a
functorial degree-zero cochain map.  The multiplicativity and the preservation
of $\bOmega$ are required here because a nonlinear AKSZ field is encoded by an
algebra morphism; without them, regularized source restriction need not define
a field on the boundary face.  All intrinsic nonlinear AKSZ statements below
use the multiplicative version of \cref{def:trace-system}.
\end{remark}

\begin{proposition}[Face differential]\label{prop:intrinsic-face-differential}
Let $\cA$ be a contravariant cochain system on $\mathsf{Face}(X)$, with
restriction maps $r_{G/F}$ satisfying strict functoriality.  Set
\[
 \cC^{q,k}(X;\cA)
 =\bigoplus_{\substack{F\in\mathsf{Face}(X)\\ |F|=k}}\cA^q(F)
\]
and define
\begin{equation}\label{eq:intrinsic-face-diff}
 (d_{\mathrm{face}}a)_G
 =\sum_{F:\,G\prec F}[F:G]\,r_{G/F}(a_F).
\end{equation}
Then $d_{\mathrm{face}}^2=0$.  If the internal differential commutes with all
$r_{G/F}$, the total operator
\[
 D=d_{\cA}+(-1)^q d_{\mathrm{face}}
\]
on a component of internal degree $q$ satisfies $D^2=0$.
\end{proposition}

\begin{proof}
Fix $H$ of codimension $k+2$.  The contribution of a face $F$ of codimension
$k$ to $(d_{\mathrm{face}}^2a)_H$ is, by functoriality,
\[
 \left(\sum_{G:\,H\prec G\prec F}[F:G][G:H]\right)
 r_{H/F}(a_F),
\]
which vanishes by \cref{lem:intrinsic-incidence}.  The total-complex statement
is the standard Koszul sign computation.
\end{proof}

\begin{remark}[No separate residue anomaly]\label{rem:no-anomaly}
If one tries to pull a logarithmic primitive back by ordinary restriction, the
result may be undefined and a residue term appears to be missing.  In
\cref{eq:stokes-trace}, the boundary map is the regularized restriction
$\Reg_{G/F}$.  The residue information is therefore contained in the
boundary field itself.
\end{remark}

\subsection{Ordinary corners and the ordered-boundary model}

In the DPP formulas below, $(\partial^kN)^{\mathrm{bas}}$ denotes the
\emph{basic} part of the $k$-fold boundary as a manifold with log corners.
Its underlying ordinary-corner manifold is the usual ordered boundary
$\partial^kN$.  The face maps in the relative logarithmic complex are the
scale-induced virtual morphisms
$(\partial^{k+1}N)^{\mathrm{bas}}\to(\partial^kN)^{\mathrm{bas}}$.
Phantom logarithmic directions enter those virtual pullbacks and their
regularized values; they do not index additional geometric faces in the
ordered sum.  This is the basic/phantom distinction in
\cite[Secs.~3.3 and~6.4]{DPP}.

\begin{definition}[Unordered DPP descent datum]\label{def:dpp-unordered-descent}
Let $F$ be an embedded face of an ordinary-corner manifold and let
$G\subseteq F$ have relative codimension $r$.  Denote by
$\mathcal O_{G/F}\subseteq(\partial^rF)^{\mathrm{bas}}$ the free $S_r$-orbit
of ordered lifts of $G$.  For an ordered lift $\widetilde G$, write
$A_{\log}^\bullet(\widetilde G)$ for its DPP logarithmic de Rham algebra,
including the phantom logarithmic directions carried by the iterated
boundary.

An \emph{unordered DPP descent datum} consists of one cdga
$A_G^\bullet$ for every embedded face and equivariant isomorphisms
\begin{equation}\label{eq:dpp-unordered-identification}
 \phi_{G/F}:
 \left(\prod_{\widetilde G\in\mathcal O_{G/F}}
 A_{\log}^\bullet(\widetilde G)\right)^{S_r}
 \xrightarrow{\;\cong\;}A_G^\bullet
\end{equation}
which intertwine the chosen scales and phantom directions.  For every flag
$H\subseteq G\subseteq F$, the isomorphism obtained by first applying
$\phi_{G/F}$ and then $\phi_{H/G}$ is required to agree with
$\phi_{H/F}$ after the canonical concatenation of the two ordered lift
orbits.  The corresponding identifications of sign representations with
orientation lines are also required to agree with the fixed face
orientations.
\end{definition}

This datum is additional to the symmetric semi-simplicial regularization of
Dupont--Panzer--Pym: their construction supplies the ordered lift algebras,
permutation actions and virtual pullbacks, but does not by itself identify all
phantom-enhanced ordered lifts with one preassigned cdga on the embedded
unordered face.

\begin{proposition}[Conditional ordered-to-unordered descent]
\label{lem:dpp-unordered-descent}
An $S_r$-equivariant DPP regularization together with an unordered descent
datum defines unital cdga maps
\begin{equation}\label{eq:dpp-unordered-reg}
 \Reg_{G/F}
 =\phi_{G/F}\circ
 \left(\prod_{\widetilde G\in\mathcal O_{G/F}}
 v_{\widetilde G/F}^*\right)^{S_r}
 \circ\phi_{F/F}^{-1}
 :A_F^\bullet\longrightarrow A_G^\bullet,
\end{equation}
which satisfy
$\Reg_{H/G}\Reg_{G/F}=\Reg_{H/F}$ for every flag
$H\subseteq G\subseteq F$.
\end{proposition}

\begin{proof}
For every ordering of the $r$ intervening boundary hypersurfaces, compose the
corresponding DPP virtual pullbacks.  Equivariance assembles these component
maps into a map to the invariant product in
\eqref{eq:dpp-unordered-identification}; composing with $\phi_{G/F}$ gives
\eqref{eq:dpp-unordered-reg}.  For a flag, concatenate the two orderings.
The symmetric semi-simplicial identities identify the resulting composite
with the virtual pullback for $H\subseteq F$, while the cocycle clause in the
descent datum identifies the two invariant products.  This is exactly the
stated strict composition law.
\end{proof}

\begin{theorem}[DPP ordered trace and conditional unordered descent]
\label{thm:dpp}
Let $N$ be an oriented manifold with ordinary corners, equipped with its basic
positive log structure and a compatible DPP regularization on all iterated
boundaries.  The logarithmic de Rham algebras, virtual face maps and
regularized integrals define the symmetric semi-simplicial ordered relative
complex and its Stokes cochain trace of
\cite[Defs.~6.13, 7.1 and Cor.~7.11]{DPP}.
Here the differential and trace are understood in the shifted-totalization
convention of DPP.  The explicit diagonal conversion to the column
totalization used later in this paper is given in
\cref{eq:dpp-total-differential,eq:dpp-column-conjugation}.
If an unordered DPP descent datum is additionally supplied, they define a
regularized Stokes trace system on $\mathsf{Face}(N)$ in the sense of
\cref{def:trace-system}.
\end{theorem}

\begin{proof}
The ordered symmetric semi-simplicial statement and its regularized Stokes
formula are precisely the cited DPP results.  Under the additional descent
datum, the cdga restriction maps and their strict composition are
\cref{lem:dpp-unordered-descent}.

For a codimension-$k$ face, its ordered lifts form a free $S_k$-orbit.
Tensoring the component index with $\sgn_k$ and taking invariants identifies
that orbit, by the orientation clause of the datum, with the orientation line
of the unordered face.  The alternating ordered face maps then become the
incidence numbers $[F:G]$.  Summing the ordered Stokes formula over the orbit
and applying $1/k!$ cancels its $k!$ ordered lifts and gives
\eqref{eq:stokes-trace}.  The factorial occurs only in the ordered-to-unordered
total trace, not in an individual face map.
\end{proof}

\begin{proposition}[Dictionary on a smooth resolution]\label{prop:ordered-face-dictionary}
Let $N$ be a manifold with ordinary corners and embedded boundary faces.
Let $\cA^\bullet$ be a coefficient system on the ordered iterated boundaries
whose pullbacks and permutation identifications are $S_k$-equivariant; this
includes the DPP logarithmic complexes with an equivariant regularization.
After choosing orientations of the faces of $N$, there is a degree-preserving identification
\begin{equation}\label{eq:ordered-unordered-identification}
 \bigoplus_{\substack{F\in\mathsf{Face}(N)\\ |F|=k}}
 \cA^\bullet(F)
 \cong
 \left(\cA^\bullet((\partial^kN)^{\mathrm{bas}})\otimes\sgn_k\right)^{S_k}.
\end{equation}
Here the summand denoted \(\cA^\bullet(F)\) is the invariant object attached
to the complete \(S_k\)-orbit of ordered lifts.  It may be replaced by a
preassigned cdga on the embedded face only when an unordered descent datum in
the sense of \cref{def:dpp-unordered-descent} is present.
Under this identification the signed face differential agrees with the
alternating ordered-boundary differential.  The factor $1/k!$ in the
ordered-boundary integration formula compensates exactly for the $k!$
ordered lifts of each unoriented codimension-$k$ face.
\end{proposition}

\begin{proof}
For ordinary corners, the underlying manifold of $(\partial^kN)^{\mathrm{bas}}$
is the free $S_k$-cover of the
codimension-$k$ corner space obtained by ordering the $k$ distinct local
boundary components.  On each $S_k$-orbit, tensoring with the sign
representation and taking invariants leaves one copy, with the sign determined
by the induced orientation.  Deleting one entry from an ordered tuple gives
the usual alternating boundary map, which becomes the oriented cellular
incidence map after quotienting.  The same orbit count gives the normalization
of the integral.
\end{proof}

\subsection{Two descriptions of the face diagram}\label{subsec:face-dictionary}

Both routes will henceforth use face categories.  On a smooth resolution
$X_\cR$, the category $\mathsf{Face}(X_\cR)$ is the face category of the
refined monoidal complex.  Intrinsically, the category
$\mathsf{Face}(X)$ is the face category of Kottke's unresolved monoidal
complex $\cP_X$.  The blow-down induces the support functor
\[
 q_\cR:\mathsf{Face}(X_\cR)\longrightarrow\mathsf{Face}(X),
\]
which sends a resolved face to the smallest intrinsic face containing its
image.  Refinement may create several resolved faces over one intrinsic face
and may create exceptional faces; therefore $q_\cR$ is not expected to be
bijective.

On the resolved ordinary-corner manifold one may, when convenient, replace
$\mathsf{Face}(X_\cR)$ by the equivalent ordered-boundary presentation
\eqref{eq:ordered-unordered-identification}.  No such replacement is made on
the unresolved $g$-corner.  In both descriptions $G\prec F$ denotes a
codimension-one incidence and $[F:G]$ its oriented incidence sign.  All face
differentials are denoted $d_{\mathrm{face}}$, while $\delta$ remains reserved
for the differential on fields and targets.

\section{Resolution-induced derived traces on affine models}\label{sec:derived-traces}

\subsection{Vertex characters and resolved regularizations}

Write a sharp toric monoid as
\[
 P=\sigma^\vee\cap\Lambda^*,
 \qquad P^{\gp}=\Lambda^*.
\]
A character
\[
 \scale\in V_P^*=\Hom(P^{\gp},\R)
\]
gives the vertex values of \cref{def:vertex-scale}.  Let $\cR$ be a smooth
refinement of the monoidal complex of $X_P$, with blow-down
\cite{KottkeGCorner}
\[
 \beta_{\cR}:X_{\cR}\longrightarrow X_P.
\]
If $\tau\subseteq\sigma$ is a full-dimensional unimodular cone, the
corresponding ordinary-corner chart has free sharp monoid
\[
 Q_\tau=\tau^\vee\cap\Lambda^*,
 \qquad P\subseteq Q_\tau,
 \qquad Q_\tau^{\gp}=P^{\gp}.
\]
The character assigns constants to the free generators of $Q_\tau$, but
these chartwise constants do not by themselves define a DPP scale.  Indeed,
under a change of boundary defining function $\rho'=u\rho$, a DPP scale must
transform by the virtual-morphism rule of
\cite[Defs.~4.12 and~7.1]{DPP},
\begin{equation}\label{eq:dpp-scale-transition}
 \lambda_{\rho'}=u|_F\lambda_\rho,
 \qquad
 \Reg(\log\rho')=\log(u|_F)+\Reg(\log\rho),
\end{equation}
and the restricted unit $u|_F$ is generally a nonconstant function on the
boundary face.

\begin{proposition}[Vertex characters do not determine resolved regularizations]
\label{prop:vertex-character-insufficient}
A character $\scale\in\Hom(P^{\gp},\R)$ does not, in general, determine a
DPP regularization of a smooth monoidal resolution.  Compatible positive
normal sections on the complete ordered boundary diagram are additional
data.
\end{proposition}

\begin{proof}
Blow up the origin in $X_{\mathbb N^2}=[0,\infty)^2$.  On the two standard
charts write
\[
 U_y:\ (x,y)=(tr,r),
 \qquad
 U_x:\ (x,y)=(s,su).
\]
On their overlap $s=tr$ and $u=t^{-1}$.  The exceptional hypersurface is
$r=0$ in $U_y$ and $s=0$ in $U_x$.  Treating the two pure monomial boundary
coordinates independently would assign the constants
\[
 \lambda_r=e^{\scale(e_2)},
 \qquad
 \lambda_s=e^{\scale(e_1)}.
\]
However, $s=tr$ and \eqref{eq:dpp-scale-transition} require
\[
 \lambda_s(t)=t\lambda_r(t)=t e^{\scale(e_2)}
\]
on the interior of the exceptional face.  This cannot equal the constant
$e^{\scale(e_1)}$ for all $t>0$.  Thus additivity of $\scale$ controls the
lattice relation but does not replace the nonconstant tangential unit in the
DPP transition law.
\end{proof}

\begin{remark}[Chosen DPP data]\label{rem:linear-vs-general-dpp}
The character $\scale$ remains meaningful as an intrinsic normalization at a
zero-dimensional stratum, but it is not a parametrization of the full
regularization on a resolution.  For every smooth resolution $X_{\cR}$ we
therefore choose a nondegenerate DPP regularization $s_{\cR}$ on its complete
ordered boundary diagram and write
$\mathbb A_{\cR,s_{\cR}}^\bullet(P)$ and
$\mathcal I_{\cR,s_{\cR}}$ for the resulting total complex and trace.  If one
also prescribes vertex-character values, any compatible extension realizing
them is part of the chosen datum $s_{\cR}$ and is not determined by the
character alone.  We do \emph{not} assume a canonical pullback of DPP
regularizations along a monoidal blow-down.  None is needed for the
unfiltered derived comparison: \cref{thm:refinement-invariance} compares
arbitrary regularized representatives through the common interior.
\end{remark}

\subsection{The total logarithmic face complex}

For a resolved face $F$ and a coefficient complex $\cV$ we write
\[
 \cA_{\log}^\bullet(F;\cV)
 :=A_F^{\bullet,\log}\otimes\cV,
\]
where $A_F^{\bullet,\log}$ is the logarithmic de Rham complex of
Dupont--Panzer--Pym for the full basic positive log structure on the
ordinary-corner face $F$, containing the pullbacks of the unresolved basic
monomials \cite[Def.~6.4]{DPP}.  A subscript $c$ denotes compact
supports, and coefficients are omitted when trivial.  Thus the resolved
$b$/logarithmic route uses
the DPP complex itself; it does not rely on identifying a naively generated
free-logarithm algebra with it.  On an intrinsic generalized face we retain the notation
$\logOmega^\bullet(F)$ for the monoidal complex of
\cref{def:logalgebra}; comparison with the resolved complexes is part of the explicit
strictification data in \textup{(G4)}, not an automatic consequence of the DPP de
Rham theorem.
For a flat coefficient complex $\cV$ on an intrinsic face, we abbreviate
\[
 \logOmega^\bullet(F;\cV):=\logOmega^\bullet(F)\otimes\cV.
\]

Put $n=\dim X_P$.  Write $(\partial^kX_{\cR})^{\mathrm{bas}}$ for the basic log-corner object
on the ordered $k$-fold boundary.  We use the sign representation to remember orientations and set
\begin{equation}\label{eq:total-face-complex}
 \mathbb A_{\cR,s_{\cR}}^{m}(P)
 =\bigoplus_{k\geq0}
 \left(
 \cA_{\log,c}^{m-k}((\partial^kX_{\cR})^{\mathrm{bas}})
 \otimes\sgn_k
 \right)^{S_k}.
\end{equation}
If a component has logarithmic form degree $q$ and face degree $k$, its total degree is $q+k$.  The total differential is
\[
 D=\db+(-1)^qd_{\mathrm{face}},
\]
where $d_{\mathrm{face}}$ is the alternating sum of the regularized virtual pullbacks along ordered face maps.  The simplicial identities and functoriality of virtual pullback imply $D^2=0$.

We record explicitly how this convention is related to the shifted
totalization in Dupont--Panzer--Pym.  On a component of face degree $k$, the
DPP differential is
\begin{equation}\label{eq:dpp-total-differential}
 D_{\mathrm{DPP}}=(-1)^k\db+d_{\mathrm{face}}.
\end{equation}
This is the ordinary differential on the summands shifted by $[-k]$ in the
relative complex of \cite[\S6.4.3]{DPP}.
Define the diagonal sign operator on bidegree $(q,k)$ by
\begin{equation}\label{eq:dpp-column-conjugation}
 \Xi|_{(\cA_{\log,c}^{q}((\partial^kX_{\cR})^{\mathrm{bas}})
 \otimes\sgn_k)^{S_k}}
 =(-1)^{qk}\id.
\end{equation}
Since $\db$ commutes with every regularized virtual pullback, a direct
calculation gives
\begin{equation}\label{eq:dpp-column-intertwining}
 D\Xi=\Xi D_{\mathrm{DPP}}.
\end{equation}
Thus $\Xi$ is a cochain isomorphism from the DPP shifted totalization to the
column totalization used in this paper.

In the DPP convention the symmetrized relative trace is
$\sum_k(-1)^k/k!$ times integration on the ordered $k$-fold boundary
\cite[\S6.4.3 and the formula following Def.~7.6]{DPP}.  Transporting that
trace through $\Xi$ gives the trace compatible with $D$:
\begin{equation}\label{eq:derived-trace}
 \mathcal I_{\cR,s_{\cR}}(\omega)
 =\sum_{k=0}^{n}\frac{(-1)^{nk}}{k!}
 \int_{((\partial^kX_{\cR})^{\mathrm{bas}},s_{\cR}^{(k)})}\omega_k,
 \qquad \omega=(\omega_0,\omega_1,\ldots).
\end{equation}
Here only the top-form component on each $(n-k)$-dimensional ordered face is
integrated; the functional is zero on homogeneous total degrees different
from \(n\).  Thus \eqref{eq:derived-trace} is the cochain trace induced on the
full relative totalization, not a single absolute regularized bulk integral.
Indeed, on total degree $n$ one has $q=n-k$, and the transported coefficient is
\[
 (-1)^k(-1)^{qk}=(-1)^{k+(n-k)k}=(-1)^{nk}.
\]
The factor $1/k!$ compensates for ordered boundary components.  In particular,
the transported signs are all positive for even $n$ and reduce to $(-1)^k$ for
odd $n$.

\begin{proposition}[Totalized Stokes identity]\label{prop:total-stokes}
The functional $\mathcal I_{\cR,s_{\cR}}$ has degree $-n$ and satisfies
\[
 \mathcal I_{\cR,s_{\cR}}\circ D=0.
\]
Equivalently, the ordinary regularized Stokes contributions on the resolved faces, including exceptional faces, cancel in the complete total complex.
\end{proposition}

\begin{proof}
Let $\mathcal I_{\mathrm{DPP}}$ denote the DPP trace in its shifted
totalization convention.  The regularized Stokes theorem on the complete
ordered boundary diagram says
\[
 \mathcal I_{\mathrm{DPP}}D_{\mathrm{DPP}}=0
\]
\cite[Cor.~7.11 and the formula following Def.~7.6]{DPP}.  By construction,
$\mathcal I_{\cR,s_{\cR}}=\mathcal I_{\mathrm{DPP}}\Xi^{-1}$ on total
degree $n$.  Using \eqref{eq:dpp-column-intertwining},
\[
 \mathcal I_{\cR,s_{\cR}}D
 =\mathcal I_{\mathrm{DPP}}\Xi^{-1}D
 =\mathcal I_{\mathrm{DPP}}D_{\mathrm{DPP}}\Xi^{-1}=0.
\]
This incorporates the ordered-incidence orientation signs and the factorial
normalization, rather than introducing a second unrecorded sign in the
componentwise Stokes formula.
\end{proof}

\begin{remark}[Two-dimensional sign check]\label{rem:two-dimensional-trace-sign}
Let $n=2$ and place a one-form $\eta$ in face degree zero.  Then
$D\eta=(\db\eta,-d_{\mathrm{face}}\eta)$ in face degrees zero and one,
whereas \eqref{eq:derived-trace} has coefficient $+1$ in both degrees.
Consequently
\[
 \mathcal I(D\eta)
 =\int_X\db\eta-\int_{\partial X}\Reg\eta=0
\]
by regularized Stokes.  This is the first parity in which the diagonal
transport changes the untransported factor $(-1)^k$.
\end{remark}

\subsection{Invariance under stellar refinement}

A stellar refinement produces a proper blow-down
\[
 \gamma:X_{\cR'}\longrightarrow X_{\cR}
\]
which maps boundary to boundary and restricts to an orientation-preserving
diffeomorphism of interiors.  We use this interior identification rather than a
purported explicit deformation retraction of the radial blow-down.

\begin{lemma}[Relative cohomology and degree under blow-down]\label{lem:relative-degree-one}
Let $\gamma:N'\to N$ be a proper map of oriented $n$-manifolds with
ordinary corners such that $\gamma(\partial N')\subseteq\partial N$ and
\(\gamma^\circ:N'^\circ\to N^\circ\) is an orientation-preserving
diffeomorphism.  Then
\begin{equation}\label{eq:relative-cohomology-interior}
 \gamma^*:H_c^\bullet(N,\partial N;\R)
 \xrightarrow{\cong}H_c^\bullet(N',\partial N';\R)
\end{equation}
and
\begin{equation}\label{eq:relative-fundamental-class}
 \gamma_*[N',\partial N']_{\mathrm{BM}}
 =[N,\partial N]_{\mathrm{BM}}.
\end{equation}
If $N$ and $N'$ are compact, compact supports and the subscript
$\mathrm{BM}$ may be omitted.
In particular, both statements hold for finite compositions of stellar monoidal
blow-downs.
\end{lemma}

\begin{proof}
For an oriented manifold with corners there is a canonical, natural
identification
\[
 H_c^\bullet(N,\partial N;\R)\cong H_c^\bullet(N^\circ;\R),
\]
obtained by excision from a boundary collar; likewise for $N'$.  Under these
identifications, $\gamma^*$ is the pullback induced by the interior map
$\gamma^\circ$.  Since $\gamma^\circ$ is a diffeomorphism,
\eqref{eq:relative-cohomology-interior} follows.

Dually, the relative fundamental class corresponds to the Borel--Moore
fundamental class of the oriented interior,
\[
 [N,\partial N]\longleftrightarrow[N^\circ]_{\mathrm{BM}}.
\]
Proper pushforward is natural for this identification, and the
orientation-preserving diffeomorphism $\gamma^\circ$ sends
$[N'^\circ]_{\mathrm{BM}}$ to $[N^\circ]_{\mathrm{BM}}$.  This gives
\eqref{eq:relative-fundamental-class}.  A stellar monoidal blow-down has exactly
these properties, and so does a finite composition.
\end{proof}

\subsection{The interior bridge and refinement roofs}\label{subsec:interior-bridge}

The total relative logarithmic complex is functorial for morphisms of the
symmetric semi-simplicial boundary diagram.  A monoidal blow-down need not
induce such a morphism degree by degree: an exceptional hypersurface of the
fine resolution may map into a coarse face of higher codimension.  We
therefore do \emph{not} use a symbol $\gamma^*$ for a chain map between the
two total face complexes unless an additional filtered refinement transfer
has been specified.

For an oriented manifold with ordinary corners $N$, put
\[
 \mathbb K_c^\bullet(N):=\Omega_c^\bullet(N^\circ).
\]
Extension by zero across the boundary, followed by the inclusion of smooth
forms into logarithmic forms, gives a canonical cochain map
\begin{equation}\label{eq:interior-bridge-map}
 \epsilon_{N,s}:\mathbb K_c^\bullet(N)
 \longrightarrow \mathbb A_{N,s}^\bullet,
\end{equation}
where the image is placed in face degree zero.  The notation retains the
regularization $s$ only to remember the target complex; the map itself is
independent of $s$.

\begin{lemma}[Interior bridge]\label{lem:interior-bridge}
For every oriented manifold with ordinary corners $N$, with finitely many
embedded boundary faces and endowed with a DPP regularization $s$, the
compactly supported map \eqref{eq:interior-bridge-map} is a
quasi-isomorphism.  Moreover,
\begin{equation}\label{eq:interior-trace-bridge}
 \mathcal I_{N,s}\circ\epsilon_{N,s}
 =\int_{N^\circ}
 \qquad\text{on }\Omega_c^{\dim N}(N^\circ).
\end{equation}
\end{lemma}

\begin{proof}
The compactly supported smooth relative de Rham complex of the pair
$(N,\partial N)$ is quasi-isomorphic to $\Omega_c^\bullet(N^\circ)$ by the
standard collar/excision model for relative cohomology.  Dupont--Panzer--Pym identify the relative
logarithmic complex, defined by the symmetric semi-simplicial boundary
totalization, with the same relative cohomology and prove that the inclusion
of compactly supported smooth relative forms into compactly supported
logarithmic relative forms is a
quasi-isomorphism.  The composite of these two canonical maps is precisely
$\epsilon_{N,s}$, hence it is a quasi-isomorphism.  For a form compactly
supported in the interior there are no boundary components and the
regularized integral agrees with the ordinary absolutely convergent integral,
which gives \eqref{eq:interior-trace-bridge}.
\end{proof}

Let now $\gamma:N'\to N$ be a stellar monoidal blow-down.  Its restriction to
the interiors is an orientation-preserving diffeomorphism, hence
\[
 (\gamma^\circ)^*:\mathbb K_c^\bullet(N)
 \xrightarrow{\cong}\mathbb K_c^\bullet(N').
\]

\begin{theorem}[Derived refinement roof]\label{thm:refinement-invariance}
Let $\cR'\to\cR$ be a finite stellar refinement and choose arbitrary DPP
regularizations $s_{\cR}$ and $s_{\cR'}$ on the two smooth resolutions.  Then
the diagram
\begin{equation}\label{eq:refinement-roof}
\begin{tikzcd}[column sep=large]
 & \Omega_c^\bullet(X_{\cR}^{\circ})
   \arrow[dl,"\epsilon_{\cR,s_{\cR}}"',"\simeq"]
   \arrow[rr,"(\gamma^\circ)^*","\cong"']
 && \Omega_c^\bullet(X_{\cR'}^{\circ})
   \arrow[dr,"\epsilon_{\cR',s_{\cR'}}","\simeq"'] & \\
 \mathbb A_{\cR,s_{\cR}}^\bullet(P) &&&&
 \mathbb A_{\cR',s_{\cR'}}^\bullet(P)
\end{tikzcd}
\end{equation}
defines a canonical isomorphism between the two total relative logarithmic
complexes in the derived category.  Under this isomorphism their trace
morphisms agree.  Equivalently, both traces correspond to ordinary integration
on the canonically identified interior.

No chain map preserving the face-degree decomposition is asserted by this
theorem.  Such a filtered comparison requires an additional refinement
transfer $U_\gamma$ as in \cref{def:refinement-descent-datum}.
\end{theorem}

\begin{proof}
The two diagonal maps in \eqref{eq:refinement-roof} are quasi-isomorphisms by
\cref{lem:interior-bridge}, while the middle map is an isomorphism because the
blow-down is an orientation-preserving diffeomorphism on interiors.  Thus the
roof defines an isomorphism in the derived category.  Trace compatibility is
immediate from \eqref{eq:interior-trace-bridge} and change of variables for
$\gamma^\circ$.
\end{proof}

\begin{remark}[Filtered versus unfiltered refinement]\label{rem:filtered-vs-derived-refinement}
The distinction in \cref{thm:refinement-invariance} is essential.  The roof
\eqref{eq:refinement-roof} proves resolution independence of the \emph{total}
derived complex and of its trace, but it forgets the face filtration.  If one
wants a refinement map that retains the stratified face-by-face information,
one must specify a degree-mixing subdivision/aggregation operator $U_\gamma$.
Locally for an elementary stellar blow-up this operator is expected to be
built from the Whitney complex of the normal simplex: a fine face of smaller
codimension is accompanied by a Whitney form of complementary degree so that
total degree is preserved.  The abstract comparison theorem of
\cref{thm:derived-comparison} applies once those filtered data are supplied.
\end{remark}

By assumption \textup{(G3)}, any two admitted resolutions have a common smooth
refinement.  Applying \cref{thm:refinement-invariance} to the two arms gives a
trace-compatible derived roof between them.  Notice that the roof itself uses
only the common interior; \textup{(G3)} is needed elsewhere in the paper when
a filtered or stellar comparison is required.

\begin{definition}[Derived monoidal trace]\label{def:derived-trace}
The roofs of \cref{thm:refinement-invariance} identify all total relative
logarithmic complexes of admitted smooth resolutions in the localization of
cochain complexes at quasi-isomorphisms.  Their common isomorphism class is
denoted
\[
 \mathbb A_{P}^{\bullet,\mathrm{der}}.
\]
For a chosen regularized representative $(\cR,s_{\cR})$, the cochain trace
$\mathcal I_{\cR,s_{\cR}}$ determines a morphism
\[
 \mathcal I_{P}^{\mathrm{der}}:
 \mathbb A_{P}^{\bullet,\mathrm{der}}\longrightarrow\R[-n],
\]
and \cref{lem:interior-bridge} shows that this derived morphism is independent
of both the smooth resolution and the chain-level regularization: under the
interior bridge it is the ordinary orientation pairing on
$H_c^n(X_P^\circ)$.  A scale or DPP regularization still matters for the
chosen chain-level logarithmic representative and for its explicit boundary
terms.
\end{definition}

\begin{corollary}[Derived affine Stokes theorem]\label{cor:derived-affine-stokes}
The morphism $\mathcal I_{P}^{\mathrm{der}}$ is resolution independent.  On
every chosen smooth regularized representative one has the exact chain-level
identity
\[
 \mathcal I_{\cR,s_{\cR}}\circ D=0.
\]
For a monoidal logarithmic form $\eta$ represented on that resolution by
$\beta_{\cR}^*\eta$ in face degree zero,
\[
 \mathcal I_{\cR,s_{\cR}}\bigl(D\beta_{\cR}^*\eta\bigr)=0.
\]
The face-degree-one term on a fixed resolution is the full sum over its
boundary hypersurfaces, including exceptional ones.  Without an additional
filtered refinement transfer $U_\gamma$, however, the individual face-degree
components on two different resolutions are not canonically identified; what
is canonical under \cref{thm:refinement-invariance} is the total derived
complex and trace.
\end{corollary}

\subsection{Derived \texorpdfstring{$b$}{b}/logarithmic AKSZ transgression}

Let $(\Y,\omega_\Y=\delta\alpha_\Y,\Theta)$ be an exact $QP$-target.  Pull the
universal $b$/logarithmic AKSZ integrands to a chosen smooth regularized
resolution and regard their face descendants as elements of its total
complex.  The source here is $\bT[1]X_\cR$ with its DPP logarithmic trace
model transported to the column convention by \eqref{eq:dpp-column-conjugation},
not the ordinary source $T[1]X_\cR$ of
\cref{sec:resolution-aksz-data}.

We make the representative used below explicit.  Under the standing formal
mapping-space hypothesis, a compatible resolved field is a family of maps
from the $b$-sources of the ordered resolved faces to $\Y$, compatible with
the DPP virtual face maps.  For every connected ordered codimension-$k$ lift
$\widetilde H\subset(\partial^kX_\cR)^{\mathrm{bas}}$ put
\begin{align}
 \alpha_{\widetilde H,s_\cR}
 &=\int_{(\widetilde H,s_\cR|_{\widetilde H})}
   \ev^*\alpha_\Y,\label{eq:resolved-aksz-alpha}\\
 \omega_{\widetilde H,s_\cR}
 &=\int_{(\widetilde H,s_\cR|_{\widetilde H})}
   \ev^*\omega_\Y,\label{eq:resolved-aksz-omega}\\
 S_{\widetilde H,s_\cR}
 &=\int_{(\widetilde H,s_\cR|_{\widetilde H})}
   \bigl(\iota_{\db}\ev^*\alpha_\Y+\ev^*\Theta\bigr).
   \label{eq:resolved-aksz-action}
\end{align}
Only the component of the appropriate source degree is integrated.  These
expressions form equivariant families under permutation of the ordered lifts.
The factorial and sign in \eqref{eq:derived-trace} enter only when such a
family is paired with the \emph{total} relative trace; they are not inserted
into an individual face transgression.  Write
\(\omega_{\cR,s_\cR}=\omega_{X_\cR,s_\cR}\) and
\(S_{\cR,s_\cR}=S_{X_\cR,s_\cR}\) for the bulk terms.  Finally, define
\(\alpha_{\partial,\cR,s_\cR}\) to be the face-degree-one incidence term
obtained by applying the column differential $D$ to the universal primitive;
equivalently, it is the oriented pullback sum of the terms
\eqref{eq:resolved-aksz-alpha} with $k=1$, with the sign transported by $\Xi$.
This definition is on one chosen resolved field diagram and makes no
cross-resolution identification of mapping spaces.

\begin{theorem}[Resolved representative of the derived BV--BFV identity]\label{thm:derived-bvbfv}
Let $(\cR,s_{\cR})$ be a smooth regularized representative.  Regularized
transgression on its total logarithmic face complex satisfies
\[
 \iota_Q\omega_{\cR,s_{\cR}}
 =(-1)^n\delta S_{\cR,s_{\cR}}
 +\alpha_{\partial,\cR,s_{\cR}},
\]
where $\alpha_{\partial,\cR,s_{\cR}}$ is the face-degree-one component of the
total differential of the transgressed primitive just defined; its next face
differential vanishes by $D^2=0$.  The roofs of
\cref{thm:refinement-invariance} identify
the underlying linear total source complexes and trace morphisms.  Therefore,
if the resulting compactly supported total source cocycles represent
corresponding classes under the common-interior roof, their integrated values
agree.  This applies in particular to the interior-bridge representatives of
compactly supported forms supported away from the boundary.  Agreement of
fields on the common interior, or membership in the basic bulk sector, does
not by itself imply the required correspondence of total cocycle classes.
The roof by itself does not identify nonlinear mapping spaces on two
resolutions.

A canonical identification of the separate BFV descendants, or of full
nonlinear field spaces, on two different resolutions is asserted only when
additional filtered/nonlinear comparison data are supplied.
\end{theorem}

\begin{proof}
Apply the calculation of \cref{app:signs} to the bulk ordered face, with its
DPP regularized trace.  The source part gives the factor $(-1)^n$ and the
face-degree-one term defined above; the target part uses
$\iota_{Q_\Y}\omega_\Y=\delta\Theta$.  This proves the displayed identity.
On the complete ordered diagram, \cref{prop:total-stokes} makes the same
calculation compatible with the transported column differential.  The
codimension-two and higher identities are the face-degree components of
$D^2=0$.  The unfiltered
resolution-independence statement for corresponding source-cocycle classes
and their trace follows from \cref{thm:refinement-invariance} and
\cref{lem:interior-bridge}.  The correspondence hypothesis cannot be replaced
by equality on the interior: already the regularized integral of $dr/r$ on an
interval depends on the chosen boundary scale
\cite[Ex.~7.13]{DPP}.  The field-space qualification is necessary because the
roofs are maps of source complexes, not maps of nonlinear mapping spaces.
\end{proof}

\subsection{Strictification}

A smooth resolution does not in general contain one distinguished resolved face
above each intrinsic face.  Consequently strictification cannot be formulated
by deleting the geometrically exceptional faces and retaining ``the'' strict
transforms of the original ones.  The correct object is a chain-level
coarse/reduced splitting associated with subdivision and aggregation.

Set, at the graded level,
\[
 \mathbb B_P^m
 =\bigoplus_{F\in\mathsf{Face}(X_P)}
   {\logOmega}_c^{m-|F|}(F),
\]
fix a regularized representative $\mathfrak s=(\cR,s_{\cR})$, and let
\(\mathbb A_{\cR,s_{\cR}}^\bullet(P)\) be its resolved total face complex.
A \emph{subdivision/aggregation datum} consists of an auxiliary differential
\(D_0\) on the resolved graded module, a differential \(d_{\mathbb B}\) on
\(\mathbb B_P^\bullet\), and graded maps
\[
 i_0:(\mathbb B_P^\bullet,d_{\mathbb B})\longrightarrow
       (\mathbb A_{\cR,s_{\cR}}^\bullet(P),D_0),\qquad
 p_0:(\mathbb A_{\cR,s_{\cR}}^\bullet(P),D_0)\longrightarrow
       (\mathbb B_P^\bullet,d_{\mathbb B}),
 \qquad p_0i_0=\id,
\]
which are cochain maps, together with a degree \(-1\) homotopy \(h_0\)
satisfying
\begin{equation}\label{eq:affine-initial-contraction}
 D_0h_0+h_0D_0=\id-i_0p_0,
 \qquad h_0i_0=0,\quad p_0h_0=0,\quad h_0^2=0.
\end{equation}
Thus these data form a strong deformation retract before the face-incidence
perturbation is turned on.  The kernel of \(p_0\) is the \emph{initial reduced
refinement summand}.  It need not be preserved by the full resolved
differential.

\begin{proposition}[Linear strictification criterion]\label{prop:strictification}
Assume a subdivision/aggregation datum as above for which
\[
 D=D_0+\delta,
\]
where $i_0,p_0$ have face degree zero in the face-regraded normal-simplex
splitting, $d_{\mathbb B}$ has face degree zero, $D_0$ preserves the intrinsic
support filtration (and may contain the normal-fibre cellular differential),
\(\delta\) has face degree \(+1\), \(h_0\) has face degree \(-1\), and
the perturbation series in \(h_0\delta\) and \(\delta h_0\) are finite (for
example by a finite filtration).  Then the basic perturbation lemma~\cite{GLS},
\cite[Sec.~2]{CrainicPerturb} gives a
strong deformation retract
\[
 (\mathbb A_{\cR,s_{\cR}}^\bullet(P),D)
 \underset{i_\infty}{\overset{p_\infty}{\rightleftarrows}}
 (\mathbb B_P^\bullet,D_{\rm str})
\]
with
\begin{align}
 i_\infty&=(1+h_0\delta)^{-1}i_0,\\
 p_\infty&=p_0(1+\delta h_0)^{-1},\\
 h_\infty&=h_0(1+\delta h_0)^{-1},\\
 D_{\rm str}&=d_{\mathbb B}+p_0\delta(1+h_0\delta)^{-1}i_0.
\end{align}
The perturbed data retain the special side conditions
\[
 p_\infty i_\infty=\id,\qquad
 h_\infty i_\infty=0,\qquad
 p_\infty h_\infty=0,\qquad h_\infty^2=0.
\]
The transferred differential has only face degrees zero and one.  Write
$\jmath_F$ for inclusion of the intrinsic summand indexed by $F$ and $\pi_G$
for projection onto the summand indexed by $G$.  If, in addition, the datum is
\emph{incidence-local} in the sense that
\begin{equation}\label{eq:affine-incidence-locality}
 \pi_Gp_0\delta(h_0\delta)^r i_0\jmath_F=0
 \quad\text{for every }r\geq0\text{ unless }G\prec F,
\end{equation}
then the face-degree-one component is a strict incidence operator on
\(\mathsf{Face}(X_P)\).  In all cases it squares to zero together with the
internal differential, and the transferred functional
\[
 \mathcal I_{P,\mathfrak s}^{\rm str}
 =\mathcal I_{\cR,s_{\cR}}\circ i_\infty
\]
is a cochain map.  The full-differential reduced sector is
\(\ker p_\infty\); it is a \(D\)-subcomplex and is contracted by the
restriction of \(h_\infty\).
\end{proposition}

\begin{proof}
Apply the basic perturbation lemma to
\eqref{eq:affine-initial-contraction}.  Standard formulas are often written
with $D_0h_0+h_0D_0=i_0p_0-\id$; replacing that homotopy by $-h_0$ to match
our convention $D_0h_0+h_0D_0=\id-i_0p_0$ changes the resolvents from
$(1-h_0\delta)^{-1}$ to $(1+h_0\delta)^{-1}$, and similarly on the other
side.  Its identities give
\[
 Dh_\infty+h_\infty D=\id-i_\infty p_\infty,
 \qquad
 Di_\infty=i_\infty D_{\rm str},
 \qquad
 p_\infty D=D_{\rm str}p_\infty.
\]
The special side conditions displayed in the statement are the normalized
form of the basic perturbation lemma and follow from
\(h_0i_0=p_0h_0=h_0^2=0\).  In particular,
\(p_\infty h_\infty=0\), so \(h_\infty\) preserves
\(\ker p_\infty\).
Since \(\delta\) raises face degree by one while \(h_0\) lowers it by one,
every term
\[
 (-1)^r p_0\delta(h_0\delta)^ri_0
\]
has face degree exactly \(+1\).  Thus no higher face-degree operation appears
in the transferred \emph{linear differential}.  Under
\eqref{eq:affine-incidence-locality}, its matrix component from $F$ to $G$
vanishes unless $G\prec F$, which is the additional support statement needed
to call it an incidence operator.  Finally,
\(\mathcal I_{\cR,s_{\cR}}D=0\) and the cochain-map identity for
\(i_\infty\) imply
\(\mathcal I_{P,\mathfrak s}^{\rm str}D_{\rm str}=0\).  On
\(\ker p_\infty\), the displayed homotopy identity reduces to
\(Dh_\infty+h_\infty D=\id\).
\end{proof}

\begin{remark}[Linear versus multiplicative strictification]
\Cref{prop:strictification} is a statement about facewise cochain complexes
and traces.  Homological transfer does not by itself make the transferred
face maps morphisms of differential graded \emph{algebras}.  Therefore this
criterion is sufficient for linear theories such as abelian BF.  A nonlinear
AKSZ theory requires additional multiplicative or, more generally,
\(A_\infty/L_\infty\) compatibility; the higher operations created by such a
transfer are treated separately below and are not silently discarded.
\end{remark}

This formulation is the chain-level replacement for the naive idea of
``contracting the exceptional faces''.  In a stellar model the initial
reduced summand is locally the reduced cochain complex of the normal simplex;
after perturbation the actual contractible subcomplex is \(\ker p_\infty\),
not necessarily the span of geometrically exceptional faces.

\section{Sheafification and descent on Joyce atlases}\label{sec:global-descent}

The affine construction is local in the base, but its definition as a resolution-independent quasi-isomorphism class makes gluing less immediate than for ordinary differential forms.  In this section we globalize it without choosing compatible affine resolutions chart by chart.  The key device is to resolve the whole Joyce manifold once and then work with sheaves pushed forward to the unresolved space.

\subsection{Resolved regularization data}

Throughout this section let $X$ satisfy \textup{(G1)}--\textup{(G3)}.  In particular, $X$ is paracompact with embedded faces and carries the face-orientation datum of \textup{(G1)}.  Choose the fixed smooth refinement of Kottke's monoidal complex $\cP_X$ supplied by \textup{(G2)}.  Kottke's blow-up construction realizes this refinement by a blow-down
\cite{KottkeGCorner}
\[
 \beta_{\cR}:X_{\cR}\longrightarrow X
\]
from a manifold with ordinary corners; by \cref{prop:proper-blowdown}, this
blow-down is proper.  A further refinement $\cR'\to\cR$ gives a commutative triangle
\[
\begin{tikzcd}[column sep=large]
 X_{\cR'} \arrow[rr,"\gamma"] \arrow[dr,"\beta_{\cR'}"'] &&
 X_{\cR} \arrow[dl,"\beta_{\cR}"]\\
 &X.&
\end{tikzcd}
\]

\begin{definition}[Resolved chain-level regularization datum]\label{def:resolved-reg}
A resolved chain-level regularization datum on $X$ is a pair
$(\cR,s_{\cR})$, where $\cR$ is a smooth global refinement and
$s_{\cR}$ is a Dupont--Panzer--Pym regularization on the complete ordered
boundary diagram of $X_{\cR}$.  We write $\mathfrak s=(\cR,s_{\cR})$ when a
specific chain-level representative has been fixed.
\end{definition}

Under \textup{(G2)}, one such pair is part of the admissible package.  The
underlying derived object and its trace morphism will be independent of this
choice by the interior-bridge argument below.  The chain-level logarithmic
representative, its explicit regularized restrictions, and its facewise
boundary terms may still depend on the chosen regularization.

\subsection{The resolved face complex as a sheaf}

For a representative $(\cR,s_{\cR})$ and an open set $U\subseteq X$, define
\begin{equation}\label{eq:resolved-sheaf}
 \mathscr A_{\cR,s_{\cR}}^{m}(U)
 =\bigoplus_{k\geq0}
 \left(
 \cA_{\log}^{m-k}
 \bigl((\partial^k\beta_{\cR}^{-1}(U))^{\mathrm{bas}}\bigr)
 \otimes\sgn_k
 \right)^{S_k},
\end{equation}
with the same total differential $D=\db+(-1)^qd_{\mathrm{face}}$ as in
\cref{eq:total-face-complex}.  Compact support is not imposed here.

Let $j:X^\circ\hookrightarrow X$ be the inclusion of the common interior and
put
\begin{equation}\label{eq:interior-bridge-sheaf}
 \mathscr J_X^\bullet:=j_!\Omega_{X^\circ}^\bullet.
\end{equation}
Since every monoidal resolution is canonically the identity over $X^\circ$,
extension by zero to face degree zero gives a natural map of sheaf complexes
\begin{equation}\label{eq:sheaf-bridge-map}
 \epsilon_{\cR}:\mathscr J_X^\bullet
 \longrightarrow \mathscr A_{\cR,s_{\cR}}^\bullet.
\end{equation}

\begin{lemma}\label{lem:resolved-fine}
The sheaf $\mathscr A_{\cR,s_{\cR}}^\bullet$ is a bounded complex of fine
$C_X^\infty$-modules.  Since \(X\) is locally compact and paracompact, its
terms are soft and \(c\)-soft.
\end{lemma}

\begin{proof}
The sheaf axiom is the ordinary sheaf axiom for logarithmic forms on each
ordered face, followed by finite direct sums and finite-group invariants.
Multiplication by $f\in C_X^\infty(U)$ acts by multiplication with
$\beta_{\cR}^*f$ on every resolved face.  A smooth partition of unity on the
paracompact $g$-corner $X$ therefore acts on the complex, proving fineness.
Fine sheaves of modules over the soft sheaf \(C_X^\infty\) are soft; on a
locally compact paracompact space they are \(c\)-soft
\cite[Chap.~II]{BredonSheaf}.  The de Rham and face degrees are both bounded by
\(\dim X\).
\end{proof}

\begin{lemma}[Boundary collapse under monoidal blow-down]\label{lem:boundary-collapse-blowdown}
Let $\beta_{\cR}:X_{\cR}\to X$ be the blow-down of a smooth finite
refinement.  For an open set $U\subseteq X$ put
\[
 W_U=\beta_{\cR}^{-1}(U),\qquad
 B_U=U\cap(X\setminus X^\circ),\qquad
 A_U=W_U\cap\partial X_{\cR}.
\]
Then
\begin{equation}\label{eq:boundary-exact-preimage}
 A_U=\beta_{\cR}^{-1}(B_U),
\end{equation}
and the restricted blow-down induces a canonical homeomorphism of quotient
spaces
\begin{equation}\label{eq:boundary-collapse-homeomorphism}
 \overline\beta_U:W_U/A_U\xrightarrow{\ \cong\ }U/B_U.
\end{equation}
For the sufficiently small monoidal neighbourhoods used below, the two pairs
are good pairs, and hence
\begin{equation}\label{eq:local-pair-comparison}
 H^\bullet(W_U,A_U;\R)\cong H^\bullet(U,B_U;\R).
\end{equation}
No description, contractibility, or acyclicity of the individual fibres of
$\beta_{\cR}$ is required.
\end{lemma}

\begin{proof}
Equation \eqref{eq:boundary-exact-preimage} is
\eqref{eq:blowdown-interior-preimage} restricted to $U$.  By
\cref{prop:proper-blowdown}, the map $W_U\to U$ is a proper surjection and is
a diffeomorphism from $W_U\setminus A_U$ onto $U\setminus B_U$.  A proper
surjection between locally compact Hausdorff spaces is a quotient map.
Let
\[
 q_W:W_U\to W_U/A_U,
 \qquad
 q_U:U\to U/B_U
\]
be the collapse maps.  Since the blow-down is one-to-one away from $A_U$ and
maps $A_U$ onto $B_U$, there is a unique continuous bijection
$\overline\beta_U$ satisfying
$q_U\beta_{\cR}=\overline\beta_U q_W$.

It remains only to check the topology.  For a subset $V\subseteq U/B_U$,
\[
 q_U^{-1}(V)\text{ is open in }U
 \Longleftrightarrow
 \beta_{\cR}^{-1}q_U^{-1}(V)\text{ is open in }W_U
 \Longleftrightarrow
 q_W^{-1}\overline\beta_U^{-1}(V)\text{ is open in }W_U.
\]
The first equivalence uses that $W_U\to U$ is a quotient map and the second
is the defining commutative square.  Since $q_U$ and $q_W$ are quotient
maps, this is exactly the statement that $V$ is open if and only if
$\overline\beta_U^{-1}(V)$ is open.  Thus
\eqref{eq:boundary-collapse-homeomorphism} is a homeomorphism.

For a sufficiently small monoidal chart, choose finitely many generators
of the local model monoid.  In the resulting embedding into a finite
orthant, $(U,B_U)$ is a semialgebraic pair; after shrinking within the
cofinal family it admits a compatible triangulation.  On the resolved side,
$(W_U,A_U)$ is a manifold-with-corners pair and the boundary admits ordinary
collars.  Hence both are good pairs.  Relative singular cohomology is then
the reduced cohomology of the corresponding quotient, so
\eqref{eq:boundary-collapse-homeomorphism} gives
\eqref{eq:local-pair-comparison}.
\end{proof}

\begin{lemma}[Boundary-stalk acyclicity]\label{lem:boundary-stalk-acyclic}
Let $x\in X\setminus X^\circ$ and put
$W_U=\beta_{\cR}^{-1}(U)$ for an open neighbourhood $U$ of $x$.  There is a
cofinal system of sufficiently small monoidal neighbourhoods $U$ for which
\begin{equation}\label{eq:local-resolved-relative-acyclicity}
 H^\bullet(W_U,\partial_{\mathrm{res}}W_U;\R)=0,
\end{equation}
where $\partial_{\mathrm{res}}W_U$ is the union of the ordinary boundary
hypersurfaces of the resolved neighbourhood.  Consequently
\begin{equation}\label{eq:resolved-boundary-stalk-zero}
 \mathcal H^q(\mathscr A_{\cR,s_{\cR}}^\bullet)_x=0
 \qquad\text{for every }q.
\end{equation}
\end{lemma}

\begin{proof}
Choose a monoidal chart centred on the stratum through $x$, and choose a
cofinal family of sufficiently small relatively compact neighbourhoods $U$
of the following product type.  In the tangential directions take a small
ball; in the sharp normal model take simultaneous sublevel sets for a finite
generating set of the normal monoid, chosen so that decreasing every normal
monomial preserves the set.  Put
\[
 B_U=U\cap(X\setminus X^\circ),\qquad
 W_U=\beta_{\cR}^{-1}(U).
\]
By \cref{lem:boundary-collapse-blowdown},
$\partial_{\mathrm{res}}W_U=\beta_{\cR}^{-1}(B_U)$ and
\begin{equation}\label{eq:local-pair-comparison-stalk}
 H^\bullet(W_U,\partial_{\mathrm{res}}W_U;\R)
 \cong H^\bullet(U,B_U;\R).
\end{equation}

We now contract the unresolved pair, where no lift to the resolution is
needed.  Let $P$ be the sharp normal monoid and choose an integral functional
$\lambda:P\to\N$ strictly positive on $P\setminus\{0\}$.  On the local
monoidal model the action
\[
 h_t(z)(p)=t^{\lambda(p)}z(p),\qquad 0<t\leq1,
\]
extends continuously at $t=0$ to the normal vertex.  By the choice of the
normal sublevel sets it preserves $U$ and $B_U$.  Thus $U$ deformation
retracts onto its tangential slice at the normal vertex, while $B_U$ retracts
onto the same slice.  Since $x$ is a boundary point, this slice is contained
in $B_U$.  Consequently the inclusion $B_U\hookrightarrow U$ is a homotopy
equivalence on this cofinal family and
\[
 H^\bullet(U,B_U;\R)=0.
\]
Together with \eqref{eq:local-pair-comparison-stalk}, this proves
\eqref{eq:local-resolved-relative-acyclicity}.  Notice that the argument does
\emph{not} attempt to lift the weighted normal contraction through an
arbitrary fan refinement; the quotient comparison of
\cref{lem:boundary-collapse-blowdown} is precisely what avoids that
unnecessary assertion.

For each such $U$, the DPP logarithmic relative de Rham theorem
\cite[Prop.~6.15]{DPP} identifies the cohomology of the symmetric
semi-simplicial totalization in \eqref{eq:resolved-sheaf} with
$H^\bullet(W_U,\partial_{\mathrm{res}}W_U;\R)$.  The cohomology-sheaf stalk
is the filtered colimit over these $U$, and filtered colimits of real vector
spaces are exact.  Hence \eqref{eq:resolved-boundary-stalk-zero} follows from
\eqref{eq:local-resolved-relative-acyclicity}.
\end{proof}

\begin{proposition}[Interior bridge for resolved sheaves]\label{prop:sheaf-refinement}
For every smooth regularized resolution, the map
\eqref{eq:sheaf-bridge-map} is a quasi-isomorphism of sheaf complexes.  Hence
for two resolutions $(\cR_1,s_1)$ and $(\cR_2,s_2)$ there is a canonical roof
in $D(\operatorname{Sh}(X))$,
\begin{equation}\label{eq:sheaf-refinement-roof}
 \mathscr A_{\cR_1,s_1}^\bullet
 \xleftarrow{\ \simeq\ }\mathscr J_X^\bullet
 \xrightarrow{\ \simeq\ }\mathscr A_{\cR_2,s_2}^\bullet.
\end{equation}
No direct face-degree-preserving pullback between the two resolved sheaves is
required.
\end{proposition}

\begin{proof}
Check stalks.  At an interior point, a sufficiently small open set meets no
boundary and \eqref{eq:sheaf-bridge-map} is the ordinary de Rham inclusion.
At a boundary point $x$, the stalk of $j_!\Omega_{X^\circ}^\bullet$ is zero,
while \cref{lem:boundary-stalk-acyclic} shows that every cohomology stalk of
$\mathscr A_{\cR,s_{\cR}}^\bullet$ is zero.  Hence the map is a
quasi-isomorphism at every stalk.  Notice that the boundary argument uses the
entire resolved preimage of a shrinking neighbourhood, including its
exceptional fibre, rather than a single corner neighbourhood in
$X_{\cR}$.
\end{proof}

\begin{definition}[Global derived logarithmic face sheaf]\label{def:global-derived-sheaf}
For any chosen resolved chain-level regularization datum $\mathfrak s$, the global derived logarithmic face sheaf
\[
 \mathscr A_{X}^{\bullet,\mathrm{der}}
 \in D\bigl(\operatorname{Sh}(X)\bigr)
\]
is the quasi-isomorphism class represented by any resolved face complex $\mathscr A_{\cR,s_{\cR}}^\bullet$ in the admissible comparison class.
\end{definition}

\Cref{prop:sheaf-refinement} identifies every resolved representative with the same interior bridge $\mathscr J_X^\bullet$ in the derived category.  Thus the object is independent of the resolution without requiring a degreewise map of boundary diagrams.  All descent arguments may be performed on one bounded fine resolved representative.

\begin{remark}[What the unfiltered object remembers]\label{rem:unfiltered-derived-scope}
There is, by construction, a canonical identification
\[
 \mathscr A_X^{\bullet,\mathrm{der}}
 \simeq j_!\Omega_{X^\circ}^\bullet
 \quad\text{in }D(\operatorname{Sh}(X)).
\]
Thus resolution independence at this level is a relative de Rham statement.
Its useful additional datum is the resolved logarithmic representative and
the trace morphism carried through the roof.  The object itself forgets the
face filtration, individual exceptional components, separate BFV descendants,
and nonlinear mapping-space data.  Retaining such information requires the
filtered transfer of \textup{(G4)}, the separate cyclic field datum
\textup{(G5)}, or a multiplicative refinement beyond this unfiltered roof.
\end{remark}

\subsection{\v{C}ech descent and chart changes}

\begin{theorem}[Sheaf descent]\label{thm:sheaf-descent}
Let $\{U_i\}_{i\in I}$ be a locally finite open cover of $X$ and let
$\mathscr A^\bullet$ be any bounded fine representative of
$\mathscr A_X^{\bullet,\mathrm{der}}$.  Then the augmentation
\[
 \Gamma(X,\mathscr A^\bullet)
 \xrightarrow{\simeq}
 \Tot\check C^\bullet\bigl(\{U_i\},\mathscr A^\bullet\bigr)
\]
is a quasi-isomorphism.  Equivalently, the \v{C}ech totalization computes
$R\Gamma(X,\mathscr A_X^{\bullet,\mathrm{der}})$.  The analogous statement
with compact supports holds for a finite cover adapted to the support.
\end{theorem}

\begin{proof}
For each fixed total degree, the augmented \v{C}ech complex of the fine
representative $\mathscr A^\bullet$ is exact, with contracting homotopy from a
partition of unity subordinate to the cover.  The face degree is bounded by
$\dim X$, so the resulting double complex is bounded in one direction.  The
spectral sequence in \v{C}ech degree therefore identifies its totalization
with $\Gamma(X,\mathscr A^\bullet)$, and hence computes the derived global
sections of the derived object.  For compact supports, use the \(c\)-soft
terms from \cref{lem:resolved-fine}; their compactly supported sections are
acyclic, and the same bounded double-complex argument applies to a finite
subcover of the support.
\end{proof}

For an affine chart $V\subseteq X_P$ with a chosen smooth regularized representative, write
$\mathscr A_{P}^{\bullet,\mathrm{der}}|_V$ for the derived sheaf represented by its resolved total-face complex.  By \cref{prop:sheaf-refinement}, this object is canonically identified with the restriction of the common interior bridge and is therefore independent of the chosen resolution and chain-level DPP regularization.

Let $(P_i,V_i,\phi_i)$ be a Joyce atlas~\cite{JoyceGCorner}.  Write $U_i=\phi_i(V_i)$.  Restriction of a global resolution to $U_i$ is a smooth resolution of the local monoidal model.  Hence \cref{prop:sheaf-refinement} and the affine derived construction give a natural local comparison
\begin{equation}\label{eq:chart-comparison}
 \mathscr A_{X}^{\bullet,\mathrm{der}}|_{U_i}
 \xrightarrow{\simeq}
 (\phi_i)_*\mathscr A_{P_i}^{\bullet,\mathrm{der}}|_{V_i}.
\end{equation}
The comparison factors through the restriction of the common interior bridge;
no pullback between the two ordered-boundary diagrams is used.

\begin{corollary}[Descent under Joyce chart changes]\label{cor:chart-descent}
On an overlap $U_{ij}=U_i\cap U_j$, the two affine derived complexes are joined by a canonical zigzag of trace-compatible quasi-isomorphisms
\[
 (\phi_i)_*\mathscr A_{P_i}^{\bullet,\mathrm{der}}|_{U_{ij}}
 \xleftarrow{\simeq}
 \mathscr A_{X}^{\bullet,\mathrm{der}}|_{U_{ij}}
 \xrightarrow{\simeq}
 (\phi_j)_*\mathscr A_{P_j}^{\bullet,\mathrm{der}}|_{U_{ij}}.
\]
On triple overlaps these zigzags satisfy the cocycle condition in the derived category, and all higher coherence maps are supplied by the single global sheaf $\mathscr A_{X}^{\bullet,\mathrm{der}}$.
\end{corollary}

\begin{proof}
Use \cref{eq:chart-comparison} for the two charts.  Both local objects are compared with the same restriction of the global object.  On a triple overlap the three comparisons factor through that same restriction, so the derived cocycle is the identity.  Replacing a local resolution by another one does not change the zigzag by \cref{prop:sheaf-refinement}.
\end{proof}

This argument avoids the need to choose and compare explicit monoid presentations on every overlap.  The chart-change information is carried by the global $b$-geometry and by the common-refinement property of monoidal resolutions.

\subsection{The global derived trace}

Retain \textup{(G1)}--\textup{(G3)} and write $n=\dim X$.  Choose one
chain-level regularized representative $(\cR,s_{\cR})$.  Its compactly
supported total-face trace is
\[
 \mathcal I_{\cR,s_{\cR}}:
 \Gamma_c(X,\mathscr A_{\cR,s_{\cR}}^\bullet)\longrightarrow\mathbb R[-n].
\]
On the common interior bridge one has
\begin{equation}\label{eq:global-derived-trace}
 \mathcal I_{\cR,s_{\cR}}\circ\epsilon_{\cR}
 =\int_{X^\circ}:
 \Omega_c^\bullet(X^\circ)\longrightarrow\mathbb R[-n].
\end{equation}

\begin{theorem}[Global derived Stokes trace]\label{thm:global-trace}
The equality \eqref{eq:global-derived-trace} defines a morphism
\[
 \mathcal I_X^{\mathrm{der}}:
 R\Gamma_c\bigl(X,\mathscr A_X^{\bullet,\mathrm{der}}\bigr)
 \longrightarrow\mathbb R[-n]
\]
which is independent of the chosen smooth resolution and of the chosen
chain-level DPP regularization.  On every regularized representative,
$\mathcal I_{\cR,s_{\cR}}\circ D=0$.  Its restriction to a Joyce chart agrees
with the affine derived trace through the corresponding interior bridge.
\end{theorem}

\begin{proof}
By \cref{prop:sheaf-refinement}, every resolved face sheaf is
quasi-isomorphic to $\mathscr J_X^\bullet=j_!\Omega_{X^\circ}^\bullet$.
Taking compactly supported derived global sections identifies all resolved
models with $\Omega_c^\bullet(X^\circ)$.  On this common model every
regularized trace is ordinary integration by
\cref{lem:interior-bridge}.  Hence the induced morphism in the derived
category is independent of both resolution and regularization.  The
chain-level Stokes identity on a chosen representative is
\cref{prop:total-stokes}.  Restriction to charts commutes with the interior
identification.
\end{proof}

\begin{corollary}[Resolution-independent derived trace of AKSZ source cocycles]\label{cor:global-aksz}
Let $(\Y,\omega_\Y=\delta\alpha_\Y,\Theta)$ be an exact $QP$-target of degree
$n-1$.  On every chosen smooth regularized representative the universal AKSZ
integrands satisfy the resolved BV--BFV identity of \cref{thm:derived-bvbfv}.
After forgetting the face filtration, their pairing with
$\mathcal I_X^{\mathrm{der}}$ is resolution-independent whenever the resulting
compactly supported total source cocycles represent corresponding classes
under the common-interior roof.  In particular, this holds for the
interior-bridge representatives of compactly supported forms supported away
from the boundary.  Compatibility of fields only on the common interior,
including basic bulk fields, does not by itself supply this hypothesis.

This is a statement about universal integrands in the linear total source
complex; it does not construct a nonlinear derived mapping space or identify
nonlinear field spaces across resolutions.  Separate face-degree descendants
are not canonically identified unless a filtered refinement transfer is part
of the comparison data.
\end{corollary}

\begin{remark}[What globalization does not solve]\label{rem:global-not-strict}
The sheaf $\mathscr A_{X}^{\bullet,\mathrm{der}}$ contains exceptional face components locally.  Global descent shows that these components glue coherently; it does not contract them.  Passing from this global derived linear source complex to a strict intrinsic incidence complex still requires compatible subdivision/aggregation contractions; a strict contravariant face functor would additionally require path-independent restriction maps.  In nonlinear theories the analogous operation is a global BV pushforward.
\end{remark}

\section{Conditional intrinsic \texorpdfstring{$b$}{b}-AKSZ transgression}\label{sec:intrinsic-transgression}

\subsection{Target data}

We use the standard framework of graded symplectic $Q$-manifolds, or
$QP$-manifolds, of Roytenberg and
\v{S}evera~\cite{Roytenberg,Severa}; the degree conventions below are theirs.
Let $(\Y,\omega_\Y,Q_\Y)$ be a graded symplectic $Q$-manifold with
\[
 |\omega_\Y|=n-1,
 \qquad
 \omega_\Y=\delta\alpha_\Y.
\]
Assume $Q_\Y$ is Hamiltonian:
\[
 \iota_{Q_\Y}\omega_\Y=\delta\Theta,
 \qquad
 |\Theta|=n,
 \qquad
 \{\Theta,\Theta\}=0.
\]

For an intrinsic face $F$ of codimension $k$, define the formal mapping space
\[
 \F_F^b=\Map(\bT[1]F,\Y).
\]
Its tangent vectors are sections of the pulled-back tangent bundle of $\Y$, and the evaluation map is
\[
 \ev:\bT[1]F\times\F_F^b\longrightarrow\Y.
\]

\begin{lemma}[Regularized restriction of AKSZ fields]\label{lem:regularized-field-restriction}
Let $G\subseteq F$ and let $\cT$ be a multiplicative regularized Stokes trace
system.  A formal field $\Phi\in\F_F^b$ is contravariantly encoded by a unital
morphism of graded-commutative algebras
\[
 \Phi^*:C^\infty(\Y)\longrightarrow\bOmega^\bullet(F).
\]
No compatibility with $Q_\Y$ and $\db$ is imposed on an arbitrary field.
Define its regularized boundary restriction by
\begin{equation}\label{eq:regularized-field-restriction}
 (\rho^{\cT}_{G/F}\Phi)^*
 :=\Reg_{G/F}\circ\Phi^*:
 C^\infty(\Y)\longrightarrow\bOmega^\bullet(G).
\end{equation}
Then $\rho^{\cT}_{G/F}:\F_F^b\to\F_G^b$ is well defined and strictly
functorial in flags.  Moreover, if $\Reg_{G/F}$ is extended to forms on
$\F_F^b$ by acting on the source coefficients and trivially on the
field-space differential, then for every differential form or tensor
$\tau$ on $\Y$ used in AKSZ transgression,
\begin{equation}\label{eq:evaluation-regularization}
 \Reg_{G/F}^{\mathrm{src}}\bigl(\ev_F^*\tau\bigr)
 =(\rho^{\cT}_{G/F})^*\bigl(\ev_G^*\tau\bigr).
\end{equation}
\end{lemma}

\begin{proof}
Because $\Reg_{G/F}$ is a unital cdga morphism, hence in particular a
unital morphism of graded-commutative algebras, and preserves the field
subalgebra $\bOmega$, the composite in
\eqref{eq:regularized-field-restriction} again defines a graded map and hence a
field on $G$.  Strict functoriality follows immediately from
\eqref{eq:reg-functorial}.  Equation \eqref{eq:evaluation-regularization} is
the naturality of evaluation: both sides are obtained by applying the same
target tensor $\tau$ to the coordinate functions of $\Phi$ and then applying
$\Reg_{G/F}$ to every source coefficient.  Multiplicativity is exactly what
allows this identity to pass through products of coordinate functions and
their differentials.
\end{proof}

\subsection{Transgressed data}

\begin{definition}\label{def:transgression}
For a regularized Stokes trace system $\cT$, define
\begin{align}
 \omega_F^b&=\Tr_F^{\cT}\ev^*\omega_\Y,\\
 \alpha_F^b&=\Tr_F^{\cT}\ev^*\alpha_\Y,\\
 S_F^b&=\Tr_F^{\cT}
 \bigl(\iota_{\db}\ev^*\alpha_\Y+\ev^*\Theta\bigr).
\end{align}
The cohomological vector field is
\[
 Q_F=\widehat{\db}+\widehat{Q_\Y}.
\]
\end{definition}

The notation is formal in the same sense as ordinary AKSZ theory on infinite-dimensional mapping spaces: transgression is interpreted on local functionals and variations are compactly supported whenever necessary.

\begin{lemma}[Degrees]\label{lem:degrees}
If $F$ has codimension $k$ in $X$, then
\[
 |\omega_F^b|=k-1,\qquad
 |\alpha_F^b|=k-1,\qquad
 |S_F^b|=k,\qquad
 |Q_F|=1.
\]
\end{lemma}

\begin{proof}
Regularized integration over $F$ lowers degree by $\dim F=n-k$.  The target symplectic form has degree $n-1$ and the Hamiltonian has degree $n$.
\end{proof}

\subsection{The BV--BFV identity}

\begin{theorem}[Intrinsic regularized presymplectic BV--BFV identity]\label{thm:bvbfv}
Assume $\cT$ is a multiplicative regularized Stokes trace system.  The data in
\cref{def:transgression} satisfy
\begin{equation}\label{eq:intrinsic-bvbfv-factorized}
 \iota_{Q_F}\omega_F^b
 =(-1)^{\dim F}\delta S_F^b
 +\sum_{G\prec F}[F:G]\,
 (\rho^{\cT}_{G/F})^*\alpha_G^b,
\end{equation}
where $\rho^{\cT}_{G/F}$ is the regularized restriction of fields from
\cref{lem:regularized-field-restriction} and
\[
 \alpha_G^b=\Tr_G^{\cT}\ev_G^*\alpha_\Y
\]
is the transgressed primitive defined intrinsically on the boundary field
space $\F_G^b$.
\end{theorem}

\begin{proof}
The target part $\widehat{Q_\Y}$ is Hamiltonian with Hamiltonian
$\Tr_F^{\cT}\ev_F^*\Theta$.  For the source part, Cartan's formula and
variation of the kinetic term give the standard AKSZ bulk contribution plus
\[
 \Tr_F^{\cT}\db(\ev_F^*\alpha_\Y).
\]
The Stokes identity \eqref{eq:stokes-trace} turns this into
\[
 \sum_{G\prec F}[F:G]\,
 \Tr_G^{\cT}\Reg_{G/F}^{\mathrm{src}}(\ev_F^*\alpha_\Y).
\]
By \cref{lem:regularized-field-restriction}, each summand equals
$(\rho^{\cT}_{G/F})^*\alpha_G^b$.  Thus the boundary contribution factors
through the boundary field space rather than merely defining a one-form on
bulk fields.  The sign $(-1)^{\dim F}$ is the convention checked in
\cref{app:signs}.
\end{proof}

\begin{corollary}[Closed source]\label{cor:closed}
If $F$ has no boundary, then the presymplectic Hamiltonian identity is
\[
 \iota_{Q_F}\omega_F^b=(-1)^{\dim F}\delta S_F^b.
\]
If, in addition, $\omega_F^b$ is weakly nondegenerate on the chosen field
class and its induced BV bracket is defined, then $Q_F$ is the Hamiltonian
vector field of $(-1)^{\dim F}S_F^b$.  Since $Q_F^2=0$, the classical master
bracket $\{S_F^b,S_F^b\}$ is a field-independent central constant.  If
$k=|F|$, then $|S_F^b|=k$ and the bracket induced by a two-form of degree
$k-1$ has degree $1-k$, so this constant has degree $k+1$.  Over the ground
field $\mathbb R$ concentrated in degree zero it therefore vanishes, and the
classical master equation holds.  Over a graded coefficient base with
positive-degree constants, its vanishing must instead be imposed separately.
\end{corollary}

\begin{corollary}[Higher descent]\label{cor:higher}
The facewise AKSZ data form a coherent $\mathrm{BF}^k\mathrm V$-type system.
The $\mathrm{BF}^k\mathrm V$ degree conventions, and their relaxed variant, are
those of \cite[Defs.~2.16 and~2.18]{CFT}.
Here this phrase means precisely the codimension-graded tower of
presymplectic AKSZ data, regularized restriction maps, and incidence identities
displayed in \cref{eq:intrinsic-bvbfv-factorized}; no additional higher
geometric structure is asserted.  In codimension two, the two
boundary-of-boundary contributions cancel because the incidence differential
on $\mathsf{Face}(X)$ squares to zero.
\end{corollary}

\begin{proof}
Apply \cref{thm:bvbfv} on each face.  Strict functoriality of the regularized field restrictions identifies the two composites to every intrinsic codimension-two face, while their incidence signs cancel by \cref{lem:intrinsic-incidence}.
\end{proof}

\section{The exceptional interval}\label{sec:exceptional-interval}

An elementary blow-up of a codimension-two face, and equally a stellar
subdivision of a two-dimensional cone, creates one exceptional normal
interval.  Every codimension-two comparison in this paper reduces at the local level
to a statement about that interval, which is also the first nontrivial member
of the higher-dimensional exceptional-simplex family.  We treat it once, in
the three forms in which it is used: as a cellular complex with a preferred contraction, as
the target of an unregularized Whitney--Dupont retraction from the continuum,
and as the target of a regularized one.

\subsection{The cellular interval and its diagonal splitting}

Let $C^\bullet(I)$ be the cellular cochain complex of $I=[0,1]$ for the CW
decomposition with vertices $0,1$ and one oriented edge:
\begin{equation}\label{eq:interval-complex}
 C^\bullet(I)=\operatorname{span}\{\epsilon_0,\epsilon_1,\epsilon_{01}\},
 \qquad
 d\epsilon_0=-\epsilon_{01},\quad
 d\epsilon_1=\epsilon_{01}.
\end{equation}
Put
\[
  \epsilon_+=\tfrac12(\epsilon_0+\epsilon_1),
  \qquad
  \epsilon_-=\epsilon_1-\epsilon_0,
\]
so that
\[
  d\epsilon_+=0,
  \qquad
  d\epsilon_-=2\epsilon_{01}.
\]
Let $p_+:C^\bullet(I)\to\mathbb R\epsilon_+$ be the diagonal projection and
$i_+$ its inclusion, and define a degree $-1$ operator by
\begin{equation}\label{eq:interval-contraction}
  h_+(\epsilon_{01})=\tfrac12\epsilon_-,
  \qquad
  h_+(\epsilon_+)=h_+(\epsilon_-)=0.
\end{equation}

\begin{proposition}[Diagonal contraction of the exceptional interval]\label{prop:diagonal-interval}
The triple $(i_+,p_+,h_+)$ is contraction data
\[
  C^\bullet(I)\quad\rightleftarrows\quad\mathbb R\epsilon_+.
\]
Equivalently, the reduced sector
\[
  C^\bullet_{\mathrm{exc}}(I)
  =\operatorname{span}\{\epsilon_-,\epsilon_{01}\}
\]
is contractible.  For every cochain complex $A^\bullet$, the exceptional
tensor factor
\[
  A^\bullet\otimes C^\bullet_{\mathrm{exc}}(I)
\]
is contractible by the tensor-product homotopy.
\end{proposition}

\begin{proof}
On $\epsilon_{01}$ one has $dh_+(\epsilon_{01})=\epsilon_{01}$, while on
$\epsilon_-$ one has $h_+d(\epsilon_-)=\epsilon_-$.  Both sides vanish on
$\epsilon_+$.  Hence
\[
  dh_++h_+d=\id-i_+p_+.
\]
The tensor-product statement is the standard tensor trick.
\end{proof}

With coefficients in a Lie algebra $\mathfrak g$ we write
\[
 C_{\mathrm{const}}^\bullet(I;\mathfrak g)
 =\mathfrak g\otimes\langle\epsilon_0+\epsilon_1\rangle,
 \qquad
 C_{\mathrm{red}}^\bullet(I;\mathfrak g)
 =\mathfrak g\otimes\langle\epsilon_-,\epsilon_{01}\rangle,
\]
the first being the image of $i_+$ and the second the acyclic complement.

\begin{lemma}[Normalized exceptional Gaussian]\label{lem:interval-torsion}
Let $A^\bullet$ be a finite-dimensional cochain complex and consider the
quadratic BV theory on $A^\bullet\otimes C^\bullet(I)$ together with
its shifted algebraic dual.  Use the cyclic cotangent splitting induced by
\eqref{eq:interval-contraction}, and take as gauge-fixing Lagrangian its
standard UV cotangent factor.  The finite-dimensional BV pushforward produces
the induced quadratic theory on the diagonal factor and a nonzero
field-independent element of the corresponding determinant line.  With the
elementary-collapse bases and the cellular BV normalization this element is
$1$.  A different compatible basis or Gaussian normalization changes it by
an overall constant Berezinian.  No assertion is made here for arbitrary
cotangent Lagrangians.
\end{lemma}

\begin{proof}
By \cref{prop:diagonal-interval} the exceptional sector is a contractible
cotangent pair, so its fields are generalized auxiliary fields.  The action is
quadratic, hence there are no interaction vertices from which a
field-dependent effective term could be formed.  In the normalized real
cellular bases
\[
 \bigl\{\tfrac12\epsilon_-\bigr\}
 \quad\text{and}\quad
 \{\epsilon_{01}\},
\]
the reduced differential is the identity because
$d(\tfrac12\epsilon_-)=\epsilon_{01}$.  Its Reidemeister torsion is therefore
$1$.  Tensoring the contraction with $A^\bullet$ and pairing it with the
shifted algebraic dual gives the determinant-line isomorphism of the cyclic
UV pair; the cellular BV convention declares the elementary acyclic pair to
have unit Gaussian.  Returning to other compatible bases multiplies that
element by the Berezinian of a constant change-of-basis matrix, so it remains
field independent.  This is precisely the normalization compatible with the
elementary cellular collapse
$I\searrow *$ used in cellular BV pushforward
\cite{CMRCellular,CMRPushforward}.
\end{proof}

\begin{remark}
The preceding torsion lemma is finite-dimensional.  When the coefficient
factor is an infinite-dimensional space of continuum forms, an additional
analytic choice is required before the same sentence can be interpreted as
an equality of Gaussian half-densities.  We isolate that choice rather than
regard it as automatic.
\end{remark}

\begin{assumption}[Gaussian determinant regularity]\label{ass:gaussian-determinant}
Whenever an infinite-dimensional acyclic quadratic BF sector is integrated
out below, a determinant-line/Gaussian regularization is chosen such that:
\begin{enumerate}[label=\textup{(\roman*)},leftmargin=2.5em]
\item the regularized Gaussian integral exists on the chosen gauge-fixing
Lagrangian and produces a nonzero element of the determinant line;
\item the determinant is compatible with the $A/B$ duality and with
regularized face restriction; and
\item determinant-line identifications are multiplicative under composition
of the deformation retracts used in the paper (the Gaussian Fubini property).
\end{enumerate}
This assumption is automatic for the finite-dimensional cellular sectors and
for a fixed finite-mode cutoff.  No construction of such a determinant
regularization for the full continuum logarithmic complex is claimed here.
\end{assumption}

\subsection{Discretizing the interval: the unregularized data}

The Whitney map, integration map, and Dupont homotopy~\cite{Dupont76} for
\eqref{eq:interval-complex} are
\begin{align}
  W(c_0\epsilon_0+c_1\epsilon_1+a\epsilon_{01})
  &=(1-\theta)c_0+\theta c_1+\dd\theta\,a,
  \label{eq:interval-Whitney}\\
  R(f(\theta)+g(\theta)\dd\theta)
  &=\epsilon_0f(0)+\epsilon_1f(1)+\epsilon_{01}\int_0^1g(u)\,\dd u,
  \label{eq:interval-integration}\\
  (\kappa\alpha)(\theta)
  &=\int_0^\theta\iota_{\partial_u}\alpha(u)\,\dd u
       -\theta\int_0^1\iota_{\partial_u}\alpha(u)\,\dd u.
  \label{eq:interval-Dupont}
\end{align}
They obey
\begin{equation}\label{eq:interval-SDR}
  d\kappa+\kappa d=\id-WR,
  \qquad
  RW=\id,
  \qquad
  \kappa W=R\kappa=\kappa^2=0.
\end{equation}
Most importantly for corners,
\begin{equation}\label{eq:Dupont-endpoint-vanish}
  \operatorname{ev}_0\kappa=0,
  \qquad
  \operatorname{ev}_1\kappa=0.
\end{equation}
The kernel in \eqref{eq:interval-Dupont} is precisely the standard interval
propagator with Dirichlet behavior in its first argument
\cite{CMRQuantum,CMRPushforward}.

\subsection{Discretizing the interval: the regularized data}

On a $b$-boundary the endpoint restriction at $0$ does not exist, and the
ordinary interval contraction must be replaced by a logarithmic one.  Let
$I_b=[0,\varepsilon]$ with logarithmic boundary at $r=0$.  The endpoint
$r=\varepsilon$ is an artificial outer cutoff of the collar, so in this subsection
we use the sub-cdga
\[
 \cA_{\log,0}^\bullet(I_b)\subset A^{\bullet,\log}(I_b)
\]
of DPP logarithmic forms which are smooth on a neighbourhood of
$r=\varepsilon$; logarithmic singularities are allowed only at $r=0$.  Fix a
scale $\lambda>0$ at $r=0$.  Since all forms in this subcomplex are smooth at
the outer endpoint, the DPP regularized integral is independent of the auxiliary
scale chosen there; we denote it simply by $\Tr_{I_b}^{\lambda}$, and write
$\Reg_0^\lambda$ for the regularized restriction at $0$.

The next lemma supplies the point that is needed for a genuine chain homotopy:
the variable-upper-limit finite part is itself a logarithmic function.

\begin{lemma}[Normalized logarithmic primitive on an interval]\label{lem:normalized-log-primitive}
For every $\eta\in\cA_{\log,0}^1(I_b)$ there is a unique
$\mathsf P_\lambda\eta\in\cA_{\log,0}^0(I_b)$ such that
\begin{equation}\label{eq:normalized-log-primitive}
 \db(\mathsf P_\lambda\eta)=\eta,
 \qquad
 \Reg_0^\lambda(\mathsf P_\lambda\eta)=0.
\end{equation}
Moreover, for every $0<r\leq\varepsilon$,
\begin{equation}\label{eq:primitive-as-finite-part}
 (\mathsf P_\lambda\eta)(r)
 =\Tr_{[0,r]}^\lambda(\eta|_{[0,r]}),
\end{equation}
where the right endpoint $r$ is evaluated ordinarily.  In particular,
\begin{equation}\label{eq:primitive-at-epsilon}
 (\mathsf P_\lambda\eta)(\varepsilon)=\Tr_{I_b}^\lambda(\eta).
\end{equation}
\end{lemma}

\begin{proof}
By the local structure theorem for DPP logarithmic functions and
\cite[Def.~6.4]{DPP}, every logarithmic one-form on the one-dimensional basic
chart has a finite expression
\[
 \eta=\sum_{j=0}^{N} f_j(r)(\log r)^j\,\frac{\dd r}{r},
 \qquad f_j\in C^\infty([0,\varepsilon]).
\]
Write $f_j(r)=f_j(0)+r g_j(r)$.  The singular part has the logarithmic
primitive
\[
 \frac{f_j(0)}{j+1}(\log r)^{j+1}.
\]
For the remainder put
\[
 J_j(g)(r):=\int_0^r g(u)(\log u)^j\,\dd u.
\]
We claim $J_j(g)$ is again a DPP logarithmic function.  For $j=0$ it is
smooth.  If $G(r)=\int_0^r g(u)\,\dd u=rh(r)$ with $h$ smooth, integration by
parts gives
\[
 J_j(g)(r)=G(r)(\log r)^j-jJ_{j-1}(h)(r),
\]
so the claim follows by induction on $j$.  Thus termwise integration produces
some $\widetilde{\mathsf P}\eta\in\cA_{\log,0}^0(I_b)$ with
$\db\widetilde{\mathsf P}\eta=\eta$.  Define
\[
 \mathsf P_\lambda\eta
 :=\widetilde{\mathsf P}\eta
   -\Reg_0^\lambda(\widetilde{\mathsf P}\eta).
\]
The regularized value is defined by the DPP virtual restriction; in the local
expansion it is obtained by substituting $\log r=\log\lambda$
\cite[Thm.~5.12 and Ex.~5.16]{DPP}.  Hence
\eqref{eq:normalized-log-primitive} holds.

If two normalized primitives existed, their difference would have zero
differential.  Restricting to the interior and using the injectivity in
\cite[Thm.~5.12]{DPP}, that difference is a constant, and the normalization at
$0$ forces it to be zero.  Finally apply the DPP regularized Stokes formula to
$\mathsf P_\lambda\eta$ on $[0,r]$:
\[
 \Tr_{[0,r]}^\lambda(\eta)
 =(\mathsf P_\lambda\eta)(r)
  -\Reg_0^\lambda(\mathsf P_\lambda\eta)
 =(\mathsf P_\lambda\eta)(r).
\]
This proves \eqref{eq:primitive-as-finite-part} and
\eqref{eq:primitive-at-epsilon}; compare \cite[Cor.~7.11 and Ex.~7.13]{DPP}.
\end{proof}

Choose a smooth function $\chi:[0,\varepsilon]\to[0,1]$ with
$\chi(0)=0$ and $\chi(\varepsilon)=1$; for collar applications we take it
constant near both endpoints.  Define the regularized Poincar\'e map
$p_\lambda$ by
\begin{align}
 p_\lambda(f)
 &=\Reg_0^\lambda(f)\,\epsilon_0+f(\varepsilon)\,\epsilon_1,
 \label{eq:regularized-poincare-0}\\
 p_\lambda(\eta)
 &=\Tr_{I_b}^\lambda(\eta)\,\epsilon_{01},
 \label{eq:regularized-poincare-1}
\end{align}
for $f\in\cA_{\log,0}^0(I_b)$ and
$\eta\in\cA_{\log,0}^1(I_b)$.  Define the Whitney map by
\begin{equation}\label{eq:regularized-whitney}
 i_\chi(\epsilon_0)=1-\chi,
 \qquad
 i_\chi(\epsilon_1)=\chi,
 \qquad
 i_\chi(\epsilon_{01})=\dd\chi.
\end{equation}
For a logarithmic one-form $\eta$, set
\begin{equation}\label{eq:regularized-normal-homotopy}
 K_{\lambda,\chi}\eta
 =\mathsf P_\lambda\eta
  -\chi\,\Tr_{I_b}^\lambda(\eta),
 \qquad
 K_{\lambda,\chi}|_{\cA_{\log,0}^0}=0.
\end{equation}
By \cref{lem:normalized-log-primitive}, this is equivalently the finite-part
formula
\[
 (K_{\lambda,\chi}\eta)(r)
 =\Tr_{[0,r]}^\lambda(\eta)
  -\chi(r)\Tr_{[0,\varepsilon]}^\lambda(\eta),
\]
but the definition through $\mathsf P_\lambda$ makes its membership in the
logarithmic function algebra explicit.

\begin{proposition}[Regularized interval contraction]\label{prop:regularized-interval-contraction}
The maps in
\cref{eq:regularized-poincare-0,eq:regularized-poincare-1,eq:regularized-whitney,eq:regularized-normal-homotopy}
define a special deformation retract
\[
 \bigl(C^\bullet(I),d\bigr)
 \underset{p_\lambda}{\overset{i_\chi}{\rightleftarrows}}
 \bigl(\cA_{\log,0}^\bullet(I_b),\db\bigr),
 \qquad K_{\lambda,\chi}.
\]
Thus
\[
 p_\lambda i_\chi=\id,
 \qquad
 \db K_{\lambda,\chi}+K_{\lambda,\chi}\db
 =\id-i_\chi p_\lambda,
\]
and
\[
 K_{\lambda,\chi}^2=0,
 \qquad
 p_\lambda K_{\lambda,\chi}=0,
 \qquad
 K_{\lambda,\chi}i_\chi=0.
\]
The endpoint maps commute with the retract, with the map at $0$ interpreted as
$\Reg_0^\lambda$ and the map at $\varepsilon$ as ordinary evaluation.
\end{proposition}

\begin{proof}
Regularized Stokes gives, for every
$f\in\cA_{\log,0}^0(I_b)$,
\[
 \Tr_{I_b}^\lambda(\db f)
 =f(\varepsilon)-\Reg_0^\lambda(f).
\]
Hence $p_\lambda\db=d\,p_\lambda$.  Since $\chi$ and $\dd\chi$ are smooth,
DPP regularization agrees with ordinary endpoint evaluation and integration on
these terms; therefore $p_\lambda i_\chi=\id$.

For a one-form $\eta$, \cref{lem:normalized-log-primitive} gives
\[
 \db K_{\lambda,\chi}\eta
 =\eta-\dd\chi\,\Tr_{I_b}^\lambda(\eta)
 =\eta-i_\chi p_\lambda\eta.
\]
For a zero-form $f$, the normalized primitive of $\db f$ is
\[
 \mathsf P_\lambda(\db f)=f-\Reg_0^\lambda(f),
\]
by uniqueness in \cref{lem:normalized-log-primitive}.  Hence
\[
 K_{\lambda,\chi}\db f
 =f-\Reg_0^\lambda(f)
 -\chi\bigl(f(\varepsilon)-\Reg_0^\lambda(f)\bigr)
 =f-i_\chi p_\lambda f.
\]
This proves the homotopy identity.  The normalization and
\eqref{eq:primitive-at-epsilon} imply
\[
 \Reg_0^\lambda(K_{\lambda,\chi}\eta)=0,
 \qquad
 (K_{\lambda,\chi}\eta)(\varepsilon)=0,
\]
so $p_\lambda K_{\lambda,\chi}=0$.  Moreover
$\mathsf P_\lambda(\dd\chi)=\chi$ and
$\Tr_{I_b}^\lambda(\dd\chi)=1$, hence
$K_{\lambda,\chi}\dd\chi=0$ and therefore
$K_{\lambda,\chi}i_\chi=0$.  Finally $K_{\lambda,\chi}^2=0$ for degree
reasons, since $K_{\lambda,\chi}$ vanishes on degree-zero forms.
\end{proof}

\subsection{The two-ended logarithmic interval and the product square}

The conifold exceptional fibre has genuine boundary on all four sides.  The
one-ended collar retract above must therefore be replaced, in the angular
directions, by a two-ended relative retract.  We give the complete formulas
because they are used below to retain the tangential logarithms of pulled-back
monomials on the exceptional square.

Let \(\overline I=[0,1]\), with its basic positive log structure at both
endpoints, and choose endpoint scales \(\lambda_0,\lambda_1>0\).  Write
\(\Reg_i\) for the corresponding DPP virtual restriction and
\(\Tr_{\overline I}^{\lambda_0,\lambda_1}\) for the regularized integral, so
that
\[
 \Tr_{\overline I}^{\lambda_0,\lambda_1}(\db f)
 =\Reg_1(f)-\Reg_0(f).
\]

\begin{lemma}[Two-ended normalized primitive]\label{lem:two-ended-primitive}
For every \(\eta\in A^{1,\log}(\overline I)\) there is a unique
\(\mathsf P_0\eta\in A^{0,\log}(\overline I)\) such that
\[
 \db\mathsf P_0\eta=\eta,
 \qquad
 \Reg_0(\mathsf P_0\eta)=0.
\]
Moreover
\[
 \Reg_1(\mathsf P_0\eta)
 =\Tr_{\overline I}^{\lambda_0,\lambda_1}(\eta).
\]
The construction is linear.  Its two regularized endpoint values are fixed by
the displayed normalization and Stokes identity; no componentwise naturality
with respect to both endpoint inclusions is asserted.
\end{lemma}

\begin{proof}
Choose a partition of unity consisting of a left endpoint collar, an interior
interval and a right endpoint collar.  On the left collar use
\cref{lem:normalized-log-primitive}.  On the right collar use the same lemma in
the coordinate \(1-t\), and on the interior use ordinary integration.  On
overlaps, two local primitives differ by a constant; subtracting these
constants glues them to a global logarithmic primitive.  The logarithmic
function algebra is preserved because the endpoint calculation is the
one-variable calculation in the proof of
\cref{lem:normalized-log-primitive}, while all interior corrections are
smooth.  Subtracting its regularized value at zero gives the stated
normalization.  Uniqueness follows because a logarithmic function with zero
differential is constant on the interior and the DPP restriction to the
interior is injective.  The last identity is regularized Stokes applied to
\(\mathsf P_0\eta\).
\end{proof}

Define the shifted complete face-diagram complex of the interval by
\[
 \mathfrak R_{\overline I}^m
 =A^{m,\log}(\overline I)
 \oplus A^{m-1,\log}(\{0,1\}).
\]
Thus its only nonzero components are logarithmic zero-forms in degree zero and
triples \((\eta;a_0,a_1)\), consisting of a logarithmic one-form and two
endpoint scalars, in degree one.  With the boundary orientation
\(\partial\overline I=\{1\}-\{0\}\), its differential is
\begin{equation}\label{eq:two-ended-relative-differential}
 D_I f=(\db f;-\Reg_0f,\Reg_1f),
 \qquad
 D_I(\eta;a_0,a_1)=0.
\end{equation}
Let
\[
 C_{\mathrm{nf}}^0(\overline I)=\mathbb R e,
 \qquad
 C_{\mathrm{nf}}^1(\overline I)
 =\mathbb R v_0\oplus\mathbb R v_1,
 \qquad
 d_{\mathrm{nf}}e=v_1-v_0.
\]
This is the cellular chain complex of the interval regraded by normal-face
codimension: the edge has degree zero and its vertices have degree one.
Define
\begin{align}
 I_I(e)&=1,& I_I(v_0)&=(0;1,0),&I_I(v_1)&=(0;0,1),\label{eq:relative-interval-inclusion}\\
 P_I(f)&=\Reg_0(f)e,&
 P_I(\eta;a_0,a_1)&=a_0v_0+
 \bigl(a_1-\Tr_{\overline I}^{\lambda_0,\lambda_1}\eta\bigr)v_1,
 \label{eq:relative-interval-projection}\\
 H_I(f)&=0,&H_I(\eta;a_0,a_1)&=\mathsf P_0\eta.
 \label{eq:relative-interval-homotopy}
\end{align}

\begin{proposition}[Two-ended relative interval retract]
\label{prop:two-ended-relative-retract}
The maps \eqref{eq:relative-interval-inclusion}--
\eqref{eq:relative-interval-homotopy} form a special deformation retract
\[
 (\mathfrak R_{\overline I}^\bullet,D_I)
 \underset{I_I}{\overset{P_I}{\rightleftarrows}}
 (C_{\mathrm{nf}}^\bullet(\overline I),d_{\mathrm{nf}}),
 \]
with
\[
 P_II_I=\id,
 \qquad
 D_IH_I+H_ID_I=\id-I_IP_I,
 \qquad
 H_I^2=P_IH_I=H_II_I=0.
\]
This is a retract of the \emph{total} relative face-diagram complex.
The homotopy vanishes on the endpoint summands and is normalized at the left
endpoint, but the three maps are not claimed to commute componentwise with
both un-totalized endpoint restrictions.
\end{proposition}

\begin{proof}
For a zero-form \(f\),
\[
 H_ID_If=\mathsf P_0(\db f)=f-\Reg_0(f)
 =f-I_IP_If.
\]
For \((\eta;a_0,a_1)\), put
\(T=\Tr_{\overline I}^{\lambda_0,\lambda_1}\eta\).  Then
\[
 D_IH_I(\eta;a_0,a_1)=(\eta;0,T)
 =(\eta;a_0,a_1)-I_IP_I(\eta;a_0,a_1).
\]
These formulas also show that \(I_I\) and \(P_I\) are cochain maps and that
\(P_II_I=\id\).  The side conditions follow from
\(\Reg_0\mathsf P_0=0\) and from the fact that \(H_I\) vanishes on degree
zero and on the endpoint components.  The endpoint coordinates in
\eqref{eq:two-ended-relative-differential} are part of the total differential;
the proof uses no separate endpoint-naturality statement.
\end{proof}

Now put \(E=\overline I_1\times\overline I_2\), orient it by the ordered
product orientation, and use product DPP scales.  Let
$\mathsf F_r(E)$ denote its oriented embedded faces of codimension $r$ and
define the complete \emph{unordered} column face complex
\begin{equation}\label{eq:square-complete-face-complex}
 \mathfrak R_E^m
 =\bigoplus_{r=0}^2\ \bigoplus_{F\in\mathsf F_r(E)}
   A^{m-r,\log}(F),
 \qquad D_E=\db+(-1)^q d_{\mathrm{face}}
 \quad\text{on form degree }q.
\end{equation}
The product regularization identifies the ordered lifts by permuting the two
interval factors, so \eqref{eq:square-complete-face-complex} is the unordered
representative of the invariant ordered DPP complex under
\cref{prop:ordered-face-dictionary}.  In particular, both ordered lifts of a
vertex and their orientation signs are already accounted for.

For a form $\alpha$ of form degree $q$ supported on a codimension-$r$ face
$F_1\leq\overline I_1$ and a form $\beta$ of form degree $q'$ supported on a
codimension-$s$ face $F_2\leq\overline I_2$, set
\begin{equation}\label{eq:square-exterior-product-sign}
 \mu(\alpha\otimes\beta)
 =(-1)^{rq'}\operatorname{pr}_1^*\alpha\wedge
   \operatorname{pr}_2^*\beta
 \quad\text{on }F_1\times F_2,
\end{equation}
where product coorientations list the first-factor normals before the
second-factor normals.  The sign is the standard interchange of the
face-degree suspension of $\alpha$ with the form degree of $\beta$.
The product rule for $\db$, the oriented product-boundary formula and
factorization of the product virtual face maps give
\[
 D_E\mu(\alpha\otimes\beta)
 =\mu\bigl(D_{I_1}\alpha\otimes\beta
   +(-1)^{|\alpha|}\alpha\otimes D_{I_2}\beta\bigr).
\]
Thus wedge product of pullbacks, with
\eqref{eq:square-exterior-product-sign}, gives the cochain map
\begin{equation}\label{eq:square-face-diagram-tensor}
 \mu:
 \mathfrak R_{\overline I_1}^\bullet\otimes
 \mathfrak R_{\overline I_2}^\bullet
 \longrightarrow \mathfrak R_E^\bullet .
\end{equation}
The tensor product here is algebraic.  Its image consists of finite sums of
separated-variable logarithmic forms; it is not the full DPP logarithmic face
diagram of \(E\).  No completed tensor-product identity for the full spaces of
smooth or logarithmic forms is asserted.

For \(i=1,2\), choose a finite-dimensional subcomplex
\(\mathfrak L_i^\bullet\subset\mathfrak R_{\overline I_i}^\bullet\) which
contains the image of \(I_{I_i}\), is preserved by \(D_{I_i}\) and
\(H_{I_i}\), and on which \(P_{I_i}\) and both endpoint maps restrict.
Assume that \(\mu\) is injective on
\(\mathfrak L_1\otimes\mathfrak L_2\), as it is for the explicit
separated-variable complex below.  Define the finite separated product
subcomplex
\begin{equation}\label{eq:finite-separated-square}
 \mathfrak R_{E,\mathrm{sep}}^\bullet
 :=\mu\bigl(\mathfrak L_1^\bullet\otimes
 \mathfrak L_2^\bullet\bigr)\subset\mathfrak R_E^\bullet.
\end{equation}
Likewise set
\[
 C_{\mathrm{nf}}^\bullet(E)
 =C_{\mathrm{nf}}^\bullet(\overline I_1)
  \otimes
  C_{\mathrm{nf}}^\bullet(\overline I_2).
\]
It has one generator for the two-cell in degree zero, four edge generators in
degree one and four vertex generators in degree two; its differential is the
oriented cellular boundary, regraded by face codimension.

\begin{proposition}[Finite product-square logarithmic contraction]
\label{prop:product-square-log-contraction}
The tensor trick applied to
\cref{prop:two-ended-relative-retract} gives a special deformation retract of
total face complexes
\[
 (\mathfrak R_{E,\mathrm{sep}}^\bullet,D_E)
 \underset{I_E}{\overset{P_E}{\rightleftarrows}}
 (C_{\mathrm{nf}}^\bullet(E),d_{\mathrm{nf}}).
\]
Explicitly,
\[
 I_E=I_{I_1}\otimes I_{I_2},
 \qquad
 P_E=P_{I_1}\otimes P_{I_2},
\]
and, on a homogeneous tensor,
\begin{equation}\label{eq:square-tensor-homotopy}
 H_E^{\log}(\alpha\otimes\beta)
 =H_{I_1}\alpha\otimes\beta
 +(-1)^{|\alpha|}(I_{I_1}P_{I_1}\alpha)
   \otimes H_{I_2}\beta.
\end{equation}
Their domain contains every selected nonconstant tangential logarithm included in
\(\mathfrak L_1\otimes\mathfrak L_2\); the maps contract its reduced part
explicitly rather than identifying the selected logarithmic coefficients with
constants.
\end{proposition}

\begin{proof}
The special deformation-retract identities follow from the standard tensor
trick and the side conditions in
\cref{prop:two-ended-relative-retract}.  The wedge map
\eqref{eq:square-face-diagram-tensor} intertwines the tensor differential with
the DPP de Rham differential plus the four oriented face maps by
\eqref{eq:square-exterior-product-sign}, and the
invariance assumptions on \(\mathfrak L_i\) make all displayed maps
endomorphisms of the finite separated total subcomplex.  The tensor homotopy
identity decomposes every selected separated logarithmic coefficient into its
cellular regularized projection and an explicitly contracted reduced term.
\end{proof}

\begin{remark}[Totalization versus componentwise face naturality]
\label{rem:two-ended-not-face-natural}
The qualification in \cref{prop:two-ended-relative-retract} is essential.
For example,
\[
 H_I(\dd t;0,0)=t,
 \qquad \Reg_1(t)=1,
\]
whereas the pullback of \(\dd t\) to the right endpoint is zero.  Thus neither
the interval homotopy nor its product-square tensor is a componentwise natural
transformation of the un-totalized normal-face diagrams.  What is used below
is the special deformation retract after the normal-face incidence maps have
been incorporated into \(D_I\) and \(D_E\).  Compatibility in \textup{(G4)}
means compatibility with restrictions in the centre and gluing of the normal
cell bundles; it does not impose the stronger, and here false, componentwise
normal-face condition.  This differs from the one-ended Whitney retract of
\cref{prop:regularized-interval-contraction}, whose homotopy vanishes at both
endpoints and is endpoint-compatible.
\end{remark}

\begin{remark}\label{rem:no-full-square-tensor}
\Cref{prop:product-square-log-contraction} is a finite algebraic statement.
Extending it to the full DPP complex would require locally convex topologies,
a completed tensor product for nuclear Fr\'echet spaces, and continuity of
the regularization and homotopy operators.  None of those analytic assertions
is used in the conifold construction below.
\end{remark}

\begin{remark}\label{rem:regularized-reduces}
Restrict the construction to the ordinary smooth de Rham subcomplex, take
$\varepsilon=1$ and $\chi(\theta)=\theta$.  On smooth functions the DPP
regularized restriction agrees with ordinary evaluation, and on smooth compactly
supported top forms the regularized integral agrees with ordinary integration
\cite[Lem.~5.15 and Sec.~7]{DPP}.  On this smooth subcomplex the maps
\cref{eq:regularized-poincare-0,eq:regularized-poincare-1,eq:regularized-whitney,eq:regularized-normal-homotopy}
therefore reduce to $R$, $W$, and $\kappa$ of
\cref{eq:interval-Whitney,eq:interval-integration,eq:interval-Dupont}.  No such
statement is made for a genuinely logarithmic function such as $\log r$, whose
ordinary value at $0$ does not exist; there the scale-dependent regularized value is
essential.  Thus the regularized construction extends the ordinary one rather than
turning every logarithmic form into an ordinarily integrable form.  The
unregularized data are
what \cref{thm:codim2-continuum-propagator} tensors up to a codimension-two
collar; the regularized data are what \cref{cor:product-log-contraction}
tensors up to a logarithmic collar.
\end{remark}

\begin{remark}
The two comparison theorems that use this section answer different questions
and neither contains the other.
\Cref{thm:abelian-continuum-blowup} compares two resolutions differing by a
single codimension-two blow-up, on a local resolved collar and without
regularization; its content is refinement invariance.
\Cref{thm:abelian-continuum-cellular} compares the continuum and cellular
theories on one global compact resolution equipped with a regularization
datum; its content is discretization.  What they share is exactly
\cref{prop:diagonal-interval,lem:interval-torsion}.
\end{remark}

\section{Global linear strictification}\label{sec:global-linear-strict}

\begin{definition}[Centre-compatible finite refinement tower after normal-face
totalization]\label{def:face-compatible-stellar}
A finite resolution tower is called \emph{centre-compatible after normal-face
totalization} relative to a
specified class of basic linear coefficient complexes if every centre is an
embedded face with a compatible product collar and its resolved normal
directions form a globally defined finite contractible cell bundle.  Denote
its fibre over a centre $Z$ of codimension $c$ by $K_Z$.  For an ordinary
stellar centre one takes $K_Z=\sd\Delta^{c-1}$; a generalized-corner centre may
instead have a contractible polyhedral link subdivision.  The following
chain-level data must exist for every coefficient complex in the specified
class:
\begin{enumerate}[label=\textup{(\alph*)},leftmargin=2.5em]
\item graded subdivision and aggregation maps of face degree zero after the
normal cell fibre is given the face grading $r=c-q$
\[
 i_0:\mathscr B_{X,\cV}^\bullet\longrightarrow\mathscr A_L^\bullet,
 \qquad
 p_0:\mathscr A_L^\bullet\longrightarrow\mathscr B_{X,\cV}^\bullet,
 \qquad p_0i_0=\id;
\]
\item an auxiliary differential \(D_0\) for which \(i_0,p_0\) are cochain
maps between
\((\mathscr B_{X,\cV}^\bullet,d_{\log,\cV})\) and
\((\mathscr A_L^\bullet,D_0)\), together with a homotopy \(h_0\) satisfying
\[
 D_0h_0+h_0D_0=\id-i_0p_0,
 \qquad h_0i_0=p_0h_0=h_0^2=0;
\]
where $d_{\log,\cV}$ has face degree zero, while $D_0$ preserves intrinsic
support and may additionally contain the normal-fibre cellular differential
of face degree $+1$;
\item on the associated graded for the normal-face filtration, aggregation
admits one of the following two explicitly specified forms.  In the
locally constant normal model, its kernel is the coefficient complex on the
centre tensored with the face-regraded reduced cellular \emph{chain} complex
$\widetilde C_{\mathrm{face}}^\bullet(K_Z)$, equipped with a special cellular
contraction.  More generally, a logarithmic cellularization of the total
normal-face complex, compatible with restrictions in the centre, first gives a
special deformation retract
\[
 \mathscr L_Z^\bullet
 \underset{I_Z^{\log}}{\overset{P_Z^{\log}}{\rightleftarrows}}
 \cV_Z^\bullet\otimes C_{\mathrm{face}}^\bullet(K_Z),
 \qquad H_Z^{\log},
\]
and this is composed with cellular augmentation and a specified special
contraction of
$\widetilde C_{\mathrm{face}}^\bullet(K_Z)$.  Thus the total initial kernel
has a two-stage filtration: its logarithmic cellularization kernel is
contracted by $H_Z^{\log}$ and the remaining reduced cellular quotient by the
cellular homotopy.  For an ordinary stellar centre with locally constant
normal coefficients the first stage is the identity and the second is the
barycentric cone contraction \eqref{eq:barycentric-cone} on
$K_Z=\sd\Delta^{c-1}$.  No identification of logarithmic differential forms
with cellular chains is permitted without such a cellularization retract;
\item for the full resolved differential \(D=D_0+\delta\), the operators
\(h_0\delta\) and \(\delta h_0\) are nilpotent with respect to a finite tower
filtration; moreover \(\delta\) has face degree \(+1\) and \(h_0\) has face
degree \(-1\).
\item the datum is incidence-local: if $\jmath_F$ includes the intrinsic
summand indexed by $F$ and $\pi_G$ projects to the summand indexed by $G$, then
\begin{equation}\label{eq:g4-incidence-locality}
 \pi_Gp_0\delta(h_0\delta)^r i_0\jmath_F=0
 \quad\text{for every }r\geq0\text{ unless }G\prec F.
\end{equation}
A sufficient geometric way to verify this condition in the models considered here is
to require the subdivision, aggregation, homotopy, and perturbation maps to preserve
the relevant order-ideal support filtration of the intrinsic face poset.  No converse
between that filtration condition and \eqref{eq:g4-incidence-locality} is needed below.
\end{enumerate}
All data are required to commute with restriction in the centre and to glue
under changes of trivialization of the normal cell bundles.  Here
\emph{centre-compatible} refers precisely to these centre restrictions and
this gluing condition.  The incidences between individual normal cells are already
components of \(D_0\) and \(\delta\); \textup{(G4)} does not additionally
require \(h_0\), \(i_0\), and \(p_0\) to form a componentwise natural
transformation of the un-totalized normal-face diagram.  Such a stronger
condition may be imposed in other models, but is neither used in the filtered
transfer theorem nor satisfied by the anchored relative interval retract of
\cref{prop:two-ended-relative-retract}.  These conditions are the concrete
content of \textup{(G4)}.
\end{definition}

The strengthened formulation of \textup{(G4)} is deliberate.  A graded
splitting alone is not enough for homological perturbation: one needs an
initial contraction and a finite perturbation series.  It is these chain-level
data, rather than a literal choice of one ``strict transform'' for each
intrinsic face, that are used below.  Conversely, \textup{(G4)} should not be
read as an existence theorem: it is the explicit chain-level hypothesis under
which the filtered transfer below is carried out.

\begin{definition}[BV-cyclic field datum]\label{def:bf-cyclic-enhancement}
Fix dual basic coefficient systems $V$ and $V^*$ and, if boundary
polarizations are present, the corresponding relative/absolute $A$- and
$B$-field complexes.  The two face directions must be distinguished.  The
$A$-fields form a contravariant face-cochain system, with restriction
\[
 r^A_{G/F}:\mathscr F_F^A\longrightarrow\mathscr F_G^A
 \qquad(G\prec F).
\]
The complementary $B$-fields form the BF-dual face-\emph{chain} system, with
corestriction
\[
 (r^A_{G/F})_!: \mathscr F_G^B\longrightarrow\mathscr F_F^B
\]
characterized by graded adjointness.  In a finite cellular model this is the
shifted algebraic transpose of $r^A_{G/F}$.  In a continuum model it requires
a Poincar\'e--Lefschetz/Gysin realization compatible with the relative and
absolute boundary conditions, and its existence is part of \textup{(G5)}; it
is not obtained by pretending that algebraic duality preserves the direction
of restriction.

Put
\[
 W_A=V,\qquad W_B=V^*,\qquad s_A=1,\qquad s_B=n-2.
\]
For a terminal resolved representative $\beta_L:X_L\to X$, let
$q_L:\mathsf{Face}(X_L)\to\mathsf{Face}(X)$ be the support map.  For
$A$ define the underlying graded field objects by
\begin{align*}
 (\mathscr F_X^A)^m
 &=\bigoplus_{G\in\mathsf{Face}(X)}
   \bigl(\Omega_b^\bullet(G;V|_G)[1]\bigr)^{m-|G|},\\
 (\mathscr F_L^A)^m
 &=\bigoplus_{H\in\mathsf{Face}(X_L)}
   \bigl(\Omega_b^\bullet
   (H;(\beta_L|_H)^*(V|_{q_L(H)}))[1]\bigr)^{m-|H|}.
\end{align*}
The objects $\mathscr F_X^B$ and $\mathscr F_L^B$ are the complementary
relative/absolute BF-dual chain totalizations.  Facewise their local field
spaces are modeled on $\Omega_b^\bullet(-;V^*)[n-2]$, but their incidence
arrows are the corestrictions above.  For finite-dimensional cellular models
they are, by definition, the shifted algebraic duals of the corresponding
$A$ totalizations.  This dual regrading makes the transpose differential
cohomological of degree $+1$, although it decreases geometric face
codimension.

For $Z\in\{A,B\}$ denote the resolved total field complex and the initial
intrinsic field complex by
\[
 (\mathscr F_L^Z,D^Z),\qquad
 (\mathscr F_X^Z,d_X^Z),
 \qquad D^Z=D_0^Z+\delta_Z.
\]
Here $d_X^Z$ has dual face degree zero, while $D_0^Z$ preserves intrinsic
support and may contain the normal-fibre cellular differential in the
appropriate cochain/chain grading.  The perturbation $\delta_A$ consists of
resolved restrictions and $\delta_B$ of the adjoint corestrictions.  On each
intrinsic summand, $d_X^Z$ is the $b$-de Rham differential together with the
coefficient-complex differential.  In particular, the intrinsic bulk
components are
\[
 \mathscr F_{X,\mathrm{bulk}}^A=\Omega_b^\bullet(X;V)[1],
 \qquad
 \mathscr F_{X,\mathrm{bulk}}^B=\Omega_b^\bullet(X;V^*)[n-2].
\]
These are field complexes; they are not the logarithmic DPP trace/descent
complexes $\mathscr A_L^\bullet$ and $\mathscr B_{X,\cV}^\bullet$ of
\textup{(G4)}.

The definition below is the continuum \(b\)-field datum.  One may formulate a
finite cellular analogue by replacing every displayed \(b\)-de Rham
\(A\)-complex with a finite cellular cochain complex and taking the shifted
algebraic dual for \(B\).  Such an analogue becomes a datum for the continuum
complexes above only after type-correct comparison maps and compatible
deformation retracts between the cellular and \(b\)-de Rham fields have been
constructed; algebraic duality alone does not provide them.

A \emph{BV-cyclic field datum} on a centre-compatible finite refinement tower consists
of the following data and identities.
\begin{enumerate}[label=\textup{(\alph*)},leftmargin=2.5em]
\item For each $Z\in\{A,B\}$ there are maps of dual face degrees $0,0,-1$,
respectively,
\[
 i_0^Z:\mathscr F_X^Z\longrightarrow\mathscr F_L^Z,
 \qquad
 p_0^Z:\mathscr F_L^Z\longrightarrow\mathscr F_X^Z,
 \qquad
 h_0^Z:\mathscr F_L^Z\longrightarrow\mathscr F_L^Z[-1]
\]
which are an initial special deformation retract:
\begin{align*}
 p_0^Zi_0^Z&=\id,
 &D_0^Zi_0^Z&=i_0^Zd_X^Z,
 &p_0^ZD_0^Z&=d_X^Zp_0^Z,\\
 D_0^Zh_0^Z+h_0^ZD_0^Z&=\id-i_0^Zp_0^Z,
 &h_0^Zi_0^Z&=0,
 &p_0^Zh_0^Z&=0,\qquad (h_0^Z)^2=0.
\end{align*}
\item The perturbations have degree $+1$ and the homotopies degree $-1$ in
their respective cochain/dual-chain total gradings; the products
$h_0^Z\delta_Z$ and $\delta_Zh_0^Z$ are nilpotent for the finite tower
filtration.  If $\jmath_F^Z$ and $\pi_F^Z$ denote inclusion and projection of
intrinsic face summands, incidence locality means, for every $r\geq0$,
\begin{align}
 \pi_G^Ap_0^A\delta_A(h_0^A\delta_A)^ri_0^A\jmath_F^A&=0
 &&\text{unless }G\prec F,\label{eq:g5-incidence-locality-A}\\
 \pi_F^Bp_0^B\delta_B(h_0^B\delta_B)^ri_0^B\jmath_G^B&=0
 &&\text{unless }G\prec F.\label{eq:g5-incidence-locality-B}
\end{align}
All $A$-maps commute with restrictions, all $B$-maps commute with
corestrictions, and the data glue along the refinement tower.
\item There are resolved and intrinsic families of regularized BF pairings,
written collectively as
\[
 \langle-,-\rangle_L:\mathscr F_L^A\otimes\mathscr F_L^B\longrightarrow\mathbb R,
 \qquad
 \langle-,-\rangle_X:\mathscr F_X^A\otimes\mathscr F_X^B\longrightarrow\mathbb R,
\]
whose component on a codimension-$k$ face has the standard BF degree $k-1$
(in particular, degree $-1$ on bulk BV fields).  The displayed maps are
collective bigraded notation, and the identities below are understood in each
bidegree.  In the shifted BF sign convention, the same symbol denotes the
Koszul-transposed pairing when the $B$-argument is written first.  The initial
differentials are graded adjoint:
\[
 \langle D_0^A a,b\rangle_L
 +(-1)^{|a|}\langle a,D_0^B b\rangle_L=0,
 \qquad
 \langle d_X^A x,y\rangle_X
 +(-1)^{|x|}\langle x,d_X^B y\rangle_X=0.
\]
\item Inclusion and projection are adjoint for these pairings:
for homogeneous elements, with the standard BF Koszul sign convention,
\[
 \langle i_0^A x,b\rangle_L=\langle x,p_0^B b\rangle_X,
 \qquad
 \langle a,i_0^B y\rangle_L=\langle p_0^A a,y\rangle_X.
\]
\item The homotopies satisfy the graded chain-adjoint relation
\begin{equation}\label{eq:g5-skew-adjoint}
 \langle h_0^A a,b\rangle_L
 =(-1)^{|a|}\langle a,h_0^B b\rangle_L,
\end{equation}
and the face perturbations are graded adjoint,
\begin{equation}\label{eq:g5-delta-adjoint}
 \langle \delta_A a,b\rangle_L
 +(-1)^{|a|}\langle a,\delta_B b\rangle_L=0.
\end{equation}
\item The pairings are nondegenerate on the chosen field class, restriction
and corestriction are adjoint, and all maps preserve the chosen complementary
boundary polarizations.  Along a flag $H\prec G\prec F$, the $A$ restrictions
compose contravariantly and the $B$ corestrictions compose covariantly:
\[
 r^A_{H/G}r^A_{G/F}=r^A_{H/F},
 \qquad
 (r^A_{G/F})_!(r^A_{H/G})_!=(r^A_{H/F})_!.
\]
Locality alone is not enough to identify the transferred terms with these
prescribed face arrows.  For \(G\prec F\), define the signed intrinsic
cochain arrow
\[
 \partial^A_{G/F}:=[F:G]\,r^A_{G/F}.
\]
Its signed BF-dual chain arrow
\(\partial^B_{F/G}:\mathscr F_G^B\to\mathscr F_F^B\) is the unique map in the
dual-chain grading satisfying, for homogeneous
\(a\in\mathscr F_F^A\) and \(b\in\mathscr F_G^B\),
\begin{equation}\label{eq:g5-signed-adjoint-arrow}
 \langle \partial^A_{G/F}a,b\rangle_G
 +(-1)^{|a|}\langle a,\partial^B_{F/G}b\rangle_F=0.
\end{equation}
The transfer data are required to recover exactly these signed arrows:
\begin{align}
 \pi_G^Ap_0^A\delta_A(1+h_0^A\delta_A)^{-1}i_0^A\jmath_F^A
 &=\partial^A_{G/F},\label{eq:g5-arrow-matching-A}\\
 \pi_F^Bp_0^B\delta_B(1+h_0^B\delta_B)^{-1}i_0^B\jmath_G^B
 &=\partial^B_{F/G}.\label{eq:g5-arrow-matching-B}
\end{align}
The inverse series are finite by the nilpotence requirement in \textup{(b)}.
\end{enumerate}
These conditions are the concrete content of \textup{(G5)}.  They are
additional to, rather than a pairing merely placed on, the cochain-level
hypotheses in \textup{(G4)}.  In particular, contractibility of the reduced
trace complex does not by itself imply that the reduced BF field sector is a
symplectic contractible pair.  A comparison between the G4 and G5 retracts,
when desired, is further input and is not part of either definition.
The output of a general \textup{(G5)} datum is therefore naturally bivariant:
contravariant in the $A$-fields and covariant in their BF-dual $B$-fields.
As with \textup{(G4)}, this is a chain-level input, not a general existence
theorem.  In particular, the arrow-matching identities are strictly stronger
than incidence locality: they prevent the equation
\((D_{\mathrm{str}}^Z)^2=0\) from being misread as pathwise functoriality of
otherwise unidentified transferred components.
\end{definition}

\subsection{Basic linear coefficients and the intrinsic graded sheaf}
\label{subsec:basic-linear-coefficients}

Let \(\cV^\bullet\) be a bounded complex of finite-rank flat vector bundles on
\(X\).  We call the induced coefficient system on a resolution \emph{basic}
if on every resolved face it is pulled back from the restriction of
\(\cV^\bullet\) to the support face under
\(q_L:\mathsf{Face}(X_L)\to\mathsf{Face}(X)\).  Thus the coefficients carry
no independent angular or radial degrees of freedom along a refinement
simplex.

Ignoring the face differential, define
\begin{equation}\label{eq:intrinsic-graded-sheaf}
 \mathscr B_{X,\cV}^{m}
 =\bigoplus_{F\in\mathsf{Face}(X)}
 \logOmega^{m-|F|}(F;\cV|_F).
\end{equation}
Its internal logarithmic and coefficient differential is denoted
\(d_{\log,\cV}\).  Before strictification there is no assumed intrinsic
face-degree-one operator on \eqref{eq:intrinsic-graded-sheaf}.

Fix a centre-compatible representative \(X_L\to X\) and write
\[
 \mathscr A_L^\bullet
 :=\mathscr A_{\cR_L,s_L}^\bullet(\cV).
\]
The maps \(i_0,p_0\) of \textup{(G4)} compare the intrinsic graded sheaf with
the entire resolved face sheaf.  Put
\[
 \mathscr E_{0,L}^\bullet=\ker p_0.
\]
Then
\[
 \mathscr A_L^\bullet
 =i_0\mathscr B_{X,\cV}^\bullet\oplus\mathscr E_{0,L}^\bullet
\]
as graded sheaves.  We call \(\mathscr E_{0,L}\) the \emph{initial reduced
refinement summand}.  It is a \(D_0\)-subcomplex, but it need not be preserved
by the full differential \(D\).

\subsection{The local reduced simplex model}

Consider one stellar blow-up with centre \(Z\) of codimension \(c\).  In a
compatible collar its normal part is the radial replacement
\[
 Q_\varepsilon
 \longleftarrow
 [0,\varepsilon)_r\times\Delta^{c-1}.
\]
After barycentric subdivision, let \(b\) denote the barycentre vertex.  The
augmented simplicial chain complex has the cone homotopy
\begin{equation}\label{eq:barycentric-cone}
 H_b[w_0,\ldots,w_q]
 =
 \begin{cases}
 [b,w_0,\ldots,w_q],&b\notin\{w_0,\ldots,w_q\},\\
 0,&b\in\{w_0,\ldots,w_q\},
 \end{cases}
\end{equation}
with the usual incidence sign, and on reduced chains
\[
 \partial H_b+H_b\partial=\id.
\]
The barycentre is fixed by every permutation of the vertices.  In the face
regrading $r=c-q$, the chain cone homotopy has face degree $-1$ and is
equivariant under the structure group of the normal simplex bundle.

\begin{lemma}[Reduced normal-face factor under \textup{(G4)}]\label{lem:exceptional-simplex}
For one centre-compatible stellar blow-up and basic linear coefficients which
are locally constant in the normal simplex direction, the associated-graded
initial reduced refinement sector supplied by the first alternative in
\textup{(G4)} is
\[
 j_{Z*}\left(
  \logOmega^\bullet(Z;\cV|_Z)
  \otimes
  \widetilde C_{\mathrm{face}}^\bullet(\sd\Delta^{c-1})
 \right).
\]
It is contracted by the identity on the coefficient factor tensored with the
barycentric cone homotopy.  The contraction is compatible with restriction in
the $Z$-direction and has face degree $-1$.
\end{lemma}

\begin{proof}
This is the locally constant case of the associated-graded clause in
\textup{(G4)}.  The normal-face
indexing is the augmented face complex of the subdivided simplex; aggregation
removes the coarse constant mode and leaves the face-regraded reduced cellular
chain complex.  The barycentric cone homotopy
\eqref{eq:barycentric-cone} contracts that factor, is equivariant under vertex
permutations, and therefore globalizes over the normal simplex bundle.  No
identification of differential forms with cellular chains is used.
\end{proof}

\begin{corollary}\label{cor:one-stage-contractible}
Under either alternative in the associated-graded clause of \textup{(G4)}, the
initial reduced refinement complex is contractible.  In the logarithmic case
its contraction is the composite
\[
 H_Z^{\log}+I_Z^{\log}H_bP_Z^{\log},
\]
with the tensor Koszul signs.  It lowers total face degree by one and commutes
with the internal differential on \(\cV^\bullet\).
\end{corollary}

\begin{proof}
Restrict the initial contraction identity
$D_0h_0+h_0D_0=\id-i_0p_0$ to $\ker p_0$.  There it becomes
$D_0h_0+h_0D_0=\id$.  In the locally constant case the
associated-graded contraction is the barycentric
cone homotopy of \cref{lem:exceptional-simplex}, which has face degree $-1$;
in the logarithmic case the standard composition formula for special
deformation retracts gives the displayed sum.  Compatibility with the
internal coefficient differential is part of \textup{(G4)}.
\end{proof}

\subsection{Finite perturbation and the actual reduced sector}\label{subsec:combining-tower}

For the finite tower, the one-stage contractions and the identity on the
coarse summand assemble, by \textup{(G4)}, into the initial contraction
\begin{equation}\label{eq:initial-global-sdr}
 (\mathscr A_L^\bullet,D_0)
 \underset{i_0}{\overset{p_0}{\rightleftarrows}}
 (\mathscr B_{X,\cV}^\bullet,d_{\log,\cV}),
 \qquad
 D_0h_0+h_0D_0=\id-i_0p_0.
\end{equation}
Write \(D=D_0+\delta\).  Nilpotence in \textup{(G4)} makes the following
operators finite sums:
\begin{equation}\label{eq:hpl-finite}
 h_\infty=h_0(1+\delta h_0)^{-1},\qquad
 i_\infty=(1+h_0\delta)^{-1}i_0,\qquad
 p_\infty=p_0(1+\delta h_0)^{-1}.
\end{equation}

\begin{proposition}[Filtered reduced-sector contraction]\label{prop:filtered-exceptional-contraction}
The maps in \eqref{eq:hpl-finite} define a strong deformation retract
\begin{equation}\label{eq:global-sdr}
 (\mathscr A_L^\bullet,D)
 \underset{i_\infty}{\overset{p_\infty}{\rightleftarrows}}
 (\mathscr B_{X,\cV}^\bullet,D_{\mathrm{str}}),
\end{equation}
where
\begin{equation}\label{eq:transferred-differential}
 D_{\mathrm{str}}
 =d_{\log,\cV}
  +p_0\delta(1+h_0\delta)^{-1}i_0.
\end{equation}
More precisely,
\begin{equation}\label{eq:perturbed-sdr-identities}
 Dh_\infty+h_\infty D=\id-i_\infty p_\infty,
 \qquad
 Di_\infty=i_\infty D_{\mathrm{str}},
 \qquad
 p_\infty D=D_{\mathrm{str}}p_\infty.
\end{equation}
The full-differential reduced sector
\[
 \mathscr E_{\infty,L}^\bullet:=\ker p_\infty
\]
is a \(D\)-subcomplex and is contracted by \(h_\infty\).  Finally,
\(D_{\mathrm{str}}\) has only face degrees zero and one.
\end{proposition}

\begin{proof}
With the convention
$D_0h_0+h_0D_0=\id-i_0p_0$, the basic perturbation lemma~\cite{GLS},
\cite[Sec.~2]{CrainicPerturb} is the
usual formula after the substitution $h_{\mathrm{std}}=-h_0$; hence the
resolvents are $(1+h_0\delta)^{-1}$ and $(1+\delta h_0)^{-1}$.  Applied to
\eqref{eq:initial-global-sdr}, it gives
\eqref{eq:global-sdr}--\eqref{eq:perturbed-sdr-identities}.  Because
\(p_\infty\) is a cochain map, its kernel is \(D\)-stable.  On that kernel,
\eqref{eq:perturbed-sdr-identities} becomes
\(Dh_\infty+h_\infty D=\id\).  For the face degree, every summand
\[
 (-1)^r p_0\delta(h_0\delta)^ri_0
\]
has degree \(+1\), since \(\delta\) raises face degree by one and \(h_0\)
lowers it by one.  Hence the transferred differential is the sum of its
face-degree-zero internal part and one face-degree-one operator.
\end{proof}

Write
\[
 D_{\mathrm{str}}=d_{\log,\cV}+\delta_{\mathrm{str}}.
\]
Then \(D_{\mathrm{str}}^2=0\) implies
\[
 d_{\log,\cV}\delta_{\mathrm{str}}
 +\delta_{\mathrm{str}}d_{\log,\cV}=0,
 \qquad
 \delta_{\mathrm{str}}^2=0.
\]
By the incidence-locality clause \eqref{eq:g4-incidence-locality},
$\delta_{\mathrm{str}}$ has a component from $F$ to $G$ only when $G\prec F$.
Thus it is a genuine strict linear incidence operator on the intrinsic face
poset.  This assertion concerns the incidence complex: without an additional
path-independence condition, it does not manufacture restriction morphisms for
all inclusions or a strict contravariant functor on the whole face category.
Its components need not be algebra
homomorphisms; no such multiplicativity is claimed in the linear theorem.

\begin{theorem}[Conditional cyclic HPL for an abelian BF pair]\label{thm:cyclic-hpl-bf}
Assume \textup{(G4)} together with the separate BV-cyclic field datum
\textup{(G5)} for the dual $A$- and $B$-field complexes.  For
$Z\in\{A,B\}$ set
\begin{align*}
 h_\infty^Z&=h_0^Z(1+\delta_Zh_0^Z)^{-1},&
 i_\infty^Z&=(1+h_0^Z\delta_Z)^{-1}i_0^Z,&
 p_\infty^Z&=p_0^Z(1+\delta_Zh_0^Z)^{-1},\\
 D_{\mathrm{str}}^Z
 &=d_X^Z+p_0^Z\delta_Z(1+h_0^Z\delta_Z)^{-1}i_0^Z.
\end{align*}
These are finite sums by \textup{(G5)} and give strong deformation retracts
\[
 (\mathscr F_L^Z,D^Z)
 \underset{i_\infty^Z}{\overset{p_\infty^Z}{\rightleftarrows}}
 (\mathscr F_X^Z,D_{\mathrm{str}}^Z).
\]
Each $D_{\mathrm{str}}^Z$ has only face degrees zero and one.  Its
face-degree-one components are exactly the prescribed signed restriction and
corestriction arrows by
\cref{eq:g5-arrow-matching-A,eq:g5-arrow-matching-B}; in particular they are
incidence-local in their respective cochain and dual-chain gradings.
The perturbed inclusions and
projections remain adjoint, the perturbed homotopies retain the graded
chain-adjoint relation above, and the transferred differentials are graded adjoint for the
intrinsic BF pairing.  Consequently the transferred quadratic action,
understood in the collective bigraded facewise sense of \textup{(G5)},
\[
 S_X^{\mathrm{BF,str}}(A,B)
 =\langle B,D_{\mathrm{str}}^A A\rangle_X
\]
defines a strict incidence-local classical linear BV--BFV/Stokes family on the
intrinsic field complexes: its bulk component is BV, and its positive
face-degree components are the corresponding BFV descendants.  Its
face-degree-one part satisfies the BFV/Stokes identity.  Together with the
flag identities in \textup{(G5)}, this incidence family is a strict bivariant
maximally extended face system: the $A$-fields restrict contravariantly and
the BF-dual $B$-fields corestrict covariantly.
\end{theorem}

\begin{proof}
Finiteness of the two field-level perturbation series is the nilpotence clause
of \textup{(G5)}; \textup{(G4)} controls the distinct logarithmic
trace/descent series.  Applied separately to the two initial retracts in
\textup{(G5)}, the basic perturbation lemma gives the displayed strong
deformation retracts and transferred differentials.  The homotopy and
perturbation adjointness relations
in \textup{(G5)} imply, for homogeneous $a$,
\[
 \langle h_0^A\delta_A a,b\rangle_L
 =\langle a,\delta_Bh_0^B b\rangle_L,
 \qquad
 \langle \delta_Ah_0^A a,b\rangle_L
 =\langle a,h_0^B\delta_B b\rangle_L.
\]
Therefore the finite inverse series $(1+h_0^A\delta_A)^{-1}$ and
$(1+\delta_Bh_0^B)^{-1}$ are adjoint term by term.  Together with the
initial adjointness of $i_0^A$ and $p_0^B$ this gives
\[
 \langle i_\infty^A x,b\rangle_L
 =\langle x,p_\infty^B b\rangle_X,
\]
and similarly with $A$ and $B$ exchanged.  Using the second displayed
identity in the same way shows that the perturbed homotopies satisfy the same
graded chain-adjoint relation as $h_0^A,h_0^B$.
Finally, the identities
$D^Zi_\infty^Z=i_\infty^ZD_{\mathrm{str}}^Z$ and the graded adjointness of
$D_0^A,D_0^B$ and $\delta_A,\delta_B$ assumed in \textup{(G5)} imply
\[
 \langle D_{\mathrm{str}}^A x,y\rangle_X
 +(-1)^{|x|}\langle x,D_{\mathrm{str}}^B y\rangle_X=0.
\]
Nondegeneracy is part of \textup{(G5)}, so the intrinsic pairing defines the
linear BV symplectic structure.  Taking the face-degree-one component of the
last cyclicity identity gives the strict BFV/Stokes identity.  The matching
identities
\cref{eq:g5-arrow-matching-A,eq:g5-arrow-matching-B} identify those
components with the signed restrictions and corestrictions, while the flag
identities in \textup{(G5)} give their strict composition along arbitrary
face flags.  Hence the transferred incidence family is the claimed bivariant
face system.  No identification
with the logarithmic retract of
\textup{(G4)} is used in this argument.
\end{proof}

\subsection{Geometric existence of stellar towers}\label{subsec:g4-existence}

Hypothesis \textup{(G4)} has a geometric part and a chain-level part.  The
geometric part can be arranged by a finite barycentric stellar refinement of
an already smooth monoidal resolution, and doing so changes none of the
unfiltered derived objects already constructed.  The chain-level data
\textup{(a)}--\textup{(e)} of
\cref{def:face-compatible-stellar} are genuine additional input, and we
delimit below exactly which of them the following proposition supplies.

\begin{proposition}[Existence of unimodular stellar towers]\label{prop:g4-existence}
Let $X$ be compact and satisfy \textup{(G1)}, and let $\beta_{\cR}:X_{\cR}\to X$
be any smooth monoidal resolution.  There is a finite composition of
elementary radial blow-ups
\[
 \gamma:X_{\cR^\sharp}\longrightarrow X_{\cR}
\]
in which every stage is an elementary radial blow-up of an embedded
ordinary-corner face along a unimodular centre $Z$, equipped with a product
collar and a globally defined exceptional simplex bundle with fibre
$\Delta^{c-1}$, $c=\operatorname{codim}Z$.  Taking the star subdivisions in
decreasing dimension gives the barycentric normal cellulation
$\sd\Delta^{c-1}$ over every original centre.  Since $\gamma$ is a finite
stellar refinement, \cref{thm:refinement-invariance,prop:sheaf-refinement}
identify the derived objects of
\cref{sec:derived-traces,sec:global-descent} attached to $X_{\cR^\sharp}$ with
those attached to $X_{\cR}$, through the refinement roof.
\end{proposition}

\begin{proof}
Since $\cR$ is smooth, every cone of $\cP_{X_{\cR}}$ is already unimodular, and
the construction below is ordinary, not generalized, corner geometry
throughout; the Joyce structure of $X$ enters only through $\beta_{\cR}$
itself, which is already fixed.

For a unimodular cone $\tau=\Cone(w_1,\ldots,w_k)$, put $w_0=w_1+\cdots+w_k$
and, for $i=1,\ldots,k$, let
$\tau_i=\Cone(w_1,\ldots,w_{i-1},w_0,w_{i+1},\ldots,w_k)$ be the cone with
$w_i$ replaced by $w_0$ --- one of the $k$ cones obtained by star-subdividing
$\tau$ at the ray through $w_0$, and geometrically the same cone regardless of
where $w_0$ is listed among its $k$ generators.  The change of basis from
$(w_1,\ldots,w_k)$ to $(w_1,\ldots,w_{i-1},w_0,w_{i+1},\ldots,w_k)$ is right
multiplication by the matrix $T_i$ agreeing with the identity off column $i$,
with column $i$ equal to $(1,\ldots,1)^{\mathsf T}$.  Expanding $\det T_i$
along column $i$: the $(i,i)$-cofactor is the identity minor, contributing
$1$; every other cofactor in that column has a zero column, since deleting row
$j\ne i$ removes the unique nonzero entry of the standard basis column $e_j$
appearing elsewhere in $T_i$.  Hence $\det T_i=1$ and every $\tau_i$ is again
unimodular.  Perform this stellar step on all cones of dimension at least two in
$\cP_{X_{\cR}}$ in decreasing order of dimension.  A new ray lies in the
relative interior of the cone being subdivided, so the step restricts to the
identity on its proper faces; consequently steps made on adjacent maximal
cones agree on their intersection.  The resulting compatible fan is the
barycentric subdivision.  Since $X$ is compact, $\cP_{X_{\cR}}$ is finite, so
the process terminates and produces a smooth refinement
$\cR^\sharp\to\cR$ that is a finite composition of elementary stellar
subdivisions of already-smooth cones.

By \cref{thm:kottke}, each combinatorial stellar step is realized by a blow-up
of the current stage along the embedded face dual to the subdivided cone;
since that cone is unimodular, the blow-up is precisely the elementary radial
blow-up $F\times[0,\varepsilon)^k\to F\times[0,\varepsilon)\times\Delta^{k-1}$
of an \emph{ordinary} corner face $F$ of an ordinary-corners manifold.  The
geometric clauses are therefore ordinary differential topology on $X_{\cR}$,
with no further $g$-corner input: $F$ has a product collar by the tubular
neighbourhood theorem for embedded faces of a manifold with ordinary
corners~\cite{Melrose}, and its normal bundle is a genuine rank-$k$ vector
bundle, so the exceptional simplex bundle over $F$ is globally defined with
fibre $\Delta^{k-1}$.  The decreasing-dimensional sequence induces the stated
barycentric cellulation on the final normal fibre.  Finally, $\gamma$ is by
construction a finite
composition of elementary radial blow-ups, hence a finite stellar refinement in
the sense used throughout Parts~II--III, and the stated identification is
\cref{thm:refinement-invariance} affinely and \cref{prop:sheaf-refinement}
globally.
\end{proof}

\begin{remark}[Exactly what the proposition supplies]\label{rem:g4-existence-scope}
\Cref{prop:g4-existence} establishes the geometric hypotheses of
\cref{def:face-compatible-stellar}: embedded centres, product collars, and a
globally defined finite contractible cell bundle with the ordinary stellar
fibre $K_Z=\sd\Delta^{c-1}$.  It also supplies the canonical barycentric cone
homotopy \eqref{eq:barycentric-cone} on the \emph{reduced cellular chain
factor}, once a cellular normal model has already been chosen.

It supplies none of the chain-level maps $i_0,p_0,h_0$ for the full DPP
logarithmic form complex.  In particular, ``basic coefficients'' in
\cref{subsec:basic-linear-coefficients} means that the flat coefficient bundle
is pulled back from the support face; it does not make the logarithmic
\emph{form} factor constant in the radial or angular variables.  The
locally-constant alternative in clause~\textup{(c)} applies only to a normal
model which is already cellular.  For the full resolved logarithmic complex a
Whitney--Dupont or other cellularization retract of the total normal-face
complex, compatible in the centre and under gluing, must still be supplied
explicitly.  Finiteness of the geometric tower likewise
does not by itself prove nilpotence of $h_0\delta$ or $\delta h_0$; those
operators and the required filtration behavior first have to be defined.

The initial resolution $\beta_{\cR}$ is not altered by the proposition and may
come from a genuinely non-simplicial cone, as in
\cref{prop:pentagon-resolve}.  In that situation
\cref{lem:exceptional-acyclic} proves only that the reduced \emph{cellular}
fibre is acyclic.  It does not globalize a contraction of logarithmic forms or
establish its compatibility with centre restrictions and gluing.  Incidence locality
\eqref{eq:g4-incidence-locality}, logarithmic cellularization, gluing, and all
perturbed identities therefore remain genuine parts of \textup{(G4)}.

The restricted conclusion is that the geometric cellular presentation can be
reached by a further smooth refinement without changing the unfiltered derived
object.  It remains open whether clauses~\textup{(a)}--\textup{(e)} can be
constructed for every basic coefficient complex.  The product--Whitney
conifold class below is a separate finite example in which the necessary
logarithmic cellularization and perturbation data are written down explicitly.
\end{remark}

\subsection{Transferred trace and the linear strictification theorem}

The deformation retract is defined for every basic linear coefficient complex
in the specified \textup{(G4)} transfer class,
but a scalar integration functional requires additional coefficient data.  For the
trace statements in this subsection, assume either $\cV=\underline{\R}$ or that a
compatible degree-zero morphism of flat coefficient complexes
\[
 \tau_{\cV}:\cV^\bullet\longrightarrow\underline{\R}
\]
has been fixed.  Apply $\tau_{\cV}$ coefficientwise before the transported
DPP integration of \eqref{eq:derived-trace}, and denote the resulting resolved trace by
$\mathcal I_{\cR,s_{\cR},\tau_{\cV}}$.  In abelian BF theory the scalar functional is
instead obtained after the canonical evaluation pairing
$V^*\otimes V\to\R$; no scalar trace on either field complex separately is implied.

Define
\begin{equation}\label{eq:strict-trace-transfer}
 \mathcal I_{X,\mathfrak s,\tau_{\cV}}^{\mathrm{str}}
 :=\mathcal I_{\cR,s_{\cR},\tau_{\cV}}\circ i_\infty.
\end{equation}

\begin{proposition}[Strict linear Stokes identity]\label{prop:strict-stokes-global}
The functional \eqref{eq:strict-trace-transfer} is a cochain map:
\[
 \mathcal I_{X,\mathfrak s,\tau_{\cV}}^{\mathrm{str}}
 \circ D_{\mathrm{str}}=0.
\]
The corresponding face-degree-one part of this cochain-map identity is a strict
linear Stokes identity on
\(\mathsf{Face}(X)\).
\end{proposition}

\begin{proof}
By \eqref{eq:perturbed-sdr-identities}, \(i_\infty\) is a cochain map.
Therefore
\[
 \mathcal I^{\mathrm{str}}D_{\mathrm{str}}
 =\mathcal I_{\cR,s_{\cR},\tau_{\cV}}i_\infty D_{\mathrm{str}}
 =\mathcal I_{\cR,s_{\cR},\tau_{\cV}}Di_\infty=0.
\]
The face-degree statement follows from
\cref{prop:filtered-exceptional-contraction}.
\end{proof}

\begin{theorem}[Filtered linear transfer under \textup{(G4)}]\label{thm:global-linear-strict}
Let $X$ satisfy \textup{(G1)}--\textup{(G4)}.  Then basic linear coefficients in
the specified transfer class admit a global strong
deformation retract from the resolved total face complex to a strict intrinsic
linear face complex.  Whenever the coefficient system carries a compatible scalar
trace as above, its transferred trace satisfies strict linear Stokes.
\end{theorem}

\begin{proof}
The deformation retract is \cref{prop:filtered-exceptional-contraction}.  When a
compatible scalar trace is present, its Stokes statement is
\cref{prop:strict-stokes-global}.  No independence statement for two different
choices of \textup{(G4)} data is needed here.
\end{proof}

\begin{remark}[What \textup{(G4)} proves, and what it does not]\label{rem:g4-scope}
\Cref{thm:global-linear-strict} is a conditional transfer theorem for the
linear logarithmic trace/descent complex.  It neither proves existence of the
\textup{(G4)} contraction data nor, by itself, equips the transferred complex
with a nondegenerate BV pairing.  In particular, the logarithmic DPP complex
used to regularize traces is not being identified with the $b$-de Rham field
complex.
\end{remark}

\begin{remark}[Conditional status of the strictification theorems]
In general, the global assertions in
\cref{thm:global-linear-strict,thm:cyclic-hpl-bf} are implication theorems:
the subdivision/aggregation retract, its incidence locality, the two cyclic
field retracts, and their nondegenerate pairing are hypotheses, not outputs of
Kottke's resolution theorem or of DPP regularization.  The conifold calculation
below is a separate explicit finite result: it constructs the
\textup{(G4)} data on a specified product--Whitney coefficient class and a
cyclic logarithmic--cellular BF shadow on its algebraic dual.  It does not
construct the \(b\)-de Rham field retracts required by \textup{(G5)}, either
for the conifold or for a general generalized corner.
\end{remark}

\begin{corollary}[Global incidence-level classical abelian BF under \textup{(G5)}]\label{cor:global-strict-bf}
Let $V$ be a finite-dimensional graded vector space with differential, or
more generally a flat finite-rank complex on $X$.  Assume
\textup{(G1)}--\textup{(G4)} and the separate \textup{(G5)} datum for the dual
BF field systems.  Then abelian BF
theory has a global strict intrinsic classical linear field system whose bulk
complex is
\[
 \Omega_b^\bullet(X;V)[1]
 \oplus
 \Omega_b^\bullet(X;V^*)[n-2],
\]
with the nondegenerate transferred BF pairing and strict incidence-level
BV--BFV/Stokes descendants indexed by codimension-one arrows of
$\mathsf{Face}(X)$.  With the arrow-matching and flag identities required in
\textup{(G5)}, it is a strict bivariant maximally extended face system:
contravariant for $A$
and covariant for the BF-dual $B$.  Forgetting the pairing leaves the field-level
linear retract supplied by \textup{(G5)}; it does not recover or identify the
logarithmic trace/descent retract of \cref{thm:global-linear-strict} unless an
additional comparison datum is supplied.
\end{corollary}

\begin{proof}
Apply \cref{thm:cyclic-hpl-bf} to the dual $A/B$ field retracts supplied by
\textup{(G5)}.  The transferred differentials are cyclic for the
nondegenerate BF pairing and the face-degree-one operator satisfies strict
Stokes.  The bulk identification with $\Omega_b$ is part of the field-level
input in \textup{(G5)}, not a consequence of the DPP logarithmic de Rham
theorem.  Thus this proof uses the field-level retract and makes no
type-changing identification with the \textup{(G4)} retract.
\end{proof}

\subsection{The first nonlinear obstruction}

The preceding theorem fails to imply strict nonlinear descent because an interaction is not merely a differential.  Let the resolved coefficient complex carry a differential graded Lie bracket.  Homotopy transfer along \cref{eq:global-sdr} gives an $L_\infty$ structure on the intrinsic complex.  With Koszul symmetrization understood, its first operations are
\begin{align}
 \ell_1(x)&=D_{\mathrm{str}}x,\\
 \ell_2(x,y)&=p_\infty[ i_\infty x,i_\infty y],\label{eq:transferred-l2}\\
 \ell_3(x,y,z)&=
 \sum_{\mathrm{cyc}}\!\epsilon(x,y,z)\,
 p_\infty\bigl[
 h_\infty[i_\infty x,i_\infty y],i_\infty z
 \bigr].\label{eq:transferred-l3}
\end{align}
Here $\epsilon$ is the Koszul sign of the permutation of the inputs, and the
convention is the one attached to $Dh_\infty+h_\infty D=\id-i_\infty p_\infty$;
we use the standard rooted-tree convention for homotopy transfer~\cite{LodayVallette}.
The displayed normalization fixes the convention in this paper; suspended
conventions may change the corresponding overall sign, and no later argument
uses that sign rather than the vanishing class of the operation.  Higher brackets are
sums over rooted trees with internal edges labelled by $h_\infty$.

\begin{proposition}[First nonlinear transfer diagnostic]\label{prop:first-nonlinear-obstruction}
The ternary operation \cref{eq:transferred-l3} is the first possible higher
operation in the transferred interaction and vanishes identically for abelian
targets.  For a fixed contraction, a nonzero value shows that this transferred
representative is not strict.  It is not by itself an invariant obstruction:
the strictification obstruction is the class of the higher operation modulo
$L_\infty$ gauge equivalence.
\end{proposition}

\begin{proof}
The homotopy-transfer theorem~\cite{LodayVallette} gives the displayed operations and the $L_\infty$ identities.  A strict differential graded Lie model has no operations of arity at least three, so the first possible discrepancy is $\ell_3$.  For an abelian target the bracket is zero, and every tree with an internal bracket vanishes.  In the nonabelian case, changing the contraction changes the transferred structure by an $L_\infty$ isomorphism; therefore nonvanishing of one formula is not by itself an invariant obstruction, but failure to remove its class by such an isomorphism is precisely the cubic strictification obstruction.
\end{proof}

\begin{remark}
The proposition is a general warning, not a calculation for the conifold.  A
direct nonlinear interval computation is needed before declaring $\ell_3$
nontrivial, and the next subsection carries it out for the exceptional
interval, where the transferred higher operations turn out to vanish on the
basic sector.  That computation is finite-dimensional and cellular; it is not
required for the linear comparison proved above.
\end{remark}

\subsection{Exact nonabelian compression of the exceptional interval}\label{sec:interval-compression}

We use the cellular interval \eqref{eq:interval-complex} of
\cref{sec:exceptional-interval}.  The canonical cellular nonabelian BF action on a $1$-simplex was computed explicitly in \cite{CMRCellular}.  We use the same notation for the graded $\mathfrak g$-valued $A$-fields $A_0,A_1,A_{01}$ and the dual edge field $B_{01}$.

\begin{proposition}[Canonical interval block]\label{prop:canonical-interval-block}
For a finite-dimensional unimodular Lie algebra $\mathfrak g$, the contribution of the open $1$-cell to the cellular quantum BF action is
\begin{align}
 \overline S_I={}&\left\langle B_{01},
 \frac12[A_{01},A_0+A_1]
 +F(\ad_{A_{01}})(A_1-A_0)\right\rangle \\
 &\hspace{2cm}-i\hbar\,\trg\log G(\ad_{A_{01}}),\label{eq:interval-action}
\end{align}
where
\[
 F(z)=\frac z2\coth\frac z2,
 \qquad
 G(z)=\frac{2}{z}\sinh\frac z2.
\]
Equivalently, the associated unimodular $L_\infty$ operations have the following support property: every operation of arity greater than two contains at least one edge input $\epsilon_{01}$, and every quantum operation contains only edge inputs.
\end{proposition}

\begin{proof}
This is the $1$-simplex building block of cellular BF theory.  Expanding the two analytic functions gives the Bernoulli series
\[
 F(z)=1+\frac{z^2}{12}-\frac{z^4}{720}+\cdots,
 \qquad
 \log G(z)=\frac{z^2}{24}-\frac{z^4}{2880}+\cdots.
\]
The explicit operations are
\[
 \ell_2(\alpha\epsilon_i,\beta\epsilon_i)
 =[\alpha,\beta]\epsilon_i,
 \qquad i=0,1,
\]
and, for $n\geq0$, the only further classical operations have $n$ edge inputs and one endpoint input, with coefficients $B_n^\pm/n!$.  The quantum operations $q_n$ have $n$ edge inputs and coefficients $B_n/(n\,n!)$.  Resumming these Bernoulli series gives \cref{eq:interval-action}; see \cite[Rem.~8.4 and Eqs.~(111)--(112)]{CMRCellular}.
\end{proof}

The constant and reduced subcomplexes
$C_{\mathrm{const}}^\bullet(I;\mathfrak g)$ and
$C_{\mathrm{red}}^\bullet(I;\mathfrak g)$, and the symmetric contraction
\eqref{eq:interval-contraction} of the second, are those of
\cref{sec:exceptional-interval}; the contraction is
\cref{prop:diagonal-interval} tensored with $\mathfrak g$.

\begin{theorem}[Exact basic-sector compression]\label{thm:interval-exact-compression}
The restriction of the canonical interval $qL_\infty$ structure to $C_{\mathrm{const}}^\bullet(I;\mathfrak g)$ is the strict unimodular Lie algebra $\mathfrak g$:
\[
 \ell_2(\alpha,\beta)=[\alpha,\beta],
 \qquad
 \ell_n=0\ (n\neq2),
 \qquad
 q_n=0\ (n\geq1).
\]
Moreover, BV pushforward along the acyclic complement $C_{\mathrm{red}}^\bullet(I;\mathfrak g)$ sends the full cellular BF theory of the interval, including its endpoint blocks, to the cellular BF theory of a point, up to quantum canonical transformation.  Consequently the BV cohomology class of the partition function is unchanged by collapsing the exceptional interval.
\end{theorem}

\begin{proof}
The constant cochain $\epsilon_0+\epsilon_1$ contains no edge component.  By the support property in \cref{prop:canonical-interval-block}, every operation of arity at least three and every quantum operation vanishes on constant inputs.  The two endpoint binary brackets add to the constant cochain, so the remaining binary operation is precisely the original bracket of $\mathfrak g$.  Thus the inclusion of constants is a strict unimodular $L_\infty$ morphism.

The quotient by constants is the two-term acyclic complex in \cref{eq:interval-contraction}.  The collapse $I\searrow *$ is an elementary simple-homotopy equivalence.  Compatibility of canonical cellular BF actions with elementary collapses and cellular aggregations implies that the corresponding BV pushforward gives the point action up to a quantum canonical transformation; the partition function changes only by a BV-exact term.  This is the elementary-collapse case of the subdivision and aggregation theorem of \cite{CMRCellular}.  The general fact that an interacting BV pushforward along an acyclic ultraviolet sector is a quasi-isomorphism is also established abstractly in \cite{CMRPushforward}.
\end{proof}

\begin{corollary}[No cubic obstruction for the conifold interval]\label{cor:no-cubic-conifold}
For the isolated exceptional interval in either small resolution of the positive real conifold, the transferred ternary operation on the intrinsic basic sector is zero.  More generally, all transferred operations of arity at least three and all quantum operations vanish after full compression of that interval.
\end{corollary}

\begin{proof}
The normal exceptional fibre is a single interval and the intrinsic basic sector is its constant subcomplex.  Apply \cref{thm:interval-exact-compression}.
\end{proof}

The preceding statement must be distinguished from a relative BV--BFV description in which the edge mode is retained.  Put
\[
 a=A_{01}.
\]
Variation of \cref{eq:interval-action} with respect to $B_{01}$ gives
\begin{equation}\label{eq:interval-transport-equation}
 \frac12[a,A_0+A_1]
 +F(\ad_a)(A_1-A_0)=0.
\end{equation}
Formally solving this equation yields
\begin{equation}\label{eq:interval-parallel-transport}
 A_1=e^{-\ad_a}A_0.
\end{equation}
Thus $a$ records angular parallel transport between the two endpoint fields.  The corresponding one-loop contribution is
\begin{equation}\label{eq:interval-one-loop}
 S_I^{(1)}(a)
 =-i\hbar\,\trg\log\left(
 \frac{2}{\ad_a}\sinh\frac{\ad_a}{2}
 \right)
 =-i\hbar\left(
 \frac1{24}\trg(\ad_a^2)
 -\frac1{2880}\trg(\ad_a^4)+\cdots
 \right).
\end{equation}

\begin{proposition}[Basic versus retained-holonomy sectors]\label{prop:two-interval-sectors}
The holonomy relation \cref{eq:interval-parallel-transport} and the determinant \cref{eq:interval-one-loop} are data of the relative sector in which $a$ is retained.  They do not obstruct the basic intrinsic compression, which integrates out $a$ together with the endpoint-difference mode and is quantum-canonically equivalent to the point theory.
\end{proposition}

\begin{proof}
If $a$ is retained, \cref{eq:interval-action} is the effective interval block and its $B$-equation is exactly \cref{eq:interval-transport-equation}; the elementary identity
\[
 \frac{\coth(z/2)-1}{\coth(z/2)+1}=e^{-z}
\]
gives \cref{eq:interval-parallel-transport}.  The one-loop term is the second line of \cref{eq:interval-action}.  In the basic sector, however, $a$ belongs to the acyclic complement and is integrated out.  \cref{thm:interval-exact-compression} identifies that pushforward with the point theory up to canonical transformation.  The two procedures correspond to different choices of residual fields and should not be conflated.
\end{proof}

\part*{IV.\ BF theory and local extensions}\addcontentsline{toc}{part}{IV.\ BF theory and local extensions}

\section{Abelian BF theory}\label{sec:abelian-bf}

Abelian BF theory is the smallest target for which both the ordinary
resolution-level and intrinsic $b$-formulations can be written explicitly.
The statements below must nevertheless be separated: ordinary BF is shown to
be invariant among admitted resolutions, while intrinsic $b$-BF satisfies its
own regularized BV--BFV identity.  Their source complexes are not declared
equivalent.  The strict resolved-$b$/intrinsic-$b$ comparison uses the separate
field-level datum \textup{(G5)} in \cref{thm:cyclic-hpl-bf}.  Since the theory
is quadratic, the pushforwards that are actually asserted below are Gaussian
and require no graph expansion.

\subsection{Linear BV data}

Let $V$ be a finite-dimensional graded vector space with graded dual $V^*$.
The $n$-dimensional abelian BF theory has superfields $A$ and $B$ of
complementary degree, and the same expressions
\[
 \omega=\int\langle\delta B,\delta A\rangle,
 \qquad
 S=\int\langle B,\dd A\rangle
\]
define the theory on either source, once the differential and the
integration functional are fixed.  The two formulations use different data:

\begin{center}
\small
\begin{tabular}{>{\raggedright\arraybackslash}p{0.19\textwidth}
                >{\raggedright\arraybackslash}p{0.35\textwidth}
                >{\raggedright\arraybackslash}p{0.35\textwidth}}
\toprule
 & ordinary resolution formulation & intrinsic $b$-formulation\\
\midrule
source & $N$, ordinary corners, typically $N=X_\cR$ & $X$, generalized corners\\
fields & $\Omega^\bullet(N;V)[1]\oplus\Omega^\bullet(N;V^*)[n-2]$
        & $\bOmega^\bullet(X;V)[1]\oplus\bOmega^\bullet(X;V^*)[n-2]$\\
differential & de Rham $\dd$ & $b$-de Rham $\db$\\
integration & $\int_N$ & regularized trace $\Tr_X^{\cT}$\\
face terms & ordinary restriction & regularized restriction $\Reg_{H/X}$\\
\bottomrule
\end{tabular}
\end{center}

Explicitly, on the resolution side
\begin{equation}\label{eq:bf-fields}
  \F_N^{\mathrm{BF}}
  =\Omega^\bullet(N;V)[1]
   \oplus
   \Omega^\bullet(N;V^*)[n-2],
\end{equation}
with $Q(A,B)=(\dd A,\dd B)$, and the facewise descendants are obtained by
restriction and Stokes' formula.  The underlying linear BV complex is
\begin{equation}\label{eq:bf-linear-complex}
  \mathbb F_N
  =\bigl(\Omega^\bullet(N;V)[1]\oplus
  \Omega^\bullet(N;V^*)[n-2],\dd\bigr),
\end{equation}
with cohomology $H^\bullet(N;V)[1]\oplus H^\bullet(N;V^*)[n-2]$.

\subsection{The intrinsic regularized identity}

\begin{definition}
On a regularized $g$-corner $(X,\cT)$ define
\begin{align}
 \omega_X^{\mathrm{BF}}&=\Tr_X^{\cT}\langle\delta B,\delta A\rangle,\\
 S_X^{\mathrm{BF}}&=\Tr_X^{\cT}\langle B,\db A\rangle.
\end{align}
\end{definition}

\begin{theorem}[Regularized BF identity]\label{thm:bf}
The cohomological vector field
\[
 Q A=\db A,\qquad Q B=\db B
\]
satisfies
\[
 \iota_Q\omega_X^{\mathrm{BF}}
 =(-1)^n\delta S_X^{\mathrm{BF}}
 +\sum_{H\prec X}[X:H]\,
 \Tr_H^{\cT}\Reg_{H/X}\langle B,\delta A\rangle.
\]
The analogous identity holds on every intrinsic face.
\end{theorem}

\begin{proof}
Abelian BF is the cotangent linear AKSZ target, and the displayed symplectic
form and action are the specialization of \cref{def:transgression}.  Apply
\cref{thm:bvbfv} with $F=X$.  Equivalently, decompose $A$ and $B$ into
homogeneous source-degree components and apply
\eqref{eq:regularized-transgression-stokes} to each complementary pair.  The
bulk components sum to $\iota_Q\omega_X^{\mathrm{BF}}$, while regularized
Stokes gives the displayed face sum.  This componentwise calculation avoids
assigning one degree $|B|$ to the inhomogeneous superfield $B$.
\end{proof}

\begin{remark}
The theory is linear, but the need for regularized face restriction is
already visible: a component of $A$ or $B$ proportional to
$\vartheta_p=\db\log\lambda_p$ has no ordinary pullback to a face on which
$\lambda_p$ vanishes.
\end{remark}

\subsection{Bulk invariance on resolutions}

\begin{theorem}[Bulk abelian BF cochain invariance]\label{thm:bf-bulk}
Let $\beta:\widetilde N\to N$ be a finite composition of elementary radial
blow-ups of ordinary corner faces.  Pullback is a quasi-isomorphism of the
underlying linear BF cochain complexes
\[
  \beta^*:\mathbb F_N\longrightarrow\mathbb F_{\widetilde N}.
\]
Consequently, the classical linear on-shell spaces of abelian BF theory are
canonically isomorphic.
\end{theorem}

\begin{proof}
Pullback commutes with the de Rham differential and acts diagonally on the
two summands in \eqref{eq:bf-linear-complex}.  By
\cref{cor:de-rham-blowup}, it is an isomorphism on the cohomology
of each de Rham summand.  Hence it is a quasi-isomorphism of the direct sum.
\end{proof}

\begin{remark}
The theorem compares the underlying linear BF cochain complexes.  The
integrated symplectic
forms are related more subtly, because integration over manifolds with
different boundary decompositions is not literally preserved by pullback on
all representatives.  On cohomology with compatible compact-support or
relative conditions, the pairings agree by change of variables on the common
interior and by Stokes control of boundary terms.  A full statement at the
symplectic-chain level requires fixing those conditions.  This is one of the
places where the intrinsic route is cleaner: \cref{thm:bf} is an identity of
regularized traces on a single source, with no comparison of boundary
decompositions to control.
\end{remark}

\subsection{Facewise invariance}

For every face $G$ of $N$, use the coefficient complex
\[
  \cA_N(G)=
  \Omega^\bullet(G;V)[1]
  \oplus
  \Omega^\bullet(G;V^*)[n-2].
\]
Restriction of forms gives the contravariant maps.

For a radial blow-up along $F$, faces not meeting $F$ are unchanged.  Over
$F$, the new faces are products of strata of $F$ with faces of the simplex.
On a sufficiently small product collar, the normal-simplex direction is
contractible.

\begin{theorem}[Facewise abelian BF comparison]\label{thm:bf-facewise}
Let $\beta:\widetilde N\to N$ be an elementary radial blow-up of a
product-type ordinary corner face.  Assume $N$ is compact, or work with
compact supports in a fixed truncation, and use the relative totalizations of
the facewise BF complexes.  Then there is a canonical roof of
quasi-isomorphisms
\[
 \Tot(N;\cA_N)
 \xleftarrow{\simeq}
 \mathbb F_c(N^\circ)
 \xrightarrow[\cong]{(\beta^\circ)^*}
 \mathbb F_c(\widetilde N^\circ)
 \xrightarrow{\simeq}
 \Tot(\widetilde N;\cA_{\widetilde N}),
\]
where $\mathbb F_c(M)=\Omega_c^\bullet(M;V)[1]\oplus
\Omega_c^\bullet(M;V^*)[n-2]$.  Thus the two facewise BF totalizations are
canonically isomorphic in the derived category.  No filtered refinement map
$U_\beta$ is asserted.
\end{theorem}

\begin{proof}
For an ordinary-corner manifold the alternating totalization over all faces
is the smooth relative de Rham model; by the collar/excision argument it is
quasi-isomorphic to compactly supported forms on the interior, exactly as in
the smooth part of \cref{lem:interior-bridge}.  The blow-down restricts to an
orientation-preserving diffeomorphism
$\widetilde N^\circ\cong N^\circ$, giving the middle isomorphism.  Tensor
with $V$ and $V^*$ and add the BF shifts.  The local reduced simplex model
\cref{prop:simplex-acyclic} explains the same comparison geometrically, but
its acyclicity alone is not used to manufacture a filtered map $U_\beta$.
\end{proof}

\begin{corollary}\label{cor:bf-compositions}
The derived-category comparison of \cref{thm:bf-facewise} is stable under a
finite composition of elementary radial blow-ups.  Two resolutions connected
through a common refinement by such moves have canonically isomorphic
relative abelian BF descent totalizations in the derived category.
\end{corollary}

\subsection{What is invariant, and where comparison data enter}

The individual BFV or $\mathrm{BF}^k\mathrm V$ spaces attached to faces are
not invariant term by term.  A resolution introduces new faces and therefore
new descendants.  The invariant object in
\cref{thm:bf-facewise} is the \emph{total derived descent system};
the exceptional descendants are invisible in this derived object, in the
sense that the relative difference is acyclic through the common-interior
roof.  This statement does not yet give a literal filtered contraction of a
geometrically selected exceptional subcomplex.  Under \textup{(G4)} the
actual full-differential reduced sector is instead $\ker p_\infty$, and that
sector is contracted by $h_\infty$.  This is analogous to replacing one cell
by a cellular subdivision: the list of cells changes, while the derived
cellular object changes only by an acyclic enlargement.

The two formulations therefore produce different objects.  The ordinary
resolution formulation produces, for each smooth refinement $\cR$, a total
descent complex $\Tot(X_\cR;\cA)$, well defined up to the quasi-isomorphisms
of \cref{cor:bf-compositions}.  The intrinsic route produces a single
complex on the intrinsic face category $\mathsf{Face}(X)$, with the strict identity
\cref{thm:bf}.  These facts alone do not identify the ordinary de Rham source
with the intrinsic $b$-source.  The following two transfers concern instead
the resolved $b$/logarithmic and intrinsic $b$ models and are related but
type-distinct.
Under \textup{(G4)} the
full-differential reduced \emph{trace/descent} refinement sector $\ker p_\infty$ contracts by
\cref{prop:filtered-exceptional-contraction}.  Separately, a strict BV--BFV
field theory requires the cyclic field-level retracts in \textup{(G5)}, as in
\cref{thm:cyclic-hpl-bf,cor:global-strict-bf}.  They are not obtained by
forgetting structure on the logarithmic retract; an identification requires
an additional comparison datum.  The resolved $b$/logarithmic route gives a
computable representative; the intrinsic $b$-route gives a strict model only
after the multiplicative trace datum and the conditional hypotheses
\textup{(G4)}--\textup{(G5)} have been supplied.
\Cref{thm:conifold-bf} gives the unconditional derived comparison in the
first non-ordinary example, while \cref{thm:conifold-trace} separates its
derived trace statement from the additional \textup{(G4)}--\textup{(G5)}
input required for strictification.

\section{Local continuum and cellular extensions}\label{sec:local-extensions}

The preceding abelian comparison is the field theory needed for the main strictification
result.  We record two further tests of the resolution mechanism: a finite-dimensional
nonabelian cellular shadow and an explicit codimension-two continuum propagator.  The
more general perturbative continuum transfer to nonlinear targets is not needed below and
is deferred to future work.

\subsection{A finite-dimensional nonabelian cellular shadow}

For a codimension-$k$ product blow-up the normal blow-down
$q_0:I\times\Delta^{k-1}\to C\Delta^{k-1}$ is a simple-homotopy equivalence relative to
the outer face.  Indeed the radial/barycentric homotopies can be chosen fixed
on that face, so after cellular approximation relative to the outer face they
give a homotopy equivalence of finite CW pairs.  The normal models are
contractible, hence their fundamental group is trivial and the relative
Whitehead torsion lies in $\operatorname{Wh}(1)=0$~\cite{CohenSimple}.
The product formula for torsion and the gluing theorem, applied relative to
the unchanged complement of the collar, show that taking the product with a
finite CW model of the centre and gluing to the identity preserve simple
homotopy.  We use this relative CW-pair statement in the following
finite-dimensional theorem.

\begin{theorem}[Nonabelian cellular resolution comparison]\label{thm:nonabelian-cellular}
Let two smooth monoidal resolutions of a compact generalized-corner manifold be connected
through an admitted common refinement by finite sequences of elementary radial blow-ups,
and choose blow-up-adapted finite ball-complex models (for instance, adapted
triangulations).  For a finite-dimensional unimodular
Lie algebra $\mathfrak g$, chosen representatives of the Cattaneo--Mnev--Reshetikhin
cellular BF actions are related by finite-dimensional
BV pushforwards and canonical transformations, after choosing the induction
data and determinant-line identifications of the cellular theory.  For either
chosen resolution
$X_{\cR_i}$, put
\[
 \F_{\mathrm{res}}(\cR_i)
 =H^\bullet(X_{\cR_i};\mathfrak g)[1]\oplus
 H_\bullet(X_{\cR_i};\mathfrak g^*)[-2].
\]
The simple-homotopy maps through the common refinement identify these
residual spaces, and under this identification the two partition functions
define the same residual BV cohomology class.
\end{theorem}

\begin{proof}
Each elementary radial blow-down is a simple-homotopy equivalence by the preceding
normal-model argument.  More explicitly, the relative homotopy equivalence
constructed in \cref{prop:radial-homotopy} is cellularly approximated relative
to the outer collar.  Its normal source and target are contractible and have
trivial fundamental group, so their Whitehead torsion lies in
\(\operatorname{Wh}(1)=0\); the product and gluing formulas used above preserve
this conclusion for the global CW pair.

For a chosen induction datum, \cite[Thm.~8.6]{CMRCellular} associates a
cellular action, well defined up to canonical BV transformation.
\cite[Lem.~8.19]{CMRCellular} identifies BV pushforward along an elementary
collapse with the action on the collapsed complex up to such a
transformation, and \cite[Prop.~8.20]{CMRCellular} identifies the pushforward
to residual fields as a simple-homotopy invariant modulo
\(\Delta_{\mathrm{res}}\)-exact half-densities.  Factor each simple equivalence
into elementary expansions and collapses and compose these finite-dimensional
pushforwards along the two arms of the common refinement.
The induced cohomology and homology isomorphisms identify
$\F_{\mathrm{res}}(\cR_1)$ with $\F_{\mathrm{res}}(\cR_2)$, which gives the
displayed residual-space and BV-cohomology statement, with determinant lines
transported by the chosen contraction data.  No identification with the cohomology
of the unresolved topological space $X$ is required.
\end{proof}

\subsection{Stellar subdivision and the transferred BF interaction}\label{sec:stellar-transfer}

The preceding comparison is at the level of partition functions and residual
BV cohomology.  We now ask the sharper local question: does resolving an
exceptional simplex change the transferred \emph{interaction} itself?  The
answer separates the one-dimensional case from the general one.

For a triangulation $K$ of a smooth manifold, let
\[
  \bigl(C^\bullet(K),d_K\bigr)
  \underset{R}{\overset{W}{\rightleftarrows}}
  \bigl(\Omega^\bullet(|K|),d\bigr),
  \qquad
  ds+sd=\id-WR,
\]
be the Whitney--integration--Dupont contraction~\cite{Dupont76}.  Tensoring with $\mathfrak g$ gives a contraction of the differential graded Lie algebra $\Omega^\bullet(|K|;\mathfrak g)$ onto cellular cochains.  Homotopy transfer produces brackets
\[
  \ell_1=d_K,
  \qquad
  \ell_2(a,b)=R[Wa,Wb],
\]
and, with standard Koszul signs,
\[
  \ell_3(a,b,c)
  =\sum_{\mathrm{cyc}}\!\epsilon(a,b,c)\,
    R\bigl[s[Wa,Wb],Wc\bigr].
\]
Higher brackets are sums over rooted trees whose internal edges are decorated by $s$.  These brackets generate the tree-level cellular BF action.

Let $K_\star$ be a stellar subdivision of one simplex of $K$.  The stellar welding contraction between $C^\bullet(K_\star)$ and $C^\bullet(K)$ will be denoted
\[
  C^\bullet(K)
  \underset{p_\star}{\overset{i_\star}{\rightleftarrows}}
  C^\bullet(K_\star),
  \qquad a_\star.
\]
Composing it with Dupont contraction on $K_\star$ gives contraction data from forms to coarse cochains with
\[
  W_{\mathrm{comp}}=W_\star i_\star,
  \qquad
  R_{\mathrm{comp}}=p_\star R_\star,
\]
\[
  s_{\mathrm{comp}}
  =s_\star+W_\star a_\star R_\star.
\]
Albert proves that the first two maps agree with the coarse Whitney and integration maps and that
\[
  s_{\mathrm{comp}}-s=d\Upsilon-\Upsilon d
\]
for a degree $-2$ operator $\Upsilon$; for subdivision of a one-simplex, the stronger equality $s_{\mathrm{comp}}=s$ holds \cite{AlbertStellar}.

\subsubsection*{Exact transfer for subdivision of an edge}

\begin{theorem}[No tree-level counterterm for an edge subdivision]\label{thm:edge-exact}
Let $K_\star$ be obtained from $K$ by stellar subdivision of an edge.  Transfer the dg Lie algebra $\Omega^\bullet(|K|;\mathfrak g)$ directly to $C^\bullet(K;\mathfrak g)$ using coarse Dupont data, or first to $C^\bullet(K_\star;\mathfrak g)$ and then through stellar welding to $C^\bullet(K;\mathfrak g)$.  The two transferred $L_\infty$ structures coincide term by term.  Consequently, the tree-level effective BF actions are identical under the canonical identification of coarse fields.
\end{theorem}

\begin{proof}
For an edge subdivision, Albert's one-dimensional compatibility theorem gives
\[
  W_{\mathrm{comp}}=W,
  \qquad
  R_{\mathrm{comp}}=R,
  \qquad
  s_{\mathrm{comp}}=s.
\]
Every homotopy-transfer bracket is a universal rooted-tree expression in $W$, $R$, the dg Lie bracket, and the propagator $s$.  Since all these ingredients coincide, every transferred bracket coincides.  Pairing the brackets with the dual $B$-field yields equality of the tree-level BV actions.
\end{proof}

This proves that the simplest exceptional interval does not generate a new classical local interaction.  At the finite-dimensional quantum level, the remaining measure-dependent discrepancy is precisely of the canonical/$\Delta$-exact type controlled by cellular BV pushforward.

\subsubsection*{Higher stellar subdivisions: equality is replaced by equivalence}

Set
\[
  \Xi=s_{\mathrm{comp}}-s=d\Upsilon-\Upsilon d.
\]
The binary bracket is unchanged because it depends only on $W$ and $R$.  The first possible defect occurs at arity three:
\[
  \ell_3^{\mathrm{comp}}(a,b,c)-\ell_3(a,b,c)
  =\sum_{\mathrm{cyc}}\!\epsilon(a,b,c)\,
   R\bigl[\Xi[Wa,Wb],Wc\bigr].
\]
Thus the correction is supported on trees containing at least one internal propagator.

\begin{theorem}[Canonical equivalence for a general stellar subdivision]\label{thm:stellar-equivalence}
Assume a complete descending filtration for which the homotopy-transfer series converge formally.  For an arbitrary stellar subdivision, the directly transferred and the fine-then-coarse transferred $L_\infty$ structures on $C^\bullet(K;\mathfrak g)$ are related by a filtered $L_\infty$ isomorphism whose linear part is the identity.  The corresponding tree-level BF actions are related by a classical canonical BV transformation.
\end{theorem}

\begin{proof}
The family
\[
  s_t=s+t(d\Upsilon-\Upsilon d),\qquad 0\leq t\leq1,
\]
still satisfies
\[
  ds_t+s_td=\id-WR.
\]
The three side conditions are not automatic along this family, but they can be
imposed without leaving it.  Put
\[
  \widehat s_t=(\id-WR)\,s_t\,(\id-WR),
  \qquad
  \check s_t=\widehat s_t\,d\,\widehat s_t .
\]
Since $RW=\id$, the operator $\id-WR$ is idempotent and commutes with $d$, so
$\widehat s_tW=0$, $R\widehat s_t=0$ and
$d\widehat s_t+\widehat s_td=(\id-WR)^3=\id-WR$.  Given these two annihilation
conditions one has $Rd\widehat s_t=d_KR\widehat s_t=0$, hence
$d\widehat s_td\widehat s_t=d\widehat s_t$ and therefore
$d\widehat s_t\widehat s_td=0$; consequently $\check s_t$ satisfies
$d\check s_t+\check s_td=\id-WR$ together with all three side conditions,
the last because
$\check s_t^2=\widehat s_t(d\widehat s_t\widehat s_td)\widehat s_t=0$.
Both replacements are polynomial in $s_t$, so $t\mapsto\check s_t$ is again
polynomial, and $\check s_0=s$, $\check s_1=s_{\mathrm{comp}}$ because $s$ and
$s_{\mathrm{comp}}$ already satisfy the side conditions.  Replacing $s_t$ by
$\check s_t$, we have a path of contraction data with fixed inclusion and
projection.  Homotopy transfer applied to a path of contractions gives a path of Maurer--Cartan elements in the convolution Lie algebra controlling $L_\infty$ structures \cite{LodayVallette}.  Differentiation with respect to $t$ is a gauge vector field; integration produces a filtered $L_\infty$ isomorphism with linear part $\id$.  Under the cotangent BF pairing, an $L_\infty$ gauge equivalence is represented by a classical canonical transformation of the BV action.  Formal completeness ensures that the rooted-tree series and the integrated gauge transformation are defined.
\end{proof}

\begin{corollary}[Interpretation of exceptional interactions]\label{cor:exceptional-interactions}
A higher-dimensional stellar subdivision may change individual higher brackets and hence may produce local interaction terms involving exceptional cellular modes.  These terms do not define new tree-level physics: after the exceptional modes are transferred out, they differ from the coarse interaction by an $L_\infty$ gauge equivalence.  Literal equality occurs for edge subdivision; canonical equivalence is the correct general statement.
\end{corollary}

\subsubsection*{What the stellar calculation establishes}

\Cref{thm:nonabelian-cellular,thm:edge-exact,thm:stellar-equivalence} answer the first refinement question at three levels:
\begin{enumerate}[label=(\roman*),leftmargin=2.2em]
\item the full finite-dimensional canonical cellular partition function is independent of resolution in residual BV cohomology;
\item for an exceptional interval, the classical transferred interaction is exactly unchanged;
\item for a higher exceptional simplex, changed higher operations are gauge-equivalent, with the first defect explicitly controlled by $\Xi=d\Upsilon-\Upsilon d$.
\end{enumerate}
What remains open is the promotion of this result to the complete continuum, facewise AKSZ--BV--BFV system.  The propagator required for a product codimension-two resolution chart, together with its compatibility with every local face restriction, is \cref{thm:codim2-continuum-propagator} below.  Global compatibility under changes of monoidal chart and composition of refinements remains open.

\subsection{The codimension-two continuum collar}

Let $\widetilde C_\Gamma=\Gamma\times I_r\times I_\theta$ be the resolved product collar,
with exceptional face $E=\{r=0\}$ and adjacent faces $F_0=\{\theta=0\}$,
$F_1=\{\theta=1\}$.  Let $j_E:E\hookrightarrow\widetilde C_\Gamma$ and let
$\pi_E$ be the radial retraction.  Put
\[
 P_E:=\pi_E^*j_E^*.
\]
Choose the normalized radial integration homotopy $H_r$ so that
\[
 dH_r+H_rd=\id-P_E,\qquad
 H_r^2=H_rP_E=P_EH_r=0,\qquad j_E^*H_r=0.
\]
Let
$(W_E,R_E,\kappa_E)$ be the Whitney--integration--Dupont contraction on the angular
interval, tensored with the identity on $\Omega^\bullet(\Gamma)$, and set
\[
 \mathcal I=\pi_E^*W_E,\qquad
 \mathcal P=R_Ej_E^*,\qquad
 \mathcal K=H_r+\pi_E^*\kappa_Ej_E^*.
\]

\begin{theorem}[Codimension-two collar propagator]\label{thm:codim2-continuum-propagator}
These maps form a special deformation retract
\[
 \Omega^\bullet(\widetilde C_\Gamma)
 \quad\rightleftarrows\quad
 \Omega^\bullet(\Gamma)\otimes C^\bullet(I),
\]
with
\[
 d\mathcal K+\mathcal Kd=\id-\mathcal I\mathcal P,
 \qquad \mathcal P\mathcal I=\id,
 \qquad \mathcal K^2=\mathcal K\mathcal I=\mathcal P\mathcal K=0.
\]
Moreover $j_E^*\mathcal K=\kappa_Ej_E^*$, while on $F_0,F_1$ it is the corresponding
radial homotopy; it vanishes at the two endpoint corners.
\end{theorem}

\begin{proof}
The tensor trick gives the type-correct identity
\[
d\mathcal K+\mathcal Kd
=(\id-P_E)+\pi_E^*(\id-W_ER_E)j_E^*
=\id-\pi_E^*W_ER_Ej_E^*
=\id-\mathcal I\mathcal P.
\]
The side conditions follow from the displayed normalization of $H_r$ and
those of $\kappa_E$.  Face compatibility uses
$j_E^*H_r=0$ and the endpoint identities
$\operatorname{ev}_0\kappa=\operatorname{ev}_1\kappa=0$.
\end{proof}

\begin{theorem}[Abelian exceptional interval]\label{thm:abelian-continuum-blowup}
For abelian BF theory on $\widetilde C_\Gamma$, choose complementary
relative/absolute $A/B$ boundary conditions (or an algebraic dual model) for
which the $B$-retract is adjoint to the $A$-retract above.  Then the preceding
retract reduces the continuum complex to
\[
 \Omega^\bullet(\Gamma;V)\otimes C^\bullet(I)
 \oplus
 \Omega^\bullet(\Gamma;V^*)\otimes C_\bullet(I).
\]
The diagonal endpoint mode is the coarse corner field and the anti-diagonal/edge pair is
contractible.  Classical BV elimination therefore gives exactly the coarse abelian BF corner
action, with no field-dependent counterterm.  Under \cref{ass:gaussian-determinant} (or a
finite-dimensional cutoff), the remaining quantum factor is the field-independent torsion of
the reduced interval complex and may be absorbed by the elementary-collapse normalization.
\end{theorem}

\begin{proof}
Apply \cref{thm:codim2-continuum-propagator} to the $A$-complex and its dual to the
$B$-complex, then use the diagonal interval contraction of
\cref{prop:diagonal-interval}.  The classical statement is elimination of a contractible
cotangent pair; the determinant statement is exactly the scope of
\cref{ass:gaussian-determinant}.
\end{proof}

\subsection{The nonabelian exceptional interval implements gauge transport}

Write $\Gamma_0=E\cap F_0$ and $\Gamma_1=E\cap F_1$ for the two codimension-three
faces of the resolved collar $\widetilde C_\Gamma$ of the previous subsection.

For nonabelian BF theory the interval sector is not eliminated by the linear contraction without producing nonlinear structure.  The relevant one-dimensional pushforward is nevertheless explicit.  Let $\mathfrak l$ be a complete filtered dg Lie algebra; in the application one takes formally, or after a finite-mode cutoff,
\[
  \mathfrak l=\Omega^\bullet(\Gamma;\mathfrak g)
\]
with differential $d_\Gamma$ and the wedge--Lie bracket.  The Whitney--Dupont pushforward of BF theory on $I$ retains endpoint ghosts $c_0,c_1$ and an edge field $a$.  Its classical cellular action has the universal interval term
\begin{equation}\label{eq:nonabelian-interval-term}
  \left\langle a^+,
    F_+(\operatorname{ad}_a)c_1
    -F_-(\operatorname{ad}_a)c_0
  \right\rangle,
\end{equation}
where
\begin{equation}\label{eq:Fpm}
  F_+(z)=\frac{z}{1-e^{-z}},
  \qquad
  F_-(z)=\frac{z}{e^z-1}=e^{-z}F_+(z).
\end{equation}
The remaining terms are the endpoint Chevalley--Eilenberg terms, the $d_\Gamma$ contribution, and their cotangent lift.  In finite dimension the one-loop correction is $-i\hbar\,\trg\log G(\operatorname{ad}_a)$ with $G$ as in \eqref{eq:interval-action}.  These formulas are the interval effective action of canonical nonabelian BF theory \cite{CMRCellular,CMRPushforward}.  In the superfield notation of \cref{prop:canonical-interval-block} one has $c_0=A_0$, $c_1=A_1$, $a=A_{01}$, $a^+=B_{01}$, and \eqref{eq:nonabelian-interval-term} is the $F_\pm$ form of the interval term in \eqref{eq:interval-action}: since
\[
  F_\pm(z)=F(z)\pm\frac z2,
\]
the identity $F_+(\operatorname{ad}_a)c_1-F_-(\operatorname{ad}_a)c_0=F(\operatorname{ad}_a)(c_1-c_0)+\frac12[a,c_0+c_1]$ recovers it exactly.

\begin{proposition}[Exceptional holonomy relation]\label{prop:exceptional-holonomy}
In a formal neighborhood of $a=0$, the equation obtained by varying $a^+$ in \eqref{eq:nonabelian-interval-term} is equivalent to
\begin{equation}\label{eq:ghost-transport}
  c_1=e^{-\operatorname{ad}_a}c_0.
\end{equation}
The nonlinear infrared projection sends a continuum connection in the exceptional direction to
\begin{equation}\label{eq:log-holonomy}
  a=-\log\operatorname{Pexp}
  \left(-\int_0^1 A_\theta\,\dd\theta\right).
\end{equation}
Hence the exceptional interval identifies the two endpoint corner complexes by gauge transport.  Its linearization at the trivial holonomy is the abelian relation $c_1=c_0$.
\end{proposition}

\begin{proof}
Since $F_+(0)=1$, the formal power series $F_+(\operatorname{ad}_a)$ is invertible.  Using the second identity in \eqref{eq:Fpm}, the equation
\[
  F_+(\operatorname{ad}_a)c_1
  =F_-(\operatorname{ad}_a)c_0
\]
is equivalent to \eqref{eq:ghost-transport}.  Formula \eqref{eq:log-holonomy} is the classical nonlinear projection associated with the Whitney--Dupont BV pushforward on the interval; it is obtained by straightening the connection by a gauge transformation relative to the endpoints \cite{CMRPushforward}.  Setting $a=0$ gives the final statement.
\end{proof}

\begin{theorem}[Formal local normal form of a nonabelian codimension-two resolution move]\label{thm:nonabelian-codim2-normal-form}
Assume product geometry near the exceptional hypersurface and work
perturbatively around trivial exceptional holonomy.  Under the standing formal
mapping-space hypothesis of the introduction, or after a fixed finite-mode
cutoff, classical tree-level homological transfer with the propagator of
\cref{thm:codim2-continuum-propagator} represents the local resolved theory by:
\begin{enumerate}[label=(\roman*),leftmargin=2.2em]
\item continuum BF data along $\Gamma$;
\item two endpoint corner fields associated with $\Gamma_0$ and $\Gamma_1$;
\item one exceptional edge field $a$ whose geometric meaning is the logarithm of angular holonomy;
\item the universal interval interaction \eqref{eq:nonabelian-interval-term}.
\end{enumerate}
The exceptional interval therefore defines a gauge-theoretic identity correspondence between the two endpoint corner theories.  In the abelian limit this correspondence reduces by a strict BV contraction to the diagonal corner theory of \cref{thm:abelian-continuum-blowup}.  In the nonabelian theory, eliminating $a$ globally requires a choice of holonomy sector and is not part of a canonical strict facewise pushforward.
No analytic continuum BV pushforward, loop expansion, or determinant identity
is asserted by this theorem.
\end{theorem}

\begin{proof}
At a finite-mode cutoff the statement follows from
\cref{thm:codim2-continuum-propagator} and ordinary homological perturbation
applied to the BF interaction; formally the same calculation is performed in
the complete filtration.  Since the contraction is the Whitney--Dupont
contraction in the angular direction, its rooted-tree expansion is exactly the
classical interval cellular BF action.  \Cref{prop:exceptional-holonomy}
identifies its residual edge variable and its endpoint relation.  The abelian
reduction follows from \cref{prop:diagonal-interval}.  In the nonabelian case
the holonomy is a genuine residual modulus: different values need not be
related by gauge transformations fixing both endpoints.  Therefore a further
pushforward to a single diagonal corner field is canonical only after
restricting to, or choosing a gauge slice in, the trivial-holonomy sector.
\end{proof}

\begin{remark}[A correction to the naive acyclicity principle]
The linear exceptional complex is acyclic after the diagonal corner mode is separated.  Nonlinearly, however, the exceptional interval carries holonomy.  Thus \enquote{exceptional acyclicity} must be interpreted perturbatively around a chosen residual holonomy sector.  The cellular resolution theorem remains valid globally in residual BV cohomology, but a continuum facewise comparison must either retain the holonomy field or specify how it is integrated.
\end{remark}

\subsection{Consequences for the refinement program}

The codimension-two local problem can now be separated into three controlled layers:
\[
\begin{array}{c}
\text{continuum resolved collar}
\\[1mm]
\downarrow\ \text{strict face-compatible propagator}
\\[1mm]
\text{continuum on }\Gamma\ \widehat\otimes\ \text{cellular BF on }I
\\[1mm]
\downarrow\ \text{exceptional interval reduction}
\\[1mm]
\text{coarse corner theory plus residual holonomy}.
\end{array}
\]
The first arrow is a genuine theorem for differential forms and is compatible with every restriction map.  The second arrow is exact in abelian BF theory.  In nonabelian BF theory it is exact only as a correspondence retaining the holonomy variable; reduction to a single coarse corner requires an additional choice.

This is the precise local form of the admissibility question recorded in \cref{sec:open-problems}.  For a codimension-two elementary resolution move, the required propagator exists in product geometry.  The remaining nonlinear datum is not an unknown ultraviolet counterterm but a finite-dimensional exceptional holonomy sector.  Consequently, a global refinement-invariance theorem should be formulated over the derived moduli of exceptional flat connections, rather than over a point, unless the relevant holonomy is constrained to be trivial.

\subsection{Continuum--cellular abelian comparison}

\begin{definition}[Regularized cellularization datum]\label{def:regularized-cellularization}
On a resolved ordered-face diagram, a regularized cellularization datum is a face-compatible
collection of special deformation retracts from logarithmic forms to a face-adapted finite
cellulation, compatible with regularized restriction and adjoint for the $A/B$ pairing.
Existence of such global data is an explicit hypothesis; the radial factor on product
logarithmic collars is supplied by \cref{prop:regularized-interval-contraction}.
\end{definition}

\begin{corollary}[Local product logarithmic collar]\label{cor:product-log-contraction}
Given a face-compatible tangential Whitney--Dupont retract on $\Gamma$, tensoring it with
\cref{prop:regularized-interval-contraction} in the logarithmic radial direction and with the
ordinary interval retract in the angular direction gives a face-compatible retract on
$\Gamma\times I_b\times I_\theta$.  This supplies only the local collar factor; no global
gluing or existence statement for \cref{def:regularized-cellularization} is implied.
\end{corollary}

\begin{theorem}[Classical continuum--cellular comparison for abelian BF]\label{thm:abelian-continuum-cellular}
Let a compact smooth monoidal resolution carry a resolved regularization, a face-adapted
finite cellulation and a regularized cellularization datum.  The cyclic special deformation
retract splits the quadratic continuum action as
\[
 S_N^{\mathrm{cont}}=S_K^{\mathrm{cell}}+S_{\mathrm{UV}}
\]
with no mixed term.  Classical elimination of the acyclic ultraviolet pair gives exactly the
cellular abelian BF action and its BFV boundary operator.  If
\cref{ass:gaussian-determinant} holds, the Gaussian pushforward differs only by the chosen
field-independent determinant-line torsion and is compatible with composition of retracts.
\end{theorem}

\begin{proof}
Adjointness and the special side conditions make the cellular and ultraviolet sectors
orthogonal for the BF pairing.  The classical action therefore splits.  The quantum clause
is precisely the determinant regularity and Fubini condition of
\cref{ass:gaussian-determinant}.
\end{proof}

\begin{remark}
For nonabelian BF, the exceptional interval can retain angular holonomy and nonlinear
transfer creates higher operations.  The finite-dimensional cellular theorem above remains
valid, but a general continuum nonlinear pushforward requires additional configuration-space
and renormalization input and is not a theorem of the present paper.
\end{remark}

\part*{V.\ Resolved--intrinsic \texorpdfstring{$b$}{b}/logarithmic comparison}\addcontentsline{toc}{part}{V.\ Resolved--intrinsic b/logarithmic comparison}

\section{Comparison with monoidal resolutions}

\subsection{\texorpdfstring{$b$}{b}-étaleness of a refinement}

Let
\[
 \beta:X_\cR\longrightarrow X
\]
be the blow-up associated with a smooth refinement of the monoidal complex of
a $g$-corner with embedded faces \cite{KottkeGCorner}.  In an affine chart
write $P=\sigma^\vee\cap\Lambda^*$.  A unimodular cone
$\tau\subseteq\sigma$ of the refinement (taken full-dimensional in the
chart) has free chart monoid
$Q_\tau=\tau^\vee\cap\Lambda^*$, and the inclusion
$P\subseteq Q_\tau$ induces the identity on the common groupification
$\Lambda^*$.

\begin{proposition}\label{prop:betale}
For a monoidal refinement, the $b$-derivative is a vector-bundle isomorphism
\[
 \bT\beta:\bT X_\cR\xrightarrow{\cong}\beta^*\bT X,
\]
and dually
\[
 \beta^*\bT^*X\xrightarrow{\cong}\bT^*X_\cR.
\]
Hence
\[
 \beta^*:\bOmega^\bullet(X)\longrightarrow\bOmega^\bullet(X_\cR)
\]
is a morphism of differential graded algebras.
\end{proposition}

\begin{proof}
Locally, a refinement subdivides the cone without changing its ambient
lattice.  Thus $Q_\tau^{\gp}=P^{\gp}=\Lambda^*$.  Under the canonical local
trivializations, both $b$-tangent bundles have fibre
$\Hom(\Lambda^*,\R)$, and the $b$-derivative is the identity on this fibre.
The statement glues over refined charts.
\end{proof}

\begin{remark}[Which theory on the resolution is being compared]\label{rem:two-sources-on-resolution}
A manifold with ordinary corners carries two AKSZ sources: the ordinary
$T[1]N$ used in \cref{sec:resolution-aksz-data}, and the $b$-tangent algebroid
$\bT[1]N$, which for ordinary corners is Melrose's.  The anchor
$\bT N\to TN$ is an isomorphism over the interior and drops rank along every
boundary hypersurface, so the induced map
$\Omega^\bullet(N)\to\bOmega^\bullet(N)$ is injective but not surjective, and
the two mapping spaces are different.  In this section $X_\cR$ carries its
$b$-source: by \cref{thm:dpp} the Dupont--Panzer--Pym regularization of
\textup{(G2)} is a regularized Stokes trace system on the ordered face diagram
of $X_\cR$, and we write
\[
 \F_{X_\cR}^b=\Map(\bT[1]X_\cR,\Y),
 \qquad
 \bigl(\omega_{X_\cR}^b,\alpha_{X_\cR}^b,S_{X_\cR}^b,Q_{X_\cR}\bigr),
\]
and likewise on every ordered face of $X_\cR$, for the data that
\cref{def:transgression} and \cref{thm:bvbfv} produce from it.  The
ordinary-corner system $\mathfrak{AKSZ}_{\mathrm{res}}(X,\Y;\cR)$ of
\cref{sec:resolution-aksz-data} maps to this one by pullback along the anchor;
the two agree on integrands with no logarithmic normal component, and the
comparison below is between the $b$-data on $X_\cR$ and the intrinsic data on
$X$.  Schematically, the three objects and the two distinct comparisons are
\[
 \operatorname{AKSZ}\bigl(T[1]X_\cR\bigr)
 \longrightarrow
 \operatorname{AKSZ}_{b,\log}\bigl(\bT[1]X_\cR\bigr)
 \xleftarrow[\text{derived/linear strict under (G4)}]{\ \beta^*\ \text{on the basic bulk sector}\ }
 \operatorname{AKSZ}_{b,\log}\bigl(\bT[1]X\bigr).
\]
The left arrow is induced by the anchor and is not asserted to be an
equivalence; the right comparison is the content of this part of the paper.
\end{remark}

\subsection{The basic bulk sector and the facewise comparison}

A field on \(X\) pulls back to a field on \(X_\cR\).  We call
\[
 \F_{X_\cR}^{\mathrm{bas,bulk}}
 =\operatorname{im}\bigl(
   \beta^*:\F_X^b\to\F_{X_\cR}^b
  \bigr)
\]
the \emph{basic bulk sector}.  This terminology is intentionally restricted
to bulk fields: there is generally no single resolved face corresponding to a
given intrinsic face, so a facewise ``basic sector'' must be defined by the
subdivision/aggregation maps of the total face complex rather than by
termwise pullback.

Pullback is injective: \(\beta\) is the identity on the common dense interior,
and a smooth \(b\)-form, hence every component of a \(b\)-field in the chosen
formal class, is determined by its restriction to that interior.  Thus
\(\F_{X_\cR}^{\mathrm{bas,bulk}}\) is genuinely identified with
\(\F_X^b\), not merely named as an image.

\begin{proposition}[Derived affine trace compatibility]\label{prop:tracecompat}
Let $(\cR_1,s_1)$ and $(\cR_2,s_2)$ be two regularized smooth resolutions of
the affine model $X_P$.  If compactly supported total cocycles
$a_i\in\mathbb A_{\cR_i,s_i}^\bullet(P)$ represent corresponding classes
under the common-interior roof, then
\[
 \mathcal I_{\cR_1,s_1}(a_1)=\mathcal I_{\cR_2,s_2}(a_2).
\]
In particular, for $\varphi\in\Omega_c^n(X_P^\circ)$, the representatives
\[
 a_i=\epsilon_{\cR_i,s_i}\bigl((\beta_i^\circ)^*\varphi\bigr)
\]
have the common value $\int_{X_P^\circ}\varphi$.  If linear strictification
data are fixed on one representative $\mathfrak s=(\cR,s_{\cR})$, the strict
intrinsic trace is related to its resolved trace by
\[
 \mathcal I_{P,\mathfrak s}^{\mathrm{str}}
 =\mathcal I_{\cR,s_{\cR}}\circ i_\infty.
\]
No resolution-independent value is asserted here for an arbitrary
face-degree-zero logarithmic top form: such a form need not be a cocycle in
the relative totalization, and its absolute regularized integral can depend
on the chain-level regularization.
\end{proposition}

\begin{proof}
The first statement is precisely trace compatibility in
\cref{thm:refinement-invariance}, applied to the two arms of a common-interior
roof.  The second follows from \eqref{eq:interior-trace-bridge} and ordinary
change of variables for the orientation-preserving diffeomorphisms of
interiors.  The last identity is the transferred-trace definition in
\cref{prop:strictification}; it does not require choosing one resolved face
above each intrinsic face.
\end{proof}

\begin{theorem}[Basic-sector comparison]\label{thm:basic-comparison}
Let \(\beta:X_\cR\to X\) be a smooth monoidal resolution.
\begin{enumerate}[label=\textup{(\roman*)},leftmargin=2.5em]
\item On bulk fields, \(\beta^*\) identifies the intrinsic $b$-source with
\(\F_{X_\cR}^{\mathrm{bas,bulk}}\) and intertwines \(\db\), evaluation of
target tensors, and all universal AKSZ integrands.
\item After universal AKSZ integrands have been represented by compactly
supported cocycles of the relative total source complexes, their traces on two
resolutions agree provided those cocycles represent corresponding classes
under the common-interior roof.  In particular this applies
to compactly supported smooth top forms on the common interior.  No equality
of arbitrary absolute regularized bulk integrals is claimed.
\item Under \textup{(G1)}--\textup{(G4)} and for basic linear coefficients in
the specified transfer class, the comparison of the full linear total face
complexes is the strong deformation retract
\[
 (\mathscr A_L^\bullet,D)
 \underset{i_\infty}{\overset{p_\infty}{\rightleftarrows}}
 (\mathscr B_{X,\cV}^\bullet,D_{\mathrm{str}}).
\]
\pagebreak[3]
When the coefficients carry a compatible scalar trace, it intertwines the resolved
total-face trace identity with the strict intrinsic linear Stokes identity.  If the
BV-cyclic field datum \textup{(G5)} is also
chosen, the corresponding classical abelian BF BV--BFV field systems are cyclically
chain-equivalent through the separate field-level retracts of \textup{(G5)};
this assertion is not obtained by forgetting structure on the displayed
logarithmic strong deformation retract.
\end{enumerate}
There is in general no assertion that the comparison in \textup{(iii)} is
termwise pullback along a bijection of faces.
\end{theorem}

\begin{proof}
Part~\textup{(i)} is \cref{prop:betale}.  Part~\textup{(ii)} follows from
\cref{prop:tracecompat}.  Part~\textup{(iii)} is
\cref{prop:filtered-exceptional-contraction,prop:strict-stokes-global}; the classical abelian BF statement additionally uses the cyclic perturbation
result \cref{thm:cyclic-hpl-bf,cor:global-strict-bf}.
\end{proof}

\begin{remark}
For nonlinear targets, the image of \(i_\infty\) is not automatically closed
under the transferred product or bracket.  Homotopy transfer may generate
higher operations such as \(\ell_3\) in \cref{eq:transferred-l3}.  Therefore
the linear equivalence of total face complexes above must not be promoted to a nonlinear
strict equivalence without additional hypotheses.
\end{remark}

\subsection{The comparison picture}

There are two different comparison statements, and it is useful not to mix
their types.  At the level of bulk $b$-fields the blow-down gives the honest
pullback
\begin{equation}\label{eq:bulk-comparison-picture}
 \beta^*:\F_X^b\longrightarrow\F_{X_\cR}^{\mathrm{bas,bulk}}.
\end{equation}
At the level of the full linear trace/descent complex, under \textup{(G4)} and for basic
linear coefficients in its specified transfer class, the comparison is instead the strong deformation
retract
\begin{equation}\label{eq:face-comparison-picture}
 (\mathscr A_L^\bullet,D)
 \underset{i_\infty}{\overset{p_\infty}{\rightleftarrows}}
 (\mathscr B_{X,\cV}^\bullet,D_{\mathrm{str}}).
\end{equation}
Thus there is no type-correct square in which a mapping space, a trace, and a
full AKSZ--BV--BFV system are simply four vertices of one diagram.  For a
general nonlinear target, integrating out the reduced refinement modes and
identifying the result with a strict intrinsic nonlinear theory is the
additional BV-pushforward statement isolated in \cref{conj:comparison}; it is
not part of the proved linear comparison.

\part*{VI.\ Worked examples}\addcontentsline{toc}{part}{VI.\ Worked examples}

\section{The positive real conifold}\label{sec:conifold}

This section tests the framework on the smallest non-ordinary example.  The
common-interior derived comparison is valid for the continuum relative
complexes.  On the explicitly specified product--Whitney coefficient class we
also construct the full global linear contraction after normal-face
totalization, and on its finite dual
sub-class we construct a cyclic logarithmic--cellular abelian BF shadow.
The latter is not a construction of the full \(b\)-de Rham datum
\textup{(G5)}.  These finite-model statements do not extend to arbitrary
continuum coefficients.

\subsection{Two squares}\label{subsec:conifold-two-squares}

Recall the model \eqref{eq:conifold-model} and the ray realization
\eqref{eq:square-rays} of its cone $\sigma$.  Two different squares occur in
the discussion of this example and it is worth separating them once; both are
drawn in \cref{fig:conifold-two-squares}.

The \emph{ray square} is the transverse slice $\sigma\cap\{z=1\}$, whose
vertices are the four rays $v_1,v_2,v_4,v_3$ in cyclic order.  This is the
picture in which refinements are visible: a diagonal of the ray square is a
new two-dimensional cone, and the star refinement adds the centre.

The \emph{dual-generator square} is a transverse polygon for
$\sigma^\vee$.  After a unimodular change of lattice coordinates relative to
the presentation in \cref{eq:conifold-monoid-presentation}, write
$P=\sigma^\vee\cap\Lambda^*$ as in \cref{subsec:weakly-toric}.  Its four
primitive ray generators $q_i$ are the facet normals of $\sigma$, one for
each pair of adjacent primal rays:
\begin{equation}\label{eq:conifold-facets}
\begin{aligned}
 q_1&=(0,1,0) &&\text{for the facet }\Cone(v_1,v_2),\\
 q_2&=(0,-1,1) &&\text{for the facet }\Cone(v_3,v_4),\\
 q_3&=(1,0,0) &&\text{for the facet }\Cone(v_1,v_3),\\
 q_4&=(-1,0,1) &&\text{for the facet }\Cone(v_2,v_4).
\end{aligned}
\end{equation}
Both $q_1+q_2$ and $q_3+q_4$ equal $(0,0,1)$, so $P$ has the single relation
\begin{equation}\label{eq:conifold-monoid-relation}
 q_1+q_2=q_3+q_4,
\end{equation}
and $X_P\cong X_\square$ under $x_i\leftrightarrow q_i$.  The relation pairs
$q_1$ with $q_2$ and $q_3$ with $q_4$, that is, it identifies the two sums of
opposite rays of $\sigma^\vee$.  The groupification $P^{\gp}$ has rank three.
The boundary hypersurfaces of $X_P$ are instead indexed by the four facets of
$\sigma^\vee$, equivalently by the primal rays $v_i$.  Thus a generator
$q_i$ is not the label of a single boundary hypersurface; its monomial may
vanish on a union of them.

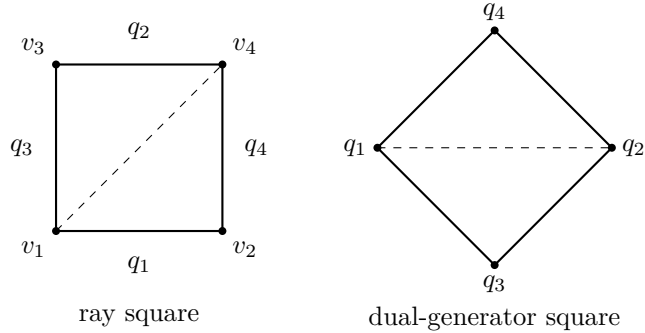
\begin{figure}[ht]
\centering
\begin{tikzpicture}[scale=1.0,every node/.style={font=\small}]
  \begin{scope}[xshift=-4.6cm]
    \coordinate (a) at (0,0); \coordinate (b) at (2.2,0);
    \coordinate (c) at (0,2.2); \coordinate (d) at (2.2,2.2);
    \draw[thick] (a)--(b)--(d)--(c)--cycle;
    \draw[dashed] (a)--(d);
    \fill (a) circle (1.5pt) node[below left] {$v_1$};
    \fill (b) circle (1.5pt) node[below right] {$v_2$};
    \fill (c) circle (1.5pt) node[above left] {$v_3$};
    \fill (d) circle (1.5pt) node[above right] {$v_4$};
    \node at (1.1,-0.45) {$q_1$};
    \node at (1.1,2.65) {$q_2$};
    \node at (-0.45,1.1) {$q_3$};
    \node at (2.65,1.1) {$q_4$};
    \node at (1.1,-1.15) {ray square};
  \end{scope}
  \begin{scope}[xshift=1.2cm,yshift=1.1cm]
    \coordinate (p1) at (-1.55,0); \coordinate (p3) at (0,-1.55);
    \coordinate (p2) at (1.55,0);  \coordinate (p4) at (0,1.55);
    \draw[thick] (p1)--(p3)--(p2)--(p4)--cycle;
    \draw[dashed] (p1)--(p2);
    \fill (p1) circle (1.5pt) node[left] {$q_1$};
    \fill (p3) circle (1.5pt) node[below] {$q_3$};
    \fill (p2) circle (1.5pt) node[right] {$q_2$};
    \fill (p4) circle (1.5pt) node[above] {$q_4$};
    \node at (0,-2.25) {dual-generator square};
  \end{scope}
\end{tikzpicture}
\caption{The two squares attached to $X_\square$.  On the left, the
transverse slice of the cone: its vertices are the rays $v_j$ and its edges
are the facets whose primitive normals are $q_i$, so a dashed diagonal is a
candidate refinement.  On the right, the polygon of dual ray generators; its
edges are dual to the $v_i$ and hence index the boundary hypersurfaces.  Its
dashed diagonal is the pair
$\{q_1,q_2\}$ appearing in the relation \eqref{eq:conifold-monoid-relation}.
Four boundary hypersurfaces meet at the vertex, but there are only three
logarithmic normal directions.}
\label{fig:conifold-two-squares}
\end{figure}

The intrinsic route uses the face lattice of $\sigma^\vee$: its boundary
hypersurfaces are the facets indexed by $v_i$, while its ray generators
$q_i$ obey the nonfree relation above.  The resolution route uses the primal
ray square, since that is where refinements are drawn.  A choice of diagonal
in the ray square corresponds to a simplicial resolution, whereas the
intrinsic $b$-source retains the unsplit monoid relation.

\subsection{Joyce corners versus iterated boundaries}\label{subsec:conifold-joyce-corners}
The same example is also the minimal test for the higher-face combinatorics.
In Joyce's notation, the corner spaces and iterated boundaries of
$X_\square$ are not related as they would be for an ordinary corner.  His
explicit calculation gives
\begin{center}
\begin{tabular}{ccl}
\toprule
$k$ & $C_k(X_\square)$ & $\partial^kX_\square$\\
\midrule
0 & $X_\square$ & $X_\square$\\
1 & $4$ copies of $[0,\infty)^2$ & $4$ copies of $[0,\infty)^2$\\
2 & $4$ copies of $[0,\infty)$ & $8$ copies of $[0,\infty)$\\
3 & one point & $8$ points\\
\bottomrule
\end{tabular}
\end{center}
See \cite[Ex.~3.31 and Prop.~3.32]{JoyceGCorner}.  In particular, over the
vertex there is one intrinsic codimension-three corner component but eight
points in the third iterated boundary.  Thus no free $S_3$-quotient of
$\partial^3X_\square$ can recover $C_3(X_\square)$; Joyce proves more generally
that the natural adjacent transpositions on $\partial^kX$ need not satisfy the
relations of $S_k$ for a $g$-corner.  This is exactly why the intrinsic descent
in this paper is indexed by $\mathsf{Face}(X)$ and its flags, while the
$S_k$-ordered-boundary model is used only after passing to an ordinary-corner
resolution.

The table also prevents a second possible misreading.  Four boundary
hypersurfaces meet at the conifold vertex, but they do not form the Boolean
face lattice of an ordinary four-hypersurface corner.  The monoidal relation
\eqref{eq:conifold-monoid-relation} compresses the logarithmic normal lattice
to rank three, while Joyce's corner spaces encode the actual intrinsic
incidence strata.

\subsection{The intrinsic source}

\begin{proposition}[Intrinsic compression]\label{prop:conifold-forms}
Let
\[
 \vartheta_i=\db\log x_i.
\]
Then
\begin{equation}\label{eq:conifold-relation}
 \vartheta_1+\vartheta_2=\vartheta_3+\vartheta_4.
\end{equation}
The $b$-cotangent bundle of $X_\square$ is a trivial rank-three bundle
generated by any three independent classes among the $\vartheta_i$.
\end{proposition}

\begin{proof}
The relation \eqref{eq:conifold-monoid-relation} in $P^{\gp}$ gives
\[
 \ell_{q_1}+\ell_{q_2}=\ell_{q_3}+\ell_{q_4}.
\]
Apply $\db$.  Since $\operatorname{rank}P^{\gp}=3$, there is exactly one
independent linear relation.
\end{proof}

\subsection{Scales at the vertex}

\begin{proposition}[Vertex-scale relation]\label{prop:scale-relation}
A monoidal scale at the vertex of $X_\square$ is equivalent to a quadruple
\[
 (c_1,c_2,c_3,c_4)\in\R^4
\]
satisfying
\begin{equation}\label{eq:scale-relation}
 c_1+c_2=c_3+c_4.
\end{equation}
The intrinsic vertex values are $\Reg_v(\log x_i)=c_i$.
\end{proposition}

\begin{proof}
A scale is a homomorphism $\scale:P^{\gp}\to\R$.  Set $c_i=\scale(q_i)$.
Applying $\scale$ to \eqref{eq:conifold-monoid-relation} gives
\eqref{eq:scale-relation}.  Conversely, any quadruple satisfying this
relation defines a unique homomorphism from the presented abelian group
$P^{\gp}$.
\end{proof}

Thus the four logarithmic generators do not carry four independent vertex
values.  This relation controls the intrinsic differential source and its
regularized evaluation at the vertex, but it does not parametrize a complete
DPP regularization on any smooth resolution; the compatible positive normal
sections remain additional data by
\cref{prop:vertex-character-insufficient}.

For compactly supported fields on a truncated neighbourhood of the vertex,
the intrinsic BF action is
\[
 S_\square^{\mathrm{BF}}
 =\Tr_{X_\square}^{\cT}\langle B,\db A\rangle.
\]
Its four codimension-one face terms are indexed by the boundary
hypersurfaces $H_{v_i}$, while all logarithmic normal components lie in the rank-three space
$P^{\gp}\otimes\R$.  At the vertex,
regularized evaluation obeys \eqref{eq:scale-relation}.  The
codimension-two incidence cancellations are consequently relations in the
face poset of $P$, rather than the Boolean cancellations of an orthant.

\subsection{The two diagonal resolutions}

The two diagonals of the ray square give the smooth refinements $\cR_+$
and $\cR_-$ of \eqref{eq:triang-plus}--\eqref{eq:triang-minus}.  Let
\[
  v_0=v_1+v_4=v_2+v_3=(1,1,2).
\]
The ray through $v_0$ gives a common star refinement $\cR_\star$ with four
maximal cones.  Here $v_0$ is already primitive, and the determinants of the
four ordered ray triples
\[
 (v_1,v_2,v_0),\quad (v_2,v_4,v_0),\quad
 (v_4,v_3,v_0),\quad (v_3,v_1,v_0)
\]
have absolute value one.  Hence $\cR_\star$ is a smooth common refinement,
obtained from either diagonal refinement by a single stellar subdivision
\textup{(}\cref{fig:square-refinements}\textup{)}.  Writing $\beta_\star=\beta_+\gamma_+=\beta_-\gamma_-$ for the composite
blow-down of the star refinement, the blow-ups fit into
\[
\begin{tikzcd}[column sep=large]
 & X_{\cR_\star}
   \arrow[dl,"\gamma_+"']
   \arrow[dr,"\gamma_-"] & \\
 X_{\cR_+}
   \arrow[dr,"\beta_+"'] &&
 X_{\cR_-}
   \arrow[dl,"\beta_-"]\\
 & X_\square.&
\end{tikzcd}
\]

\begin{figure}[ht]
\centering
\begin{tikzpicture}[scale=1.0,every node/.style={font=\small}]
  \begin{scope}[xshift=-5.0cm]
    \coordinate (a) at (0,0); \coordinate (b) at (2.2,0);
    \coordinate (c) at (0,2.2); \coordinate (d) at (2.2,2.2);
    \draw[thick] (a)--(b)--(d)--(c)--cycle;
    \draw[thick] (a)--(d);
    \fill (a) circle (1.5pt) node[below left] {$v_1$};
    \fill (b) circle (1.5pt) node[below right] {$v_2$};
    \fill (c) circle (1.5pt) node[above left] {$v_3$};
    \fill (d) circle (1.5pt) node[above right] {$v_4$};
    \node at (1.1,-0.7) {$\cR_+$};
  \end{scope}
  \begin{scope}[xshift=-1.3cm]
    \coordinate (a) at (0,0); \coordinate (b) at (2.2,0);
    \coordinate (c) at (0,2.2); \coordinate (d) at (2.2,2.2);
    \draw[thick] (a)--(b)--(d)--(c)--cycle;
    \draw[thick] (b)--(c);
    \fill (a) circle (1.5pt) node[below left] {$v_1$};
    \fill (b) circle (1.5pt) node[below right] {$v_2$};
    \fill (c) circle (1.5pt) node[above left] {$v_3$};
    \fill (d) circle (1.5pt) node[above right] {$v_4$};
    \node at (1.1,-0.7) {$\cR_-$};
  \end{scope}
  \begin{scope}[xshift=2.4cm]
    \coordinate (a) at (0,0); \coordinate (b) at (2.2,0);
    \coordinate (c) at (0,2.2); \coordinate (d) at (2.2,2.2);
    \coordinate (o) at (1.1,1.1);
    \draw[thick] (a)--(b)--(d)--(c)--cycle;
    \draw[thick] (a)--(o)--(b) (c)--(o)--(d);
    \fill (a) circle (1.5pt) node[below left] {$v_1$};
    \fill (b) circle (1.5pt) node[below right] {$v_2$};
    \fill (c) circle (1.5pt) node[above left] {$v_3$};
    \fill (d) circle (1.5pt) node[above right] {$v_4$};
    \fill (o) circle (1.5pt) node[above right] {$v_0$};
    \node at (1.1,-0.7) {$\cR_{\star}$};
  \end{scope}
\end{tikzpicture}
\caption{The two unimodular triangulations of the ray square and their common
star refinement.  The pictures are transverse slices of the
three-dimensional cones.}
\label{fig:square-refinements}
\end{figure}
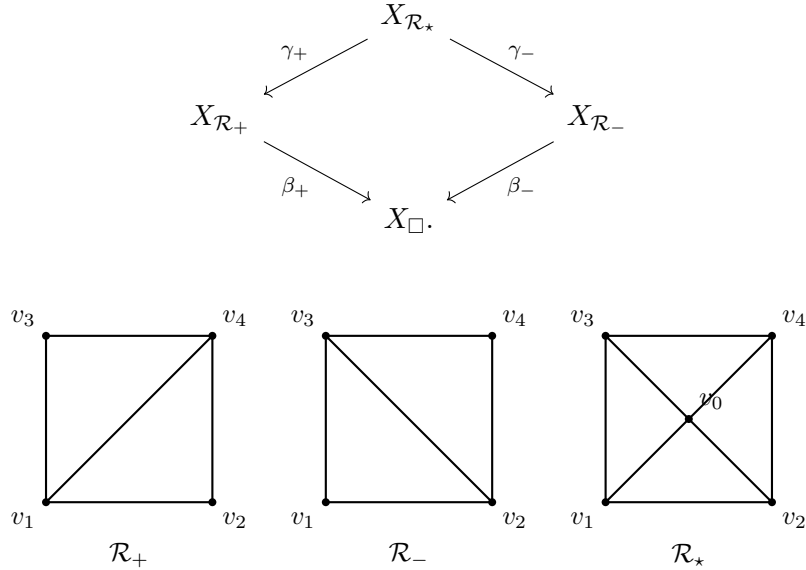

\subsection{Cellular comparison}

Each diagonal triangulation has two two-cells, five one-cells, and four
vertices in the transverse slice.  The common star refinement has four
two-cells, eight one-cells, and five vertices.  The subdivision chain maps
\[
  C_\bullet(\cR_+)\longrightarrow C_\bullet(\cR_\star),
  \qquad
  C_\bullet(\cR_-)\longrightarrow C_\bullet(\cR_\star)
\]
are chain-homotopy equivalences by \cref{thm:subdivision-equivalence}.

The extra central vertex and radial edges do not define new cellular homology.
Their reduced cellular contribution is contractible.  This is the finite
combinatorial prototype of the \emph{initial} reduced refinement sector; by
itself it is not a contraction of the full-differential logarithmic or BV
field complex.

\begin{definition}[Product--Whitney conifold class]\label{def:conifold-product-whitney}
Let $v$ be the vertex of a compact toric truncation $U\subset X_\square$ and
let $E=\beta_\star^{-1}(v)$ be the exceptional square.  Choose the star
resolution's dual square cellulation $K_v=E$, with four vertices, four edges
and one two-cell, and choose compatible finite face cellulations away from
$v$.  The ray $v_0$, the four radial two-cones and the four maximal cones are
dual to that two-cell, its edges and its vertices.  On every intrinsic face
$F$ choose a bounded finite-dimensional linear subcomplex
$W_F^\bullet\subset\logOmega^\bullet(F;\cV|_F)$ which is preserved by the
chosen regularized face restrictions.  Require these restrictions to compose
strictly and choose the DPP normal sections to be product-type on the toric
vertex collar.

For the two angular interval factors choose finite invariant subcomplexes
\(\mathfrak L_1,\mathfrak L_2\) as in
\eqref{eq:finite-separated-square}.  For every product cell
$C=C_1\times C_2\leq E$, let $L_C^\bullet$ be the restriction to $C$ of
\(\mu(\mathfrak L_1\otimes\mathfrak L_2)\).  Require the following:
\begin{enumerate}[label=\textup{(\roman*)},leftmargin=2.5em]
\item the $L_C$ contain all angular logarithmic factors obtained by pulling the
selected unresolved monomials to the corresponding resolved face and applying
the chosen regularized normal restriction;
\item the chosen interval factors are closed under the de Rham differential,
endpoint restrictions, and the two-ended primitive of
\cref{lem:two-ended-primitive}; consequently the resulting finite product
\emph{total} complex is preserved by $D_E$ and $H_E^{\log}$, while $I_E$ and
$P_E$ restrict to the finite source and target of
\cref{prop:product-square-log-contraction}.  No componentwise naturality of
$H_E^{\log}$ with respect to individual normal-cell inclusions is assumed;
\item for an intrinsic face $F\supset v$, the support-changing restriction
from its strict transform to a cell $C$ over $v$ factors through a map
\[
 R_{C/F}:W_F^\bullet\longrightarrow W_v^\bullet\otimes L_C^\bullet
\]
whose two components are the chosen coefficient restriction
$r^W_{v/F}$ and the regularized angular restriction.  On a vertex cell
$u\leq E$, where $L_u^0=\mathbb R$, this factorization is normalized by
\[
 (\id\otimes\epsilon_u)R_{u/F}=r^W_{v/F},
 \qquad \epsilon_u(1)=1,
\]
and all the maps $R_{C/F}$ are compatible with flags.
\end{enumerate}
Existence of these finite invariant subcomplexes is an explicit hypothesis on
the coefficient class.  No completion and no closure under arbitrary products
is imposed.  In particular the class contains constants, but is not restricted
to coefficients constant in the angular variables.

The vertex-support part of the resolved coefficient complex is now defined
from the actual logarithmic face diagram,
\begin{equation}\label{eq:conifold-pw-vertex-sector}
 \mathscr L_{v,\mathrm{pw}}^m
 =\bigoplus_{C\leq E}
 \bigl(W_v^\bullet\otimes L_C^\bullet\bigr)
 ^{m-(3-\dim C)},
\end{equation}
with the DPP differential and oriented regularized face restrictions.  Away
from $v$ there is one unchanged support summand $W_F^\bullet$ for each
intrinsic face $F\ne v$.  A \emph{product--Whitney conifold coefficient
complex} is the finite resolved subcomplex obtained by adjoining
\eqref{eq:conifold-pw-vertex-sector} to these unchanged summands and the
support-changing maps in \textup{(iii)}.  Thus no cellular splitting of the
logarithmic sector is assumed in the definition; it will be produced by
\cref{prop:product-square-log-contraction}.

The intrinsic graded target is
\[
 \mathscr B_{U,W}^m
 =\bigoplus_{F\in\mathsf{Face}(U)}W_F^{m-|F|}.
\]
\end{definition}

\begin{example}[An explicit nonconstant product--log complex]
\label{ex:explicit-product-log}
Take scalar coefficients \(W_F=\mathbb R\) with identity restrictions and put
$x_a=\lambda_{q_a}$.  The toric chart calculation is encoded by the pairing
table
\[
\begin{array}{c|ccccc}
 &v_1&v_2&v_4&v_3&v_0\\
\hline
q_1&0&0&1&1&1\\
q_2&1&1&0&0&1\\
q_3&0&1&1&0&1\\
q_4&1&0&0&1&1
\end{array}
\]
obtained directly from \eqref{eq:conifold-facets}.  On a maximal star chart
$\tau=\Cone(v_i,v_j,v_0)$ with canonical monoidal boundary coordinates
$(\rho_i,\rho_j,\rho_0)$, it gives
\begin{equation}\label{eq:conifold-chart-monomial-pullback}
 \beta_\star^*x_a
 =\rho_i^{\langle q_a,v_i\rangle}
  \rho_j^{\langle q_a,v_j\rangle}\rho_0,
 \qquad
 \Reg_{\rho_0}\log\beta_\star^*x_a
 =c_0+\langle q_a,v_i\rangle\log\rho_i
      +\langle q_a,v_j\rangle\log\rho_j,
\end{equation}
where $c_0=\Reg_{\rho_0}(\log\rho_0)$.  On the four standard product
parametrizations of the exceptional square, the angular boundary coordinates
are projective toric monomials in the two interval variables.  Equivalently,
each $\log\rho_i$ and $\log\rho_j$ is an integral linear combination of
\[
 \log t_1,\quad \log(1-t_1),\quad
 \log t_2,\quad \log(1-t_2);
\]
for example, transition coordinates include the familiar ratios
$t_k/(1-t_k)$ and $(1-t_k)/t_k$.  Thus every angular logarithm in
\eqref{eq:conifold-chart-monomial-pullback} belongs to the separated span of
$\{1,\log t,\log(1-t)\}$ in each interval factor; the relation
$q_1+q_2=q_3+q_4$ is preserved because the exponents are lattice pairings.

On \(\overline I=[0,1]\), use the ordered bases
\[
\begin{split}
 f_0&=1,\qquad f_1=t,\qquad
 f_2=\log t-a,\qquad f_3=\log(1-t),\\
 \eta_1&=\dd t,\qquad
 \eta_2=\frac{\dd t}{t},\qquad
 \eta_3=-\frac{\dd t}{1-t},
\end{split}
\]
where \(a=\Reg_0(\log t)\) and
\(b=\Reg_1(\log(1-t))\).  The logarithms are smooth at the opposite endpoint,
so \(\Reg_1(\log t)=\Reg_0(\log(1-t))=0\).  Put
\[
 \mathfrak L^0=\langle f_0,f_1,f_2,f_3\rangle,\qquad
 \mathfrak L^1=
 \langle\eta_1,\eta_2,\eta_3\rangle\oplus
 \mathbb Rv_0\oplus\mathbb Rv_1.
\]
Relative to these bases, the differential, homotopy and projection are
\[
 [D_I]=
 \begin{pmatrix}
 0&1&0&0\\
 0&0&1&0\\
 0&0&0&1\\
 -1&0&0&0\\
 1&1&-a&b
 \end{pmatrix},
 \qquad
 [H_I]=
 \begin{pmatrix}
 0&0&0&0&0\\
 1&0&0&0&0\\
 0&1&0&0&0\\
 0&0&1&0&0
 \end{pmatrix},
\]
\[
 [P_I^0]=\begin{pmatrix}1&0&0&0\end{pmatrix},
 \qquad
 [P_I^1]=
 \begin{pmatrix}
 0&0&0&1&0\\
 -1&a&-b&0&1
 \end{pmatrix}.
\]
Here the degree-one basis is
\((\eta_1,\eta_2,\eta_3,v_0,v_1)\); \(I_I\) sends \(e\) to \(f_0\) and
sends the two cellular vertices to \(v_0,v_1\).  These matrices verify
directly that \(\mathfrak L\) is preserved by \(D_I,I_I,P_I,H_I\).
The endpoint restriction rows on degree zero are
\[
 R_0=\begin{pmatrix}1&0&0&0\end{pmatrix},
 \qquad
 R_1=\begin{pmatrix}1&1&-a&b\end{pmatrix}.
\]

For the exceptional square take
\(\mathfrak R_{E,\mathrm{sep}}=\mu(\mathfrak L\otimes\mathfrak L)\).
The square differential and the maps \(I_E,P_E,H_E^{\log}\) are the displayed
Kronecker products and the tensor homotopy
\eqref{eq:square-tensor-homotopy}.  All edge and vertex restrictions are
\(R_i\otimes\id\), \(\id\otimes R_j\), or \(R_i\otimes R_j\), so the flag
identities of the coefficient system are literal matrix identities.  This
concerns composition of the restrictions, not commutation of
$H_E^{\log}$ with them; see
\cref{rem:two-ended-not-face-natural}.  This gives a nonconstant member of
\cref{def:conifold-product-whitney}, containing \(f_2\otimes1\),
\(1\otimes f_3\), and \(f_2\otimes f_3\), rather than only constants.
\end{example}

\begin{theorem}[Finite product--Whitney conifold strictification]\label{thm:conifold-g4}
For every coefficient complex in
\cref{def:conifold-product-whitney}, the common star resolution carries a
global centre-compatible, incidence-local deformation-retract datum satisfying
all clauses of \textup{(G4)} after normal-face totalization.  No componentwise
naturality of its contracting homotopy with respect to individual cells of the
exceptional square is asserted.  Its perturbation series are finite, its actual
reduced full-differential sector is $\ker p_\infty$, and the transferred
differential is the signed intrinsic cellular incidence differential, with
coefficient arrows $r^W_{G/F}$, on $\mathscr B_{U,W}^\bullet$.
\end{theorem}

\begin{proof}
First apply \cref{prop:product-square-log-contraction}, tensored with
$W_v^\bullet$, to the vertex-support complex
\eqref{eq:conifold-pw-vertex-sector}.  After the external face shift by one it
gives a special deformation retract
\begin{equation}\label{eq:conifold-angular-sdr}
 (\mathscr L_{v,\mathrm{pw}}^\bullet,D_{v,\log})
 \underset{I_{v,\log}}{\overset{P_{v,\log}}{\rightleftarrows}}
 \bigl(W_v^\bullet\otimes
 C_{\mathrm{face}}^\bullet(K_v),d_W+\partial\bigr),
 \qquad H_{v,\log},
\end{equation}
where
$C_{\mathrm{face}}^r(K_v)=C_{3-r}(K_v)$ for $1\leq r\leq3$.
This is the step which retains the tangential logarithms on $E$; it is not an
identification of those logarithms with constants.

Orient the vertices of $K_v$ cyclically as $u_0,u_1,u_2,u_3$, write
$e_j=[u_j,u_{j+1}]$ with indices modulo four, and orient its two-cell $s$ by
$\partial s=e_0+e_1+e_2+e_3$.  Let
$\epsilon:C_\bullet(K_v)\to\mathbb R$ be augmentation and
$\iota_0(1)=u_0$.  Define
\begin{align*}
 H_Eu_0&=0,&H_Eu_1&=e_0,&H_Eu_2&=e_0+e_1,&H_Eu_3&=-e_3,\\
 H_Ee_0&=0,&H_Ee_1&=0,&H_Ee_2&=s,&H_Ee_3&=0,&H_Es&=0.
\end{align*}
A direct boundary calculation gives the special contraction
\[
 \partial H_E+H_E\partial=\id-\iota_0\epsilon,
 \qquad H_E\iota_0=\epsilon H_E=H_E^2=0.
\]
Tensor this cellular contraction with $W_v^\bullet$, with the Koszul sign, and
compose it with \eqref{eq:conifold-angular-sdr}.  On the vertex-support sector
the resulting initial maps are
\begin{equation}\label{eq:conifold-initial-sdr}
 i_0=I_{v,\log}(\id\otimes\iota_0),\qquad
 p_0=(\id\otimes\epsilon)P_{v,\log},
\end{equation}
and
\begin{equation}\label{eq:conifold-composite-homotopy}
 h_0=H_{v,\log}
 +I_{v,\log}(\id\otimes H_E)P_{v,\log},
\end{equation}
where the tensor Koszul signs are understood.
Use the identity, identity and zero, respectively, on every other support
summand.  Let $D_0$ be the actual internal differential on the unchanged
summands and the actual same-support logarithmic face differential
$D_{v,\log}$ on \eqref{eq:conifold-pw-vertex-sector}.  The composition rule for
special deformation retracts gives
\[
 D_0h_0+h_0D_0=\id-i_0p_0,
 \qquad h_0i_0=p_0h_0=h_0^2=0.
\]
Thus the initial kernel contains both the angular logarithmic kernel of
$P_{v,\log}$ and the inverse image of the reduced cellular square; the first
is contracted by $H_{v,\log}$ and the second by the lifted $H_E$.  This is the
two-stage reduced sector required by the strengthened clause of
\textup{(G4)}.

Let $\delta=D-D_0$.  By construction it is precisely the sum of resolved
codimension-one incidences which change intrinsic support.  The image of
$h_0$ is supported over the minimal intrinsic face $v$, from which there is no
further support-decreasing incidence.  Consequently
\[
 \delta h_0=0,
 \qquad
 (h_0\delta)^2=0.
\]
The perturbation formulas are therefore finite and the higher correction
$p_0\delta h_0\delta i_0$ vanishes.

For completeness, the surviving matrix entries of $p_0\delta i_0$ can be read
directly from the star fan.  Its rays, two-cones and maximal cones give the
following four types:
\begin{center}
\begin{tabular}{lll}
\toprule
resolved incidence & intrinsic support change & aggregation\\
\midrule
bulk $\to H_{v_i}$ & $U\to H_i$ & identity\\
$H_{v_i}\to H_{\Cone(v_i,v_j)}$ on the perimeter
 & $H_i\to L_{ij}$ & identity\\
$H_{\Cone(v_i,v_j)}\to u_{ij}$ & $L_{ij}\to v$
 & $\epsilon(u_{ij})=1$\\
bulk $\to H_{v_0}$ or $H_{v_i}\to H_{\Cone(v_i,v_0)}$
 & jump to $v$ by codimension $>1$ & zero\\
\bottomrule
\end{tabular}
\end{center}
Here the last row lands in the two-cell or an edge of $K_v$, on which cellular
augmentation is zero.  Each perimeter two-cone belongs to one maximal cone,
so the third row has multiplicity one.  Orienting the dual cells by the fan
orientation makes its sign exactly $[L_{ij}:v]$; the first two rows inherit
$[U:H_i]$ and $[H_i:L_{ij}]$ from the intrinsic coorientations of the
corresponding strict transforms.  The chosen coefficient restriction system
supplies the first two coefficient factors, while clause \textup{(iii)} of
\cref{def:conifold-product-whitney} supplies the third.  They are, respectively,
\[
 r^W_{H_i/U},\qquad r^W_{L_{ij}/H_i},\qquad r^W_{v/L_{ij}}.
\]
The normalization
$(\id\otimes\epsilon_{u_{ij}})R_{u_{ij}/L_{ij}}
=r^W_{v/L_{ij}}$ is exactly what removes any chart-dependent scalar in the
third row.  Thus for every intrinsic cover relation $G\prec F$,
\[
 \pi_Gp_0\delta i_0\jmath_F=[F:G]r^W_{G/F}.
\]
Hence
\[
 D_{\mathrm{str}}=d_W+p_0\delta i_0
\]
is precisely the signed intrinsic incidence differential, with no hidden
multiple or nonlocal component.

It remains to verify precisely the compatibility required in \textup{(G4)}.
Here the centre $v$ is a point.  Fix the global toric product collar, the
cyclic ordering $u_0,u_1,u_2,u_3$, and the product-type DPP normal sections
used above.  Consequently no transition-function equivariance of the anchored
cellular homotopy $H_E$ is required.  The interval and square maps are
$W_v^\bullet$-linear, while compatibility with restrictions inside the centre
is vacuous.  The normal-cell incidences themselves are already part of
$D_{v,\log}$ in \eqref{eq:conifold-angular-sdr}; by
\cref{rem:two-ended-not-face-natural}, no componentwise commutation of
$H_{v,\log}$ with the individual edge and vertex restrictions is claimed or
needed.

Choose the outer overlap of the toric vertex collar disjoint from
\(\beta_\star^{-1}(v)\).  Every exceptional face summand in
\eqref{eq:conifold-pw-vertex-sector} restricts to zero on that overlap, while
the persistent strict-transform summands carry the identity, identity and zero
as \(i_0,p_0,h_0\).  Thus the collar maps and the exterior maps have identical
restrictions on the overlap and glue by the sheaf axiom; no cutoff is inserted
into the homotopy.  This verifies the
initial contraction, the two-stage associated-graded reduced sector,
nilpotence, incidence locality, centre-restriction compatibility and gluing required
by \textup{(G4)}.  Finally
\cref{prop:filtered-exceptional-contraction} identifies the contracted
full-differential summand as $\ker p_\infty$.
\end{proof}

\begin{theorem}[Finite cyclic logarithmic--cellular conifold BF shadow]
\label{thm:conifold-finite-bf}
Let \(\mathscr C_L^A\) be a finite product--Whitney logarithmic coefficient
complex of \cref{def:conifold-product-whitney}, and let
\(\mathscr C_X^A=\mathscr B_{U,W}\) be its finite intrinsic incidence target.
Give their shifted algebraic BF duals
\[
 \mathscr C_L^B=(\mathscr C_L^A)^{\vee_{\mathrm{BF}}},
 \qquad
 \mathscr C_X^B=(\mathscr C_X^A)^{\vee_{\mathrm{BF}}}
\]
the dual face grading.  Then the star model carries a nondegenerate cyclic
finite BF pairing, and the retract of \cref{thm:conifold-g4} dualizes to a
strict bivariant incidence model: \(A\) restricts contravariantly and \(B\)
corestricts covariantly.

This is a logarithmic--cellular BF shadow.  It is not the field datum
\textup{(G5)} of \cref{def:bf-cyclic-enhancement}: in particular,
\(\mathscr C_X^A\) is not identified with
\(\Omega_b^\bullet(U;V)[1]\), and no assertion about the full intrinsic
\(b\)-de Rham field complex follows without an additional type-correct
comparison.
\end{theorem}

\begin{proof}
Use the \emph{perturbed} finite \(A\)-retract supplied by
\cref{thm:conifold-g4,prop:filtered-exceptional-contraction}:
\[
 (\mathscr C_L^A,D^A)
 \underset{i_\infty^A}{\overset{p_\infty^A}{\rightleftarrows}}
 (\mathscr C_X^A,D_{\mathrm{str}}^A),\qquad h_\infty^A.
\]
The distinction from the initial retract is essential.  In the conifold model
\(\delta_Ah_0^A=0\) and \((h_0^A\delta_A)^2=0\) by
\eqref{eq:conifold-initial-sdr}--\eqref{eq:conifold-composite-homotopy} and the
support calculation in the proof of \cref{thm:conifold-g4}.  Consequently
\[
 p_\infty^A=p_0^A,\qquad h_\infty^A=h_0^A,\qquad
 i_\infty^A=(1+h_0^A\delta_A)^{-1}i_0^A
 =i_0^A-h_0^A\delta_Ai_0^A,
\]
and no vanishing of the last correction is assumed.

Finite-dimensional algebraic duality defines
\[
 i_\infty^B=(p_\infty^A)^\vee,\qquad
 p_\infty^B=(i_\infty^A)^\vee,
\]
with \(D_L^B\) the graded adjoint of \(-D^A\),
\(D_X^B\) the graded adjoint of \(-D_{\mathrm{str}}^A\), and
\(h_\infty^B\) the graded transpose of \(h_\infty^A\).  The evaluation
pairings are nondegenerate, and
transposing the \(A\)-contraction identities gives the \(B\)-contraction
identities and side conditions.  Thus
\[
 S_X^{\mathrm{fin}}(A,B)=\langle B,D_{\mathrm{str}}^A A\rangle
\]
is a finite cyclic quadratic BF action.

For every inclusion \(G\prec F\), define
\[
 (r^A_{G/F})_!:=(r^A_{G/F})^{\vee_{\mathrm{BF}}}:
 \mathscr C_G^B\longrightarrow\mathscr C_F^B.
\]
The \(A\)-maps compose by the product--Whitney definition and, for a flag
\(H\prec G\prec F\),
\[
 (r^A_{G/F})_!(r^A_{H/G})_!
 =\bigl(r^A_{H/G}r^A_{G/F}\bigr)^{\vee_{\mathrm{BF}}}
 =(r^A_{H/F})_!.
\]
Hence the dual arrows compose in the covariant order.  The surviving
\(A\)-incidence matrix in \cref{thm:conifold-g4} is the signed cellular
coboundary; the \(B\)-matrix is its shifted transpose.  Every step takes place
inside the finite logarithmic--cellular complex on the common star resolution;
no \(b\)-de Rham field complex is inserted.
\end{proof}

\begin{remark}[No finite cyclic claim on the diagonal arms]
The cellular subdivision maps
\(C_\bullet(\cR_\pm)\to C_\bullet(\cR_\star)\) are chain-homotopy
equivalences by \cref{thm:subdivision-equivalence}.  This fact alone does not
show that a selected logarithmic product--Whitney subcomplex, its DPP normal
sections and its cyclic pairing are preserved by subdivision.  Since no such
diagonal logarithmic subcomplexes and compatible cyclic retracts have been
fixed here, \cref{thm:conifold-finite-bf} makes no finite cyclic assertion on
the two diagonal arms.  Their comparison in this paper is instead the
unfiltered derived comparison of complete relative complexes through the
common interior, stated below in \cref{thm:conifold-trace}.
\end{remark}

\subsection{Abelian BF theory, twice}

The resolution-level and intrinsic analyses now give complementary
statements.  The ordinary resolution-level result compares the descent
systems of the two resolutions; it is not by itself an equivalence with the
intrinsic $b$-source.

\begin{theorem}[Abelian BF on the positive real conifold]\label{thm:conifold-bf}
Let $U\subset X_\square$ be a compact truncated neighborhood of the vertex,
and let $U_{\cR_+}$ and $U_{\cR_-}$ be the two smooth monoidal
resolutions induced by the diagonal triangulations.  For abelian BF theory
with product-type behavior near the exceptional set, there is a canonical
zigzag in the derived category
\[
  \Tot(U_{\cR_+};\cA_{\mathrm{BF}})
  \xrightarrow{\simeq}
  \Tot(U_{\cR_\star};\cA_{\mathrm{BF}})
  \xleftarrow{\simeq}
  \Tot(U_{\cR_-};\cA_{\mathrm{BF}}).
\]
In particular, the two resolution-level classical linear descent theories
have isomorphic cohomology.
\end{theorem}

\begin{proof}
The common star refinement is obtained from either diagonal triangulation by
inserting the ray through the interior point of its diagonal edge.
Equivalently, it is the star subdivision of the two-dimensional cone
corresponding to that diagonal.  On the resolved manifold this is realized by
an elementary toric blow-up of the associated codimension-two boundary face,
locally modeled by a radial blow-up whose exceptional transverse fiber is an
interval and hence contractible.  The BF coefficient system is homotopy
constant along that fiber under the product-type hypothesis.
\Cref{thm:bf-facewise} therefore gives the two derived roofs through the
common interior, and hence the displayed zigzag after choosing the star
refinement as the intermediate representative.
\end{proof}

\begin{remark}
The word canonical refers to the roof determined by the chosen common
refinement in the derived category.  It does not assert a canonical
diffeomorphism $U_{\cR_+}\cong U_{\cR_-}$.  Indeed, the resolutions
correspond to different diagonal choices.
\end{remark}

The intrinsic $b$/logarithmic analysis instead compares corresponding total
relative trace classes carried by the two resolutions.

\begin{theorem}[Derived conifold trace comparison]\label{thm:conifold-trace}
Choose nondegenerate DPP regularizations $s_+$, $s_\star$, and $s_-$ on the
complete ordered boundary diagrams of $\cR_+$, $\cR_\star$, and $\cR_-$,
respectively.  Their regularized total complexes are identified by
trace-compatible roofs through the common interior.  Thus,
if total cocycles $a_+,a_\star,a_-$ represent the same class under these roofs,
\[
 \mathcal I_{\cR_+,s_+}(a_+)
 =\mathcal I_{\cR_\star,s_\star}(a_\star)
 =\mathcal I_{\cR_-,s_-}(a_-).
\]
For $\varphi\in\Omega_c^3(X_\square^\circ)$ this applies to the three interior
representatives and their common value is
$\int_{X_\square^\circ}\varphi$.
An intrinsic vertex scale $(c_1,c_2,c_3,c_4)$ must satisfy
$c_1+c_2=c_3+c_4$ by \cref{prop:scale-relation}, but this vertex datum neither
determines nor is needed to compare the three chosen DPP regularizations.

For the product--Whitney coefficient class of
\cref{def:conifold-product-whitney}, the star resolution supplies the finite
global datum of \textup{(G4)} by \cref{thm:conifold-g4}; hence the actual
full-differential reduced sector contracts and the trace/descent complex
strictifies.  On the finite algebraic-dual BF sub-class,
\cref{thm:conifold-finite-bf} supplies path-independent
$A$-restrictions/$B$-corestrictions and a strict bivariant finite
logarithmic--cellular BF shadow.  It does not supply \textup{(G5)}, and no such
conclusion is asserted here for the full continuum field classes.
\end{theorem}

\begin{proof}
The trace-compatible roofs and the displayed equality are
\cref{thm:refinement-invariance,prop:tracecompat}; they compare total cocycles,
not arbitrary face-degree-zero logarithmic forms.  The assertion about
$\varphi$ follows from the interior trace identity
\eqref{eq:interior-trace-bridge}.

For either diagonal-to-star arm, the new exceptional cellular incidence data
form an interval, and the $c=2$ case of \eqref{eq:barycentric-cone} is its
local cellular contraction.  This observation does not assert preservation of
a finite logarithmic cyclic subcomplex on either arm.  The global vertex fibre of the star resolution is instead the
exceptional square; its explicit cellular contraction, support-filtered perturbation and
gluing maps are constructed in \cref{thm:conifold-g4}.  Applying
\cref{prop:filtered-exceptional-contraction} gives the asserted
strictification, while \cref{thm:conifold-finite-bf} gives the finite cyclic
logarithmic--cellular BF shadow.
\end{proof}

The results expose complementary parts of the same mechanism.  The common
interior gives the unconditional derived comparison and equality of trace
values on corresponding compactly supported total-cocycle classes.  The
exceptional intervals control the cellular geometry of the two refinement
arms, while the exceptional square supplies the global contraction for the
product--Whitney class.  Its
finite algebraic dual supplies the cyclic BF shadow.  These constructions do
not identify arbitrary continuum models, and no equality of raw absolute
regularized bulk integrals is inferred from the common refinement.

\section{A second example: the cone over a pentagon}\label{sec:pentagon}

The conifold is the smallest non-ordinary $g$-corner, and its symmetry is
special: the dihedral group of the square forces its unique monoidal
relation to have unit coefficients, and both of its natural triangulations
happen to already be smooth, so \cref{thm:kottke} is never seen doing more
than choosing a diagonal. Neither feature survives at the next rank-three
case, the cone over a pentagon, and this section records what changes.

\subsection{The monoid}

Let $\sigma\subset\R^3$ be the cone over the pentagon with vertices, in
cyclic order,
\begin{equation}\label{eq:pentagon-rays}
 v_1=(0,0,1),\ v_2=(1,0,1),\ v_3=(2,1,1),\ v_4=(1,2,1),\ v_5=(0,1,1),
\end{equation}
a genuine convex lattice pentagon: the five edge vectors turn consistently
to the left, so $\sigma$ is a strictly convex rational polyhedral cone with
five rays. Let $P=\sigma^\vee\cap\Lambda^*$. Its facet normals, one per
edge of the pentagon, are the primitive vectors
\begin{equation}\label{eq:pentagon-facets}
\begin{aligned}
 q_1&=(0,1,0)&&\text{for the facet }\Cone(v_1,v_2),\\
 q_2&=(-1,1,1)&&\text{for the facet }\Cone(v_2,v_3),\\
 q_3&=(-1,-1,3)&&\text{for the facet }\Cone(v_3,v_4),\\
 q_4&=(1,-1,1)&&\text{for the facet }\Cone(v_4,v_5),\\
 q_5&=(1,0,0)&&\text{for the facet }\Cone(v_5,v_1),
\end{aligned}
\end{equation}
each vanishing on the stated pair of rays and strictly positive on the
other three, exactly as for the conifold's $q_i$ in
\cref{eq:conifold-facets}. These five vectors generate all of $\Z^3$: directly,
$q_5=(1,0,0)$, $q_1=(0,1,0)$, and $q_2-q_1+q_5=(0,0,1)$. Hence $P^{\gp}=\Z^3$: this
$P$ has the same rank as the conifold's, which is what makes the two
examples directly comparable.

\begin{proposition}[Pentagon group relations]\label{prop:pentagon-relations}
The five primitive dual-ray vectors in \eqref{eq:pentagon-facets}, which
generate $P^{\gp}$ as a group, satisfy exactly two independent integral
relations,
\begin{equation}\label{eq:pentagon-monoid-relations}
 2q_1+q_3=2q_2+q_4,
 \qquad
 2q_1+q_4=q_2+2q_5.
\end{equation}
\end{proposition}

\begin{proof}
Both identities are checked directly from \eqref{eq:pentagon-facets}: e.g.
$2q_1+q_3=(0,2,0)+(-1,-1,3)=(-1,1,3)$ and
$2q_2+q_4=(-2,2,2)+(1,-1,1)=(-1,1,3)$.  Let
\[
 r^{(1)}=(2,-2,1,-1,0),\qquad
 r^{(2)}=(2,-1,0,1,-2)
\]
be the corresponding relation vectors.  Since
$\operatorname{rank}P^{\gp}=3$, the integral kernel of the $3\times5$ matrix
with columns $q_1,\ldots,q_5$ has rank two.  It remains to check saturation,
not merely linear independence.  If
$a=(a_1,\ldots,a_5)$ is any integral relation, put
$\lambda=a_3$ and $\mu=-a_2-2a_3$.  Then
$\lambda r^{(1)}+\mu r^{(2)}$ has the same second and third coordinates as
$a$.  Their difference is a relation supported on $q_1,q_4,q_5$; these three
vectors form a unimodular basis of $\mathbb Z^3$, so that difference is zero.
Thus every integral relation is a unique integral combination of
$r^{(1)},r^{(2)}$, proving the claim.
\end{proof}

Unlike the conifold's single relation, the two relations here do not have
unit coefficients.  No lattice symmetry of the displayed realization is used
to force a unit-coefficient presentation analogous to the square's
diagonal-preserving reflections.  This statement concerns the chosen lattice
embedding, not the dihedral automorphism group of the abstract pentagonal face
poset.

\begin{proposition}[Pentagon scale relation]\label{prop:pentagon-scale}
A monoidal scale at the vertex of $X_P$, \cref{def:vertex-scale}, is
equivalent to a quintuple $(c_1,\ldots,c_5)\in\R^5$ satisfying
\[
 2c_1+c_3=2c_2+c_4,
 \qquad
 2c_1+c_4=c_2+2c_5,
\]
with $\Reg(\log x_i)=c_i$ for the monomial functions $x_i=\lambda_{q_i}$.
\end{proposition}

\begin{proof}
As in \cref{prop:scale-relation}: set $c_i=\scale(q_i)$ for a scale
$\scale:P^{\gp}\to\R$ and apply $\scale$ to
\eqref{eq:pentagon-monoid-relations}; conversely any quintuple satisfying
these two linear relations defines a unique homomorphism on the presented
group $P^{\gp}$.
\end{proof}

Thus the five chosen monomial values determine a three-parameter, not a
five-parameter, family of intrinsic vertex scales.  They do \emph{not}
parameterize DPP regularizations on a smooth resolution: such a regularization
is compatible normal-section data on the complete ordered boundary diagram,
as emphasized in \cref{prop:vertex-character-insufficient}.  The two
independent constraints are exactly the two relation directions measured by
$5-\operatorname{rank}P^{\gp}=2$.

\subsection{A resolution that is not just a choice of diagonal}

A convex pentagon triangulates using two non-crossing diagonals from a
single vertex; from $v_1$ these are $v_1v_3$ and $v_1v_4$, giving the three
cones $\Cone(v_1,v_2,v_3)$, $\Cone(v_1,v_3,v_4)$, $\Cone(v_1,v_4,v_5)$.
Direct computation of the three determinants against the standard basis
gives $1$, $3$, and $1$: the outer two cones are already unimodular, but
the middle one is not. This is the point of departure from the conifold,
where both diagonal triangulations of the square are unimodular for free;
here the naive triangulation is not smooth, and producing a smooth
refinement genuinely requires the extra ray constructed in
\cref{prop:pentagon-resolve}, followed by the geometric realization of
\cref{thm:kottke}, not merely a choice of diagonal.

\begin{proposition}[Resolving the singular cone]\label{prop:pentagon-resolve}
The semi-open fundamental parallelepiped of the cone
$\tau=\Cone(v_1,v_3,v_4)$ contains exactly the three lattice points
$0$, $w$, and $2w$, where
\[
 w=(1,1,1)=\tfrac13(v_1+v_3+v_4)
\]
is primitive. Star-subdividing $\tau$ at the ray through $w$ produces the
three cones
\[
 \Cone(v_3,v_4,w),\qquad\Cone(v_1,v_4,w),\qquad\Cone(v_1,v_3,w),
\]
with determinants $1$, $-1$, $1$: all three are unimodular.
\end{proposition}

\begin{proof}
Writing $\tau=\Cone(v_1,v_3,v_4)$ in the basis $(v_1,v_3,v_4)$, a point of
$\Z^3$ lies in the fundamental parallelepiped iff its $(v_1,v_3,v_4)$-coordinates
lie in $[0,1)^3$; since $\det\tau=3$, exactly three lattice points do so,
and direct computation identifies them as $0$, $(1,1,1)$ at coordinates
$(\tfrac13,\tfrac13,\tfrac13)$, and $(2,2,2)$ at coordinates
$(\tfrac23,\tfrac23,\tfrac23)$ -- the latter lying on the same ray as the
former, so $w=(1,1,1)$ is the unique new ray the singularity requires. The
three determinants against the standard basis are computed directly from
$v_1,v_3,v_4,w$.
\end{proof}

Combining \cref{prop:pentagon-resolve} with the two already-smooth cones
$\Cone(v_1,v_2,v_3)$ and $\Cone(v_1,v_4,v_5)$ gives a smooth refinement
\[
 \cR_P=\bigl\{\Cone(v_1,v_2,v_3),\ \Cone(v_1,v_3,w),\ \Cone(v_3,v_4,w),\
 \Cone(v_1,v_4,w),\ \Cone(v_1,v_4,v_5)\bigr\}
\]
of $\sigma$, with five maximal cones against the pentagon's three naive
triangles. By \cref{thm:kottke} this is realized by a blow-up
$\beta_{\cR_P}:X_{\cR_P}\to X_P$ with ordinary corners, to which
\cref{sec:resolution-aksz-data,sec:derived-traces} apply exactly as they do
for the conifold's $\cR_+$.  After choosing an arbitrary DPP regularization
$s_{\cR_P}$ on the complete ordered boundary diagram,
\cref{thm:refinement-invariance} gives the resolution-independent affine
derived trace.  A quintuple $(c_1,\ldots,c_5)$ satisfying
\cref{prop:pentagon-scale} is separate intrinsic vertex-scale input; it neither
determines the DPP datum nor is required for the common-interior comparison.
The same theorem applies to the two arms of any smooth common refinement with
a second smooth resolution built from a different starting triangulation,
without the resolutions needing to agree diagonal by diagonal.  A filtered
facewise comparison would additionally require the transfer data of
\cref{def:refinement-descent-datum}.  We do not carry out
a second resolution here in the same explicit detail as $\cR_P$: the point
of this section is that \cref{prop:pentagon-resolve,thm:kottke} are doing
genuine resolution-of-singularities work already at rank three once the
square's accidental symmetry is removed, and that the abstract
refinement-invariance machinery of \cref{sec:derived-traces} was built to
absorb exactly this, not only the conifold's single interval.

\subsection{Combinatorics}

The pentagon's flags -- a ray together with one of the two facets through
it -- number $5\times2=10$, and the combinatorial dihedral group of the
abstract pentagon, of order ten, acts simply transitively on them.  This is a
face-poset action, not an assertion that the full dihedral group preserves the
chosen lattice embedding in \eqref{eq:pentagon-rays}.  Since $10$ is not a
multiple of $3!=6$, no free
$S_3$-action on $\partial^3X_P$ exists here either, exactly as for the
conifold in \cref{subsec:conifold-joyce-corners}. The diagnosis of \cref{subsec:conifold-joyce-corners} was never special to the square: it is the statement that a
$g$-corner's link is the boundary of a polytope with its own combinatorial
automorphism group, while a chosen lattice model may realize only a subgroup.

\part*{VII.\ Scope and conclusion}\addcontentsline{toc}{part}{VII.\ Scope and conclusion}

\section{Open problems and precise scope}\label{sec:open-problems}

The results above separate three levels.  First, under
\textup{(G1)}--\textup{(G3)} the resolution-independent total trace and its
sheaf descent are relative de Rham statements: the unfiltered object is
$j_!\Omega_{X^\circ}^\bullet$ and does not retain individual descendants.
Second, \textup{(G4)} and \textup{(G5)} are conditional transfer criteria.
When supplied, \textup{(G4)} produces the strict linear logarithmic incidence
model and \textup{(G5)} produces the cyclic classical abelian BF model; their
general existence is not deduced from monoidal resolution theory.  Third, the
paper constructs these structures only in more restricted forms: the
product--Whitney conifold class realizes the full \textup{(G4)} criterion and
a finite algebraic-dual BF shadow, while finite-dimensional cellular
nonabelian BF is invariant under the admitted elementary resolution moves.
None of these statements constructs a general nonlinear continuum BV
pushforward on the unresolved $g$-corner.  Nor do the derived roofs identify
arbitrary absolute face-degree-zero regularized integrals: their invariant
output is the total relative trace morphism.

\begin{conjecture}[Intrinsic multiplicative Stokes trace]\label{conj:trace-existence}
Under suitable analytic hypotheses, the resolution-independent derived trace admits a strict
multiplicative representative on $\mathsf{Face}(X)$.  Its expected uniqueness should be
formulated as coherent equivalence in a homotopy theory of multiplicative trace systems;
constructing that homotopy theory is part of the conjecture.
\end{conjecture}

\begin{conjecture}[Nonlinear strictification]\label{conj:nonlinear-strictification}
For a perturbatively regular unimodular AKSZ target, the derived theory admits a nonlinear
strictification whenever the transferred reduced refinement $L_\infty$ structure is
gauge-equivalent to one with no higher face operations.
\end{conjecture}

\begin{conjecture}[Resolved--intrinsic $b$-equivalence]\label{conj:comparison}
For an AKSZ target for which the exceptional BV pushforward is defined and compatible with
all face restrictions, pushforward of the resolved $b$/logarithmic theory along
a smooth monoidal resolution produces the intrinsic $b$-AKSZ theory up to
canonical transformation and a field-independent determinant factor.  No
equivalence with the ordinary AKSZ source is part of this conjecture.
\end{conjecture}

At loop level one would additionally need logarithmic compactifications of
configuration spaces, extension and integrability of propagator products,
regularized Stokes identities on all boundary strata, and control of
short-loop counterterms.  These analytic questions are deliberately outside
the present paper.  More basically, three existence problems remain distinct:
constructing a strict multiplicative intrinsic trace system, deriving
\textup{(G4)} from usable geometric hypotheses beyond the classes treated
here, and constructing the type-correct continuum field datum \textup{(G5)}.
None follows formally from Kottke's resolution theorem.  The
product--Whitney conifold class of \cref{def:conifold-product-whitney} is the
finite non-ordinary example for which \textup{(G4)} and the separate
logarithmic--cellular BF shadow are constructed directly; it is not an example
of the full field datum \textup{(G5)}.

\section{Conclusion}

The paper studies three related classical source models for AKSZ--BV--BFV
theory on Joyce generalized corners.  Ordinary AKSZ on a Kottke monoidal
resolution supplies a resolution-level baseline.  The unconditional
comparison proved for resolved $b$/logarithmic models is the unfiltered
relative comparison through the common interior.  The intrinsic $b$-model
based on $\bT[1]X$ and Joyce's face category is a separate, conditional layer:
its facewise transgression assumes a multiplicative regularized Stokes trace
system, and its strict linear and cyclic realizations require the additional
transfer criteria \textup{(G4)} and \textup{(G5)}.  The anchor from the
ordinary source to the resolved $b$-source is not asserted to be an
equivalence.

The results therefore have an unconditional core and a conditional
strictification layer.  On bulk $b$-fields, monoidal blow-down is
$b$-\'etale.  On the full face complex, smooth resolutions represent a common
unfiltered relative de Rham object through their common interior; equality of
integrated representatives requires corresponding compactly supported
total-cocycle classes and does not follow merely from equality of bulk fields
on the interior.  A monoidal character fixes intrinsic vertex values but does
not canonically determine DPP normal-section data on a resolution.  If the
finite subdivision/aggregation criterion \textup{(G4)} is supplied, the
reduced logarithmic trace/descent refinement sector contracts and yields a
strict linear intrinsic Stokes complex.  If the separate cyclic field-level
criterion \textup{(G5)} is supplied, including the matching of transferred
terms with the prescribed signed face arrows, one obtains the strict
incidence-level classical abelian BF/BV--BFV--Stokes model without identifying
the two retracts.  Its maximally extended face system is
bivariant: $A$ restricts and the BF-dual $B$ corestricts.  The positive real conifold demonstrates why the
intrinsic face category, rather than $\partial^kX/S_k$, is essential and gives an explicit
common-interior comparison of the two resolutions.  On the specified
product--Whitney class, its exceptional square gives the full global
\textup{(G4)} datum after normal-face totalization; the finite algebraic-dual
sub-class gives instead a cyclic
logarithmic--cellular BF shadow with path-independent
restrictions/corestrictions.  It does not give the full intrinsic
\(b\)-de Rham datum \textup{(G5)}, nor any such datum for arbitrary continuum
coefficients.

The local exceptional interval provides the basic geometric check: its reduced cellular
sector is acyclic, and the one-ended radial/ordinary angular propagator used in
the codimension-two continuum collar is compatible with its endpoint
restrictions.  By contrast, the two-ended relative normal-face retract used
for the finite conifold square is a retract after totalization and is not
componentwise natural on the un-totalized normal-face diagram, as made
explicit in \cref{rem:two-ended-not-face-natural}.  Finite-dimensional
nonabelian cellular BF for unimodular Lie algebras
supplies a nonlinear shadow of the same resolution invariance.  General nonlinear continuum strictification and quantum
logarithmic graph integrals remain separate problems rather than hidden assumptions.

\appendix

\section{Sign convention and transgression--Stokes calculation}\label{app:signs}

We use exactly the local-form convention of the ordinary-corner companion
paper \cite[Apps.~A and C]{AnghelPartI}, with the ordinary Stokes functional
replaced by the regularized trace.  The replacement changes no Koszul sign:
regularized restriction acts on source coefficients, commutes with the
field-space differential, and satisfies \eqref{eq:stokes-trace}.

Let $F$ have dimension $d$.  Write $\delta_{\Y}$ for the target de Rham
differential and $\delta$ for the field-space differential.  In homogeneous
coordinates $z^a$ on $\Y$, a field is a superfield
$Z^a\in\bOmega^\bullet(F)$ of total degree $|z^a|$.  We place source
coefficients before field-space differentials and impose
\begin{equation}\label{eq:sign-anticommute}
 \delta\db=-\db\delta .
\end{equation}
The total degree is the sum of source-form degree, target/internal degree and
field-space form degree; every Koszul exponent below is the reduction of this
integer modulo two.
The regularized fiber trace has degree $-d$; consequently an operator of
parity $\pi$ acquires the sign $(-1)^{\pi d}$ when moved through it.  In
particular,
\begin{equation}\label{eq:sign-trace-delta}
 \delta\,\Tr_F^{\cT}(\xi)
 =(-1)^d\Tr_F^{\cT}(\delta\xi).
\end{equation}
The source differential in the Stokes formula acts only on the source
coefficient, so no additional sign depending on the field-space form degree
is inserted in \eqref{eq:stokes-trace}.

For a homogeneous target form $\chi$, define
\begin{equation}\label{eq:regularized-transgression-k}
 \mathsf T_F^0\chi:=\Tr_F^{\cT}\ev_F^*\chi,
 \qquad
 \mathsf T_F^k\chi
 :=\iota_{\widehat{\db}_F}^{\,k}\mathsf T_F^0\chi .
\end{equation}
This is the local-form convention of \cite[App.~D]{CMRClassical}, applied face
by face.  Here the superscript $k$ records the number of source differentials
inserted into the target form; it is neither a power of the fiber trace nor
the codimension of $F$, which is denoted by $k$ elsewhere in this paper.
Because $\widehat{\db}_F$ has degree one,
$\iota_{\widehat{\db}_F}$ is even in the total parity convention.  On
generators it sends $\delta Z^a$ to $\db Z^a$, annihilates $Z^a$ and
$\db Z^a$, and hence obeys the ungraded Leibniz rule.  The superfield
generating-function calculation, using
\cref{eq:sign-anticommute,eq:sign-trace-delta} and regularized Stokes, gives
\begin{equation}\label{eq:regularized-transgression-stokes}
 \delta\bigl(\mathsf T_F^k\chi\bigr)
 =(-1)^d\left(
   \mathsf T_F^k(\delta_{\Y}\chi)
   -k\sum_{G\prec F}[F:G]\,
      (\rho_{G/F}^{\cT})^*\mathsf T_G^{k-1}\chi
 \right),
\end{equation}
with the boundary sum absent for $k=0$.  Indeed, expanding
$\delta Z^a+\db Z^a$ in every target one-form slot groups the terms with
$k$ source differentials into
$\mathsf T_F^k/k!$; applying
$\delta+\db$ and then \eqref{eq:stokes-trace} produces
\eqref{eq:regularized-transgression-stokes}.  This argument is componentwise
in homogeneous superfields and therefore does not use a nonexistent single
degree for an inhomogeneous field.

Taking $k=0$ and $\chi=\alpha_\Y$ shows, with the definitions in
\cref{def:transgression}, that
\begin{equation}\label{eq:transgressed-primitive-sign}
 \omega_F^b=(-1)^d\delta\alpha_F^b.
\end{equation}
Taking $k=1$ and using
$S_{F}^{\mathrm{kin}}=\mathsf T_F^1\alpha_\Y$ gives
\begin{equation}\label{eq:source-hamiltonian-sign}
 \iota_{\widehat{\db}_F}\omega_F^b
 =(-1)^d\delta S_F^{\mathrm{kin}}
 +\sum_{G\prec F}[F:G]\,
   (\rho_{G/F}^{\cT})^*\alpha_G^b .
\end{equation}

For a homogeneous target vector field $V$, moving its contraction through a
degree-$d$ fiber trace gives
\begin{equation}\label{eq:target-contraction-sign}
 \iota_{\widehat V}\mathsf T_F^k\chi
 =(-1)^{(|V|+1)d}\mathsf T_F^k(\iota_V\chi).
\end{equation}
Since $|Q_\Y|=1$, the exponent in
\eqref{eq:target-contraction-sign} is even.  The target Hamiltonian identity
and the $k=0$ case of
\eqref{eq:regularized-transgression-stokes} therefore imply
\begin{equation}\label{eq:target-hamiltonian-sign}
 \iota_{\widehat Q_\Y}\omega_F^b
 =\mathsf T_F^0(\delta_\Y\Theta)
 =(-1)^d\delta\,\mathsf T_F^0\Theta .
\end{equation}
Adding
\cref{eq:source-hamiltonian-sign,eq:target-hamiltonian-sign} proves
\cref{eq:intrinsic-bvbfv-factorized}.  Finally,
\cref{eq:evaluation-regularization} identifies every boundary summand with
the pullback of the primitive on the boundary field space.  Thus both the
global factor $(-1)^d$ and the oriented boundary sign are fixed by explicit
operator conventions.

\bigskip
\noindent\textit{Author's address:}\\
Simion Stoilow Institute of Mathematics of the Romanian Academy,\\
21 Calea Grivi\c{t}ei Street, 010702 Bucharest, Romania.\\
e-mail: cristian.anghel@imar.ro

\end{document}